\documentclass[11pt,reqno]{article}
\usepackage{amsmath,amssymb,amsthm}
\usepackage{color,url}
\newtheorem{theorem}{Theorem}[section]
\newtheorem{proposition}[theorem]{Proposition}
\newtheorem{corollary}[theorem]{Corollary}

\newtheorem{lemma}[theorem]{Lemma}
\newtheorem{example}[theorem]{Example}
\newtheorem{remark}[theorem]{Remark}
\newtheorem{problem}[theorem]{Problem}
\numberwithin{equation}{section}

\usepackage%[dvipdfmx]
{hyperref}

\def\bR{\mathbb{R}}
\def\bN{\mathbb{N}}
\def\bM{\mathbb{M}}
\def\bC{\mathbb{C}}
\def\cA{\mathcal{A}}
\def\cH{\mathcal{H}}
\def\cM{\mathcal{M}}
\def\cN{\mathcal{N}}

\def\Tr{\mathrm{Tr}}
\def\<{\langle}
\def\>{\rangle}

\def\eps{\varepsilon}

\def\diag{\mathrm{diag}}
\def\chao{\mathrm{chao}}

\def\Re{\mathrm{Re}\,}
\def\near{\mathrm{near}}

\begin{document}
\allowdisplaybreaks

\bigskip
\centerline{\LARGE Various inequalities between quasi-arithmetic mean}
\medskip
\centerline{\LARGE and quasi-geometric type means for matrices}

\bigskip
\bigskip
\centerline{\large
Fumio Hiai\footnote{{\it E-mail:} hiai.fumio@gmail.com
\ \ {\it ORCID:} 0000-0002-0026-8942}}

\medskip
\begin{center}
$^1$\,Graduate School of Information Sciences, Tohoku University, \\
Aoba-ku, Sendai 980-8579, Japan
\end{center}

\medskip

\begin{abstract}
In this paper, for $0<\alpha<1$, $p>0$ and positive semidefinite matrices $A,B\ge0$, we consider the
quasi-extension $\mathcal{A}_{\alpha,p}(A,B):=((1-\alpha)A^p+\alpha B^p)^{1/p}$ of the
$\alpha$-weighted arithmetic matrix mean, and the quasi-extensions
$\mathcal{M}_{\alpha,p}(A,B):=\mathcal{M}_\alpha(A^p,B^p)^{1/p}$ of several different $\alpha$-weighted
geometric-type matrix means $\mathcal{M}_\alpha(A,B)$ such as the $\alpha$-weighted geometric mean
in Kubo and Ando's sense and two types of $\alpha$-weighted version of Fiedler and Pt\'ak's spectral
geometric mean, as well as the R\'enyi mean and the $\alpha$-weighted Log-Euclidean mean. For these
we examine the inequalities $\mathcal{A}_{\alpha,p}(A,B)\triangleleft\mathcal{A}_{\alpha,q}(A,B)$ and
$\mathcal{M}_{\alpha,p}(A,B)\triangleleft\mathcal{A}_{\alpha,q}(A,B)$ of arithmetic-geometric type, where
$\triangleleft$ is one of several different matrix orderings varying from the strongest Loewner order to the
weakest order determined by trace inequality. For each choice of the above inequalities, our goal is to
hopefully obtain the necessary and sufficient condition on $p,q,\alpha$ under which the inequality holds
for all $A,B\ge0$.

\bigskip\noindent
{\it 2020 Mathematics Subject Classification:}
15A45, 47A64

\medskip\noindent
{\it Keywords and phrases:}
quasi matrix mean,
weighted arithmetic mean,
weighted harmonic mean,
weighted geometric mean,
spectral geometric mean,
R\'enyi mean,
Log-Euclidean mean,
Loewner order,
chaotic order,
near order,
entrywise eigenvalue order,
weak majorization,
Lie--Trotter--Kato product formula
\end{abstract}

{\small
\tableofcontents}

\section{Introduction}\label{Sec-1}

For positive semidefinite matrices (also operators) the most developed two-variable operator/matrix means
are Kubo and Ando's operator means \cite{KA}, defined corresponding to operator monotone functions
$f\ge0$ on $[0,\infty)$ with $f(1)=1$ as
$
A\sigma_fB:=A^{1/2}f(A^{-1/2}BA^{-1/2})A^{1/2}
$
for positive definite matrices (also operators) $A,B>0$ and extended to general positive semidefinite
$A,B\ge0$ as $A\sigma_fB:=\lim_{\eps\searrow0}(A+\eps I)\sigma_f(B+\eps I)$. Among them the most
typical ones are the $\alpha$-weighted arithmetic, harmonic and geometric means $\triangledown_\alpha$,
$!_\alpha$ and $\#_\alpha$ for $0\le\alpha\le1$, whose definitions for $A,B>0$ are
\begin{align*}
A\triangledown_\alpha B&:=(1-\alpha)A+\alpha B, \\
A!_\alpha B&:=((1-\alpha)A^{-1}+\alpha B^{-1})^{-1}, \\
A\#_\alpha B&:=A^{1/2}(A^{-1/2}BA^{-1/2})^\alpha A^{1/2}.
\end{align*}
The so-called arithmetic-geometric-harmonic mean inequality in the Loewner order is
\begin{align}\label{F-1.1}
A!_\alpha B\le A\#_\alpha B\le A\triangledown_\alpha B.
\end{align}

Geometric-type matrix means not in Kubo and Ando's sense have recently been in active consideration in
matrix analysis, partly motivated by recent development of quantum divergences in quantum information.
The most familiar one is the $\alpha$-weighted Log-Euclidean mean
\[
LE_\alpha(A,B):=\exp((1-\alpha)\log A+\alpha\log B)
\]
for $A,B>0$. The spectral geometric mean
\[
F(A,B):=(A^{-1}\#B)^{1/2}A(A^{-1}\#B)^{1/2}
\]
for $A,B>0$ was formerly introduced in \cite{FP} and extended in \cite{LL} to the $\alpha$-weighted version
as
\[
F_\alpha(A,B):=(A^{-1}\#B)^\alpha A(A^{-1}\#B)^\alpha,
\]
which has recently been reconsidered in \cite{GK,GT,GH} from a new perspective. Another new
$\alpha$-weighted version of the spectral geometric mean was also recently introduced in \cite{DTV} as
\[
\widetilde F_\alpha(A,B):=(A^{-1}\#_\alpha B)^{1/2}A^{2(1-\alpha)}(A^{-1}\#_\alpha B)^{1/2}
\]
Yet another geometric-type mean discussed in \cite{Hi2,DF} is
\[
R_{\alpha,z}(A,B):=\bigl(A^{1-\alpha\over2z}B^{\alpha\over z}A^{1-\alpha\over2z}\bigr)^z
\]
with two parameters $\alpha\in(0,1)$ and $z>0$, which is called the R\'enyi mean because of its close
relation with the $\alpha$-$z$-R\'enyi divergence \cite{AD}.

For any matrix mean $\cM(A,B)$ we have its quasi-extension $\cM(A^p,B^p)^{1/p}$ with a parameter $p>0$.
For $0<\alpha<1$ and $p>0$ the quasi-extension
$(A^p\triangledown_\alpha B^p)^{1/p}=((1-\alpha)A^p+\alpha B^p)^{1/p}$ of the $\alpha$-weighted arithmetic
mean has been discussed by many authors under the name ``operator power mean.'' We notice that
$LE_\alpha(A,B)$ is invariant under taking quasi-extension. In the present paper we denote the
quasi-extensions of $A\triangledown_\alpha B$, $A\#_\alpha B$, $A!_\alpha B$, $F(A,B)$ and
$\widetilde F(A,B)$ respectively by $\cA_{\alpha,p}(A,B)$, $G_{\alpha,p}(A,B)$, $\cH_{\alpha,p}(A,B)$,
$SG_{\alpha}(A,B)$ and $\widetilde SG_{\alpha,p}(A,B)$. We also write $R_{\alpha,p}(A,B)$ for
$R_{\alpha,z}(A,B)$ with parameter $p=1/z$ instead of $z$. See Section \ref{Sec-2.1} for more details on
these quasi-extended matrix means. The quasi matrix means $G_{\alpha,p}$, $SG_{\alpha,p}$,
$\widetilde SG_{\alpha,p}$, $R_{\alpha,p}$ as well as $LE_\alpha$ are referred to as ``quasi-geometric type
matrix means'' because all of them are reduced to $A^{1-\alpha}B^\alpha$ when $AB=BA$.

Our main aim of this paper is to extend the arithmetic-geometric mean inequality in \eqref{F-1.1} to those
for the quasi arithmetic mean and the quasi-geometric type means mentioned above. A special feature of
our study is that we consider not only the Loewner order ($X\le Y$ for $X,Y\ge0$) but also several of other
weaker orderings such as the chaotic order (denoted by $X\le_\chao Y$), the near order recently introduced
in \cite{DF1} (denoted by $X\le_\near Y$), the entrywise eigenvalue order (denoted by $X\le_\lambda Y$),
the weak majorization ($X\prec_wY$) and the trace inequality $\Tr\,X\le\Tr\,Y$ (denoted by $X\le_\Tr Y$).
See Section \ref{Sec-2.2} for the explicit definitions of these orderings. It may also be stressed that we deal
with quasi matrix means for general positive semidefinite matrices though restricted to positive definite
matrices in most references. Our goal is to hopefully obtain the necessary and sufficient condition on
$p,q,\alpha$ under which the inequality $\cM_{\alpha,p}(A,B)\triangleleft\cA_{\alpha,q}(A,B)$ holds for all
$A$ and $B$, when $\cM_{\alpha,p}$ is one of the quasi-geometric type matrix means and $\triangleleft$ is
one of the above matrix orderings.

The structure of the paper is as follows. In Section \ref{Sec-2.1}, for $0\le\alpha\le1$ and $p>0$ we review
the definitions of the above stated quasi matrix means $\cA_{\alpha,p}(A,B)$, $G_{\alpha,p}(A,B)$, etc.\ for
general positive semidefinite matrices $A,B\ge0$ (with the support condition $s(A)\ge s(B)$ for
$SG_{\alpha,p}$ and $\widetilde SG_{\alpha,p}$). Some general basic facts on these quasi matrix means
are summarized. In Section \ref{Sec-2.2} we review the above mentioned matrix orderings and summarize
some basic properties of them such as the strength relationship between them. In particular, for any quasi
matrix mean $\cM_{\alpha,p}$ and any matrix ordering $\triangleleft$ among stated above, we show
(Theorem \ref{T-2.3}) that if $\cM_{\alpha,p}(A,B)\triangleleft\cA_{\alpha,q}(A,B)$ holds for all $A,B>0$,
then the same holds for all $A,B\ge0$ (with $s(A)\ge s(B)$ for $\cM_{\alpha,p}=SG_{\alpha,p}$,
$\widetilde SG_{\alpha,p}$). Section~\ref{Sec-3} provides some technical computations for specified
$2\times2$ positive definite matrices as lemmas, which are repeatedly used in Section~\ref{Sec-4}.
The main Section~\ref{Sec-4} are divided into six subsections. When $\cM_{\alpha,p}$ is respectively
$\cA_{\alpha,p}$, $LE_\alpha$, $R_{\alpha,p}$, $G_{\alpha,p}$, $SG_{\alpha,p}$ and
$\widetilde SG_{\alpha,p}$, we examine the inequality $\cM_{\alpha,p}\triangleleft\cA_{\alpha,q}$ for any
choice $\triangleleft$ of the above matrix orderings. At the end of each subsection a table surveying the
results of the subsection is attached for the reader's convenience. Finally in Section~\ref{Sec-5} several
remarks are in order to supplement characteristics and motivation of our study, additional facts, open
questions, etc. The paper contains an appendix on the Lie--Trotter--Kato product formula for operator means,
which is used in Section~\ref{Sec-2}. This formula for positive semidefinite matrices is expected to be of
quite use, while it has been nowhere published in its complete form.

\section{Preliminaries}\label{Sec-2}

For each $n\in\bN$ we write $\bM_n$ for the $n\times n$ complex matrices. Let $\bM_n^+$ and
$\bM_n^{++}$ be the positive semidefinite $n\times n$ matrices and the positive definite $n\times n$
matrices, respectively. We simply write $A\ge0$ for $A\in\bM_n^+$ and $A>0$ for $A\in\bM_n^{++}$. The
$n\times n$ identity matrix is denoted by $I_n$ or simply $I$. Let $\Tr$ be the usual trace on $\bM_n$ and
$\|X\|_\infty$ be the operator norm of $X\in\bM_n$. We write ``for all $A,B\ge0$'' to mean
``for all $A,B\in\bM_n^+$ with any $n\in\bN$'', and ``for all $A,B>0$'' to mean ``for all $A,B\in\bM_n^{++}$
with any $n\in\bN$''. For $A\ge0$ there are two options of the convention of $A^0$; the one is $A^0:=I$
and the other is $A^0:=s(A)$, the support projection of $A$. In our discussions below we adopt the latter
convention. We write $A^{-1}$ for the generalized inverse of $A$, i.e., the inverse of $A$ under the
restriction to the support of $A$. Moreover, for $p<0$ we define $A^p:=(A^{-1})^{-p}$.

In this preliminary section we will explain several examples of quasi matrix means and different notions of
matrix orders, which provide the basis for our discussions in the main Section \ref{Sec-4}.

\subsection{Quasi matrix means}\label{Sec-2.1}

We first enumerate the definitions of quasi extensions of several binary matrix means for matrices, whose
order properties will be discussed in this paper. Let $0\le\alpha\le1$ and $p>0$. Let $A,B\in\bM_n^+$.

(i)\enspace
The \emph{$\alpha$-weighted arithmetic mean} is $A\triangledown_\alpha B:=(1-\alpha)A+\alpha B$. Define
the \emph{quasi $\alpha$-weighted arithmetic mean} by
\[
\cA_{\alpha,p}(A,B):=(A^p\triangledown_\alpha B^p)^{1/p}=((1-\alpha)A^p+\alpha B^p)^{1/p},
\]
which is also called the ($\alpha$-weighted) \emph{matrix power mean}.

(ii)\enspace
For $A,B>0$ the \emph{$\alpha$-weighted harmonic mean} is
$A!_\alpha B:=((1-\alpha)A^{-1}+\alpha B^{-1})^{-1}$, extended to general $A,B\ge0$ as
$A!_\alpha B:=\lim_{\eps\searrow0}(A+\eps I)!_\alpha(B+\eps I)$. Define the
\emph{quasi $\alpha$-weighted harmonic mean} by
\[
\cH_{\alpha,p}(A,B):=(A^p!_\alpha B^p)^{1/p}.
\]

(iii)\enspace
For $A,B>0$ the \emph{$\alpha$-geometric mean} is
$A\#_\alpha B:=A^{1/2}(A^{-1/2}BA^{-1/2})^\alpha A^{1/2}$, extended to $A,B\ge0$ as
$A\#_\alpha B:=\lim_{\eps\searrow0}(A+\eps I)\#_\alpha(B+\eps I)$. The \emph{geometric mean}
$\#$ ($=\#_{1/2}$) was first introduced by Pusz and Woronowicz \cite{PW}, and $\#_\alpha$, together with
the above $\triangledown_\alpha$ and $!_\alpha$, is a typical example of Kubo and Ando's operator means
\cite{KA}. Define the \emph{quasi $\alpha$-weighted geometric mean} by
\[
G_{\alpha,p}(A,B):=(A^p\#_\alpha B^p)^{1/p}.
\]

(iv)\enspace
For $A,B>0$ the \emph{spectral geometric mean} due to Fiedler and Pt\'ak \cite{FP} is
\[
F(A,B):=(A^{-1}\#B)^{1/2}A(A^{-1}\#B)^{1/2},
\]
which was extended to the $\alpha$-weighted version in \cite{LL} as
\[
F_\alpha(A,B):=(A^{-1}\#B)^\alpha A(A^{-1}\#B)^\alpha.
\]
For recent study of the weighted spectral
geometric mean, see, e.g., \cite{KL,GK,GT,GH} where $F_\alpha(A,B)$ is denoted by $A\natural_\alpha B$.
We define the \emph{quasi $\alpha$-weighted spectral geometric mean} by
\[
SG_{\alpha,p}(A,B):=F_\alpha(A^p,B^p)^{1/p}.
\]
Although the definitions of $F_\alpha$ and $SG_{\alpha,p}$ are available for any $A,B\ge0$ with $A^{-p}$
in the generalized sense, they are meaningful as far as $s(A)\ge s(B)$. Indeed, in this case we have
$SG_{\alpha,p}(A,B)=\lim_{\eps\searrow0}SG_{\alpha,p}(A+\eps I,B+\eps I)$ as verified in
Proposition \ref{P-2.2} below. When $s(A)\le s(B)$, the situation is similar by interchanging $A,B$ since
$SG_{\alpha,p}(A,B)=SG_{1-\alpha,p}(B,A)$ for $A,B>0$ (see Remark \ref{R-2.1}(2) below). In this paper
we will consider $SG_{\alpha,p}(A,B)$ for $A,B\ge0$ with $s(A)\ge s(B)$.

(v)\enspace
For $A,B>0$ another new weighted version of the spectral geometric mean was recently introduced in
\cite{DTV} as
\[
\widetilde F_\alpha(A,B):=(A^{-1}\#_\alpha B)^{1/2}A^{2(1-\alpha)}(A^{-1}\#_\alpha B)^{1/2}.
\]
We define the quasi version of this by
\[
\widetilde SG_{\alpha,p}(A,B):=\widetilde F_\alpha(A^p,B^p)^{1/p}.
\]
Note that $\widetilde F_{1/2}=F_{1/2}$ ($=F$) and so $\widetilde SG_{1/2,p}=SG_{1/2,p}$ for all $p>0$.
The definitions of $\widetilde F_\alpha$ and $\widetilde SG_{\alpha,p}$ are meaningful for any $A,B\ge0$
with $s(A)\ge s(B)$ similarly to $SG_{\alpha,p}$ in (iv); see Proposition \ref{P-2.2} below. We will use
$\widetilde SG_{\alpha,p}$ as well as $SG_{\alpha,p}$ in this situation.

(vi)\enspace
We consider one more quasi matrix mean defined for all $A,B\ge0$ by
\[
R_{\alpha,p}(A,B):=\bigl(A^{{1-\alpha\over2}p}B^{\alpha p}A^{{1-\alpha\over2}p}\bigr)^{1/p},
\]
which is called the \emph{R\'enyi mean} in \cite{DF}. Note that the trace function $\Tr\,R_{\alpha,1/z}(B,A)$
(for general $\alpha,z>0$) appears as the main component in the definition of the
\emph{$\alpha$-$z$-R\'enyi divergence} (see \cite{AD,Hi2}) in quantum information, that is the reason for
the terminology.

(vii)\enspace
For $A,B>0$ the \emph{Log-Euclidean mean} is
\[
LE_\alpha(A,B):=\exp((1-\alpha)\log A+\alpha\log B).
\]
Restricting to $0<\alpha<1$ we extend this to general $A,B\ge0$ as
\[
LE_\alpha(A,B):=P_0\exp\{(1-\alpha)P_0(\log A)P_0+\alpha P_0(\log B)P_0\},
\]
where $P_0:=s(A)\wedge s(B)$. This extended definition is justified by Proposition \ref{P-2.2} and
Theorem \ref{T-A.1} in Appendix~\ref{Sec-A}. Note that $LE_\alpha(A^p,B^p)^{1/p}=LE_\alpha(A,B)$ holds
for all $p>0$, so we have no quasi extension for $LE_\alpha$.

In the next remark we collect a few general simple facts on the quasi matrix means defined in (i)--(vii)
above.

\begin{remark}\label{R-2.1}\rm
(1)\enspace
For $\cM_{\alpha,p}\in\{\cA_{\alpha,p},\cH_{\alpha,p},G_{\alpha,p}\}$, it is obvious by definition that
$\cM_{0,p}(A,B)=A$ and $\cM_{1,p}(A,B)=B$ for all $A,B\ge0$. From our definitions in (iv) and (v) it is easy
to verify that for any $A,B\ge0$ with $s(A)\ge s(B)$ and any $p>0$,
\begin{align*}
&SG_{0,p}(A,B)=(s(B)A^ps(B))^{1/p},\qquad SG_{1,p}(A,B)=B, \\
&\widetilde SG_{0,p}(A,B)=A,\qquad\widetilde SG_{1,p}(A,B)=(B^{p/2}s(A)B^{p/2})^{1/p}.
\end{align*}
Also, note that $R_{0,1}(A,B)=(A^{p/2}s(B)A^{p/2})^{1/p}$ and $R_{1,p}(A,B)=(s(A)B^ps(A))^{1/p}$ for all
$A,B\ge0$ and $p>0$. Hence the cases $\alpha=0,1$ will be trivial (or quite simple) for our purpose, so
we will concentrate our considerations to $0<\alpha<1$ below.

(2)\enspace
Let $0<\alpha<1$ and $p>0$. It is clear that any $\cM_{\alpha,p}$ from $\cA_{\alpha,p}$, $\cH_{\alpha,p}$,
$G_{\alpha,p}$, $LE_\alpha$ is symmetric in the sense that $\cM_{\alpha,p}(A,B)=\cM_{1-\alpha,p}(B,A)$
for all $A,B\ge0$. This is the case also for $SG_{\alpha,p}$ when restricted to $A,B>0$; see \cite{LL}.
However, this is not the case for $\widetilde SG_{\alpha,p}$.

(3)\enspace
It is obvious that $\cA_{\alpha,p}$ and $\cH_{\alpha,p}$ are transformed each other by taking inverse,
that is, $\cA_{\alpha,p}(A^{-1},B^{-1})=\cH_{\alpha,p}(A,B)^{-1}$ for all $A,B>0$ and $p>0$. It is also
immediate to see that any $\cM_{\alpha,p}$ of the quasi matrix means given in (iii)--(vii) is invariant under
inverse, i.e., for any $A,B>0$,
\[
\cM_{\alpha,p}(A^{-1},B^{-1})=\cM_{\alpha,p}(A,B)^{-1}.
\]

(4)\enspace
We note that among quasi matrix means in (i)--(vii), only $\cA_{\alpha,1}$, $\cH_{\alpha,1}$ and
$G_{\alpha,1}$ for $0\le\alpha\le1$ are Kubo and Ando's operator means. Indeed, although the function
$\cA_{\alpha,p}(1,x)=(1-\alpha+\alpha x^p)^{1/p}$ is operator monotone on $(0,\infty)$ for $0\le\alpha\le1$
and $0<p\le1$, the corresponding operator mean is
\[
(A,B)\mapsto A^{1/2}\bigl\{(1-\alpha)I+\alpha(A^{-1/2}BA^{-1/2})^p\bigr\}^{1/p}A^{1/2},
\]
which is obviously different from $\cA_{\alpha,p}(A,B)$ except for $\alpha=0,1$ or $p=1$. The situation for
$\cH_{\alpha,p}$ is similar. For each $\cM$ from $G_{\alpha,p}$, $SG_{\alpha,p}$,
$\widetilde SG_{\alpha,p}$, $R_{\alpha,p},LE_\alpha$, the function $\cM(1,x)=x^\alpha$ is operator
monotone on $(0,\infty)$ for $0\le\alpha\le1$ but the corresponding operator mean is $\#_\alpha$, i.e.,
$G_{\alpha,1}$.
\end{remark}

\begin{proposition}\label{P-2.2}
Let $0<\alpha<1$, $p>0$, and $\cM$ be any of the quasi matrix means given in (i)--(vii). Let $A,B\ge0$,
with an additional assumption $s(A)\ge s(B)$ when $\cM=SG_{\alpha,p}$ or $\widetilde SG_{\alpha,p}$.
Then we have
\[
\cM(A,B)=\lim_{\eps\searrow0}\cM(A+\eps I,B+\eps I).
\]
\end{proposition}

\begin{proof}
When $\cM=\cA_{\alpha,p}$ or $\cH_{\alpha,p}$ or $G_{\alpha,p}$, it is easy to verify the result from the
downward continuity of Kubo and Ando's operator means since $(A+\eps I)^p\searrow A^p$ and
$(B+\eps I)^p\searrow B^p$ as $\eps\searrow0$. For $\cM=R_{\alpha,p}$ the result is obvious, and for
$\cM=LE_\alpha$ it was verified in \cite[Sec.~4]{HP}; see also Remark \ref{R-A.2} in Appendix~\ref{Sec-A}.
For the remaining, we first consider the case $A>0$. Since
\begin{align*}
(A+\eps I)^{-p}\#(B+\eps I)^p
&=(A+\eps I)^{-p/2}\{(A+\eps I)^{p/2}(B+\eps I)^p(A+\eps I)^{p/2}\}^{1/2}(A+\eps I)^{-p/2} \\
&\to A^{-p/2}(A^{p/2}B^pA^{p/2})^{1/2}A^{-p/2}=A^{-p}\#B^p\quad\mbox{as $\eps\searrow0$},
\end{align*} 
we see that $SG_{\alpha,p}(A+\eps I,B+\eps I)\to SG_{\alpha,p}(A,B)$ as $\eps\searrow0$. Similarly we
have $\widetilde SG_{\alpha,p}(A+\eps I,B+\eps I)\to\widetilde SG_{\alpha,p}(A,B)$ too. For general
$A,B\ge0$ with $s(A)\ge s(B)$ take the decomposition $\bC^n=\cH_1\oplus\cH_2$ where $\cH_1$ is the
range of $s(A)$ and $\cH_2:=\cH_1^\perp$. For $\cM=SG_{\alpha,p}$ or $\widetilde SG_{\alpha,p}$
we can write
\begin{align*}
\cM(A,B)&=\cM(A_1,B_1)\oplus\cM(0I_2,0I_2), \\
\cM(A+\eps I,B+\eps I)&=\cM(A_1+\eps I_1,B_1+\eps I_1)\oplus\cM(\eps I_2,\eps I_2)
=\cM(A_1+\eps I_1,B_1+\eps I_1)\oplus\eps I_2,
\end{align*}
where $I_k:=I_{\cH_k}$, $k=1,2$. From the above shown case it follows that
$\cM(A_1+\eps I_1,B_1+\eps I_1)\to\cM(A_1,B_1)$ as $\eps\searrow0$. Hence the assertion follows.
\end{proof}

The next theorem says that the quasi matrix means in (i)--(vi) satisfy the Lie--Trotter--Kato product formula.
So we may consider the Log-Euclidean mean $LE_\alpha$ as a sort of attractor for those quasi matrix
means. Note that when $A,B>0$, the proof is much simpler without use of Theorem \ref{T-A.1} and
Remark \ref{R-A.4}.

\begin{theorem}\label{T-2.3}
Let $0<\alpha<1$, $p>0$, and $\cM_{\alpha,p}$ be any of the quasi matrix means given in (i)--(vi). Let
$A,B\ge0$, with an additional assumption $s(A)\ge s(B)$ when $\cM=SG_{\alpha,p}$ or
$\widetilde SG_{\alpha,p}$. Then we have
\[
LE_\alpha(A,B)=\lim_{p\searrow0}\cM_{\alpha,p}(A,B).
\]
\end{theorem}

\begin{proof}
The convergences for $\cA_{\alpha,p}$, $\cH_{\alpha,p}$ and $G_{\alpha,p}$ are special cases of
Theorem \ref{T-A.1}, and that for $R_{\alpha,p}$ follows from \eqref{F-A.2} in Appendix~\ref{Sec-A}. For
the remaining, assume that $s(A)\ge s(B)$. Let $P_0:=s(A)\wedge s(B)=s(B)$ and $\cH_0$ be the range
of $P_0$. Note that $SG_{\alpha,p}(A,B)=\{(A^{-p}\#B^p)^\alpha A^p(A^{-p}\#B^p)\}^{1/p}$ is supported on
$\cH_0$. Set $L:={1\over2}\{-P_0(\log A)P_0+P_0(\log B)P_0\}$, where $\log A$ is defined in the
generalized sense as $\log A:=s(A)(\log A)$ and similarly for $\log B$. Then by \eqref{F-A.25} in
Remark \ref{R-A.4} one can write $A^{-p}\#B^p=P_0+pL+o(p)$ as $p\searrow0$, which implies that
$(A^{-p}\#B^p)^\alpha=P_0+p\alpha L+o(p)$. Therefore, one has
\begin{align*}
SG_{\alpha,p}(A,B)\big|_{\cH_0}
&=\bigl\{(P_0+p\alpha L+o(p))(I+p\log A+o(p))(P_0+p\alpha L+o(p))\bigr\}\big|_{\cH_0} \\
&=\bigl\{P_0+p(2\alpha L+P_0(\log A)P_0+o(p)\bigr\}\big|_{\cH_0} \\
&=\bigl\{P_0+p\bigl((1-\alpha)P_0(\log A)P_0+\alpha P_0(\log B)P_0\bigr)+o(p)\bigr\}\big|_{\cH_0}.
\end{align*}
Applying this to the Taylor expansion of $\log(1+x)$ gives
\[
{1\over p}\log SG_{\alpha,p}(A^p,B^p)\big|_{\cH_0}
=\bigl\{(1-\alpha)P_0(\log A)P_0+\alpha P_0(\log B)P_0+o(1)\bigr\}\big|_{\cH_0},
\]
which yields the assertion for $SG_{\alpha,p}$. The proof for $\widetilde SG_{\alpha,p}$ is similar, which is
omitted and left to the reader.
\end{proof}

\subsection{Matrix order relations}\label{Sec-2.2}

Here we recall different types of order relations between matrices in $\bM_n^+$. In Section \ref{Sec-4} we
will examine several of these orderings between quasi matrix means introduced in Section \ref{Sec-2.1}.
Let $X,Y\in\bM_n^+$.

(a)\enspace
We write $X\le Y$ as usual to denote the \emph{Loewner order}, i.e., the positive semidefiniteness order in
the sense that $Y-X$ is positive semidefinite.

(b)\enspace
We write $X\le_\chao Y$ and call it the \emph{chaotic order} if $s(X)\le s(Y)$ and
$s(X)(\log X)s(X)\le s(X)(\log Y)s(X)$. When $X,Y>0$, this simply reduces to $\log X\le\log Y$.

(c)\enspace
We write $X\le_\near Y$ and call it the \emph{near order} if $s(X)\le s(Y)$ and $X\#Y^{-1}\le I$,
where $Y^{-1}$ is the generalized inverse of $Y$. This ordering was introduced in \cite{DF1} and further
discussed in \cite{GH} for $X,Y>0$. Note that $\le_\near$ is not transitive (though `near' transitive); see
\cite[Theorem 2]{DF1}.

(d)\enspace
Let $\lambda(X)=(\lambda_1(X),\dots,\lambda_n(X))$ denote the eigenvalues of $X$ in decreasing
order with multiplicities. We write $X\le_\lambda Y$ to denote the \emph{entrywise eigenvalue order}, i.e.,
$\lambda_i(X)\le\lambda_i(Y)$ for each $i=1,\dots,n$.

(e)\enspace
The \emph{weak (sub)majorization} $X\prec_wY$ means that
\[
\sum_{i=1}^k\lambda_i(X)\le\sum_{i=1}^k\lambda_i(Y),\qquad1\le k\le n.
\]
This is equivalent to that $\|X\|\le\|Y\|$ for any unitarily invariant norm $\|\cdot\|$; see, e.g.,
\cite[Proposition 4.4.13]{Hi}. More details on majorization theory are found in \cite{An2}, \cite[Chap.~II]{Bh}
and \cite{MOA}.

(f)\enspace
The \emph{weak log-majorization} $X\prec_{w\log}Y$ means that
\[
\prod_{i=1}^k\lambda_i(X)\le\prod_{i=1}^k\lambda_i(Y),\qquad1\le k\le n.
\]
The \emph{log-majorization} $X\prec_{\log}Y$ means that $X\prec_{w\log}Y$ and
$\prod_{i=1}^n\lambda_i(X)=\prod_{i=1}^n\lambda_i(Y)$, i.e., $\det X=\det Y$.

(g)\enspace
We write $X\le_\Tr Y$ if the trace inequality $\Tr\,X\le\Tr\,Y$ holds.

In the next proposition we summarize the strength relationship between the matrix orderings defined in
(a)--(g) above, together with a few basic properties.
 
\begin{proposition}\label{P-2.4}
Let $X,Y\ge0$.
\begin{itemize}
\item[(1)] $X\le Y$ holds if and only $\log(tI+X)\le\log(tI+Y)$ for all $t>0$.
\item[(2)] $X\le_\chao Y$ holds if and only if $s(X)\le s(Y)$ and $X^p\#Y^{-p}\le I$ for all $p>0$ (equivalently,
for all $p\in(0,\delta)$ for some $\delta>0$), with $Y^{-p}$ in the generalized sense.
\item[(3)] We have
\begin{align*}
X\le Y&\implies X\le_\chao Y\implies X\le_\near Y\implies X\le_\lambda Y \\
&\implies X\prec_{w\log}Y\implies X\prec_wY\implies X\le_\Tr Y.
\end{align*}
\end{itemize}
\end{proposition}

\begin{proof}
(1)\enspace
The `only if' is obvious since $\log x$ ($x>0$) is operator monotone. Conversely, assume that
$\log(tI+X)\le\log(tI+Y)$ for all $t>0$. Since
\[
\eps X+o(\eps)=\log(I+\eps X)\le\log(I+\eps Y)=\eps Y+o(\eps)\quad\mbox{as $\eps\searrow0$},
\]
we have $X\le Y$.

(2)\enspace
This equivalence was shown in \cite[Theorem 1]{An1} when $X,Y>0$. Assuming that $P:=s(X)\le s(Y)$,
we may show that $P(\log X)P\le P(\log Y)P$ if and only if $X^p\#Y^{-p}\le I$ for all $p>0$ (equivalently, for
all $p\in(0,\delta)$). Recall \cite[Theorem 2.1]{AH} that
\begin{align}\label{F-2.1}
(X^p\#Y^{-p})^{1/p}\prec_{\log}(X^q\#Y^{-q})^{1/q},\qquad0<q\le p.
\end{align}
Moreover, by Theorem \ref{T-A.1},
\begin{align}\label{F-2.2}
\lim_{q\searrow0}(X^q\#Y^{-q})^{2/q}=P\exp\{P(\log X)P-P(\log Y)P\}.
\end{align}
If $X^p\#Y^{-p}\le I$ for all $p\in(0,\delta)$, then \eqref{F-2.2} gives $P\exp\{P(\log X)P-P(\log Y)P\}\le I$ so
that $P(\log X)P-P(\log Y)P\le0$. Conversely, if $P(\log X)P\le P(\log Y)P$, then it follows from \eqref{F-2.1}
and \eqref{F-2.2} that
\[
(X^p\#Y^{-p})^{2/p}\prec_{\log}P\exp\{P(\log X)P-P(\log Y)P\}\le P\le I
\]
so that $X^p\#Y^{-p}\le I$ for all $p>0$.

(3)\enspace
Assume that $X\le Y$. Then $s(X)\le s(Y)$ and $X^p\#Y^{-p}\le Y^p\#Y^{-p}=s(Y)\le I$ for all $p\in(0,1)$.
Hence the first implication follows from (2). The second implication is seen from (2) as well. The third was
proved in \cite[Theorem 2.4]{GH}, whose proof is valid in the present setting. The remaining implications are
obvious or well known; see, e.g., \cite[Proposition 4.1.6]{Hi} for the penultimate implication.
\end{proof}

\begin{lemma}\label{L-2.5}
Let $X,Y,X_n,Y_n,\ge0$ for $n\in\bN$ be given such that $X_n\to X$ and $Y_n\to Y$. For each
$\triangleleft\in\{\le,\le_\lambda,\prec_w,\prec_{w\log},\le_\Tr\}$, if $X_n\triangleleft Y_n$ for all $n\in\bN$,
then $X\triangleleft Y$. This holds for $\triangleleft=\,\le_\chao$, $\le_\near$ as well under
the additional assumption $s(X)\le s(Y)$.
\end{lemma}

\begin{proof}
The first statement is immediately seen since $\lambda(X_n)\to\lambda(X)$ and
$\lambda(Y_n)\to\lambda(Y)$. For the latter, assume that $s(X)\le s(Y)$ and $X_n\le_\chao Y_n$ for all $n$.
By Proposition \ref{P-2.4}(2) one has $s(X_n)\le s(Y_n)$ and $X_n^p\#Y_n^{-p}\le I$, $p>0$. For any $p>0$,
since
\[
(Y_n^{p/2}X_n^pY_n^{p/2})^{1/2}=Y_n^{p/2}(X_n^p\#Y_n^{-p})Y_n^{p/2}\le Y_n^p,
\]
letting $n\to\infty$ gives $(Y^{p/2}X^pY^{p/2})^{1/2}\le Y^p$ for all $p>0$. Thanks to $s(X)\le s(Y)$, this in turn
implies that $X^p\#Y^{-p}\le s(Y)\le I$. Hence $X\le_\chao Y$ follows by Proposition \ref{P-2.4}(2) again. The
proof for $\le_\near$ is similar.
\end{proof}

The next theorem is of quite importance from the viewpoint of the scope of our study.

\begin{theorem}\label{T-2.6}
Let $0<\alpha<1$ and $p,q>0$. Let $\cM_{\alpha,p}$ be any of the quasi matrix means in (i)--(vii) of
Section \ref{Sec-2.1}, and $\triangleleft$ be any of $\le$, $\le_\chao$, $\le_\near$, $\le_\lambda$, $\prec_w$,
$\prec_{w\log}$, $\le_\Tr$. If $\cM_{\alpha,p}(A,B)\triangleleft\cA_{\alpha,q}(A,B)$ (resp.,
$\cH_{\alpha,q}(A,B)\triangleleft\cM_{\alpha,p}(A,B)$) holds for all $A,B>0$, then the same holds for all
$A,B\ge0$, where $s(A)\ge s(B)$ is assumed for $\cM_{\alpha,p}=SG_{\alpha,p}$, $\widetilde SG_{\alpha,p}$.
\end{theorem}

\begin{proof}
Assume that $\cM_{\alpha,p}(A,B)\triangleleft\cA_{\alpha,q}(A,B)$ for all $A,B>0$. Let $A,B\ge0$ be arbitrary.
The assumption implies that $\cM_{\alpha,p}(A+\eps I,B+\eps I)\triangleleft\cA_{\alpha,q}(A+\eps I,B+\eps I)$.
Now $s(A)\ge s(B)$ is assumed when $\cM_{\alpha,p}=SG_{\alpha,p}$ or $\widetilde SG_{\alpha,p}$. Then
by Proposition \ref{P-2.2} we have $\cM_{\alpha,p}(A+\eps I,B+\eps I)\to\cM_{\alpha,p}(A,B)$ as well as
$\cA_{\alpha,q}(A+\eps I,B+\eps I)\to\cA_{\alpha,q}(A,B)$ as $\eps\searrow0$. Moreover, it is clear that
$s(\cM_{\alpha,p}(A,B))\le s(A)\vee s(B)=s(\cA_{\alpha,q}(A,B))$. Hence we have
$\cM_{\alpha,p}(A,B)\triangleleft\cA_{\alpha,q}(A,B)$ by Lemma \ref{L-2.5}. The proof is similar for
$\cH_{\alpha,q}(A,B)\triangleleft\cM_{\alpha,p}(A,B)$ in view of
$s(\cM_{\alpha,p}(A,B))\ge s(A)\wedge s(B)=s(\cH_{\alpha,q}(A,B))$.
\end{proof}

Concerning the quasi matrix means and matrix orderings mentioned above, some general facts relevant to
our study are in order.

\begin{remark}\label{R-2.7}\rm
(1)\enspace
Let $0<\alpha,\beta<1$ with $\alpha\ne\beta$, and $p,q>0$ be arbitrary. Let
\begin{align*}
\cM_\alpha&\in\{\cA_{\alpha,p},\cH_{\alpha,p},G_{\alpha,p},SG_{\alpha,p},
\widetilde SG_{\alpha,p},R_{\alpha,p},LE_\alpha\}, \\
\cN_\beta&\in\{\cA_{\beta,q},\cH_{\beta,q},G_{\beta,q},SG_{\beta,q},
\widetilde SG_{\beta,q},R_{\beta,q},LE_\beta\}.
\end{align*}
Let $f_\alpha(x):=\cM_\alpha(1,x)$ and $g_\beta(x):=\cM_\beta(1,x)$ for scalars $A=x>0$ and $B =1$.
Since $f_\alpha(1)=g_\beta(1)=1$, $f'_\alpha(1)=\alpha$ and $g'_\beta(1)=\beta$, one has
$f_\alpha(1)-g_\beta(1)=0$ and $f'_\alpha(1)-g'_\beta(1)\ne0$. Therefore, the sign of
$f_\alpha(x)-g_\beta(x)$ is different between the left and the right of $x=1$. This implies that
$\cM_\alpha(A,B)$ and $\cN_\beta(A,B)$ are not definitively comparable with respect to any ordering. So
we may only discuss inequalities between two of the above quasi matrix means under the same $\alpha$.

(2)\enspace
For any $\alpha\in(0,1)$ it is clear that
$G_{\alpha,p}(a,b)=SG_{\alpha,p}(a,b)=\widetilde SG(a,b)=R_{\alpha,p}(a,b)=LE_\alpha(a,b)=
a^{1-\alpha}b^\alpha$ for scalars $a,b>0$, independently of $p>0$. For this reason we refer to
$G_{\alpha,p},SG_{\alpha,p},\widetilde SG,R_{\alpha,p},LE_\alpha$ as \emph{quasi-geometric type}
matrix means. Also it is immediate to see that
\[
\cH_{\alpha,q}(a,b)=((1-\alpha)a^{-q}+\alpha b^{-q})^{-1/q}\le a^{1-\alpha}b^\alpha
\le\cA_{\alpha,p}(a,b)=((1-\alpha)a^q+\alpha b^q)^{1/q}
\]
for scalars $a,b>0$, and moreover $\cH_{\alpha,q}(a,b)\nearrow a^{1-\alpha}b^\alpha$ and
$\cA_{\alpha,p}(a,b)\searrow a^{1-\alpha}b^\alpha$ as $p\searrow0$. Therefore, for each quasi-geometric
type mean $\cM_{\alpha,p}$ and for each ordering $\triangleleft$ mentioned above, we may consider
possible inequalities between $\cM_{\alpha,p},\cH_{\alpha,q}$ and between
$\cM_{\alpha,p},\cA_{\alpha,q}$ in the respective directions $\cH_{\alpha,q}\triangleleft\cM_{\alpha,p}$ and
$\cM_{\alpha,p}\triangleleft\cA_{\alpha,q}$ only.

(3)\enspace
Let $0<\alpha<1$, $p>0$ and $\cM_{\alpha,p}$ be any quasi-geometric type mean as in Theorem \ref{T-2.6}.
For any $X,Y>0$ the following are straightforward by definitions:
\begin{align*}
X\le Y&\iff Y^{-1}\le X^{-1},\\
X\le_\chao Y&\iff Y^{-1}\le_\chao X^{-1}, \\
X\le_\near Y&\iff Y^{-1}\le_\near X^{-1},\\
X\le_\lambda Y&\iff Y^{-1}\le_\lambda X^{-1}.
\end{align*}
From these, Remark \ref{R-2.1}(3) and Theorem \ref{T-2.6} together we see that for any
$\triangleleft\in\{\le,\le_\chao,\le_\near,\le_\lambda\}$, $\cH_{\alpha,q}(A,B)\triangleleft\cM_{\alpha,p}(A,B)$
holds for all $A,B>0$ if and only if $\cM_{\alpha,p}(A,B)\triangleleft\cA_{\alpha,q}(A,B)$ holds for all
$A,B>0$. So the characterizations of inequalities between $\cH_{\alpha,q},\cM_{\alpha,p}$ are to large
extent reduced to those between $\cM_{\alpha,p},\cA_{\alpha,q}$.
\end{remark}

\section{Some technical computations for $2\times2$ matrices}\label{Sec-3}

In this section we perform some technical computations for specified $2\times2$ matrices, which will be of
quite use in the next main section. Below we will often treat $2\times2$ matrices with parameter $\theta\in\bR$,
given in the approximate form up to $o(\theta^2)$ as
\begin{align}\label{F-3.1}
X_\theta:=\begin{bmatrix}a+\theta^2x_{11}&\theta x_{12}\\\theta x_{12}&b+\theta^2x_{22}\end{bmatrix}
+o(\theta^2)\quad\mbox{as $\theta\to0$},
\end{align}
where $a,b,x_{11},x_{22},x_{12}\in\bR$, and little-o notation $o(\theta^2)$ means that
${o(\theta^2)\over\theta^2}\to0$ as $\theta\to0$. In the first lemma we explain the way of computing the
functional calculus of $X_\theta$ based on Taylor's theorem and Daleckii and Krein's derivative formulas
(see, e.g., \cite[Sec.~2.3]{Hi}).

\begin{lemma}\label{L-3.1}
Let $X_\theta$ be given in \eqref{F-3.1} with $a\ne b$ and  $f$ be a $C^2$-function on an interval containing
$a,b$. Then the functional calculus $f(X_\theta)$ is given in the approximate form as
\begin{align}\label{F-3.2}
f(X_\theta)=\begin{bmatrix}f(a)+\theta^2y_{11}&\theta y_{12}\\
\theta y_{12}&f(b)+\theta^2y_{22}\end{bmatrix}+o(\theta^2)\quad\mbox{as $\theta\to0$},
\end{align}
where
\begin{align}\label{F-3.3}
\begin{cases}
y_{11}:=f^{[1]}(a,a)x_{11}+f^{[2]}(a,a,b)x_{12}^2=f'(a)x_{11}+{f'(a)(a-b)-f(a)+f(b)\over(a-b)^2}\,x_{12}^2, \\
y_{22}:=f^{[1]}(b,b)x_{22}+f^{[2]}(a,b,b)x_{12}^2=f'(b)x_{22}+{f(a)-f(b)-f'(b)(a-b)\over(a-b)^2}\,x_{12}^2, \\
y_{12}:= f^{[1]}(a,b)x_{12}={f(a)-f(b)\over a-b}\,x_{12}.\end{cases}
\end{align}
In the above, $f^{[1]}(a_1,a_2)$ is the first divided difference of $f$ and $f^{[2]}(a_1,a_2,a_3)$ is the second one
(see \cite[Sec.~2.2]{Hi}).
\end{lemma}

\begin{proof}
First, note that since $X_\theta\to X_0=\begin{bmatrix}a&0\\0&b\end{bmatrix}$ as $\theta\to0$, $f(X_\theta)$
is well defined for any $\theta$ near $0$. We write $X_\theta=X_0+\theta H+\theta^2 K+o(\theta^2)$ as
$\theta\to0$ where $H:=\begin{bmatrix}0&x_{12}\\x_{12}&0\end{bmatrix}$ and
$K:=\begin{bmatrix}x_{11}&0\\0&x_{22}\end{bmatrix}$. By Taylor's theorem for the functional calculus $f(X)$
at $X_0$ (see \cite[Theorem 2.3.1]{Hi}) one has
\[
f(X_\theta)=f(X_0)+Df(X_0)(\theta H+\theta^2 K)+{1\over2}D^2f(\theta H,\theta H)+o(\theta^2)
\quad\mbox{as $\theta\to0$},
\]
where $Df(X_0)(\cdot)$ and $D^2f(X_0)(\cdot,\cdot)$ are the first and the second Fr\'echet derivatives of $f(X)$
at $X_0$. Thanks to Daleckii and Krein's derivative formulas (see \cite[p.~163]{Hi}) one can write, with the
Schur product $\circ$,
\begin{align*}
Df(X_0)(\theta H+\theta^2K)
&=\begin{bmatrix}f^{[1]}(a,a)&f^{[1]}(a,b)\\f^{[1]}(a,b)&f^{[1]}(b,b)\end{bmatrix}
\circ(\theta H+\theta^2 K) \\
&=\theta\begin{bmatrix}0&{f(a)-f(b)\over a-b}\,x_{12}\\{f(a)-f(b)\over a-b}\,x_{12}&0\end{bmatrix}
+\theta^2\begin{bmatrix}f'(a)x_{11}&0\\0&f'(b)x_{22}\end{bmatrix}
\end{align*}
and
\begin{align*}
{1\over2}D^2f(X_0)(\theta H,\theta H)
&=\theta^2\begin{bmatrix}f^{[2]}(a,a,b)x_{12}^2&0\\0&f^{[2]}(a,b,b)x_{12}^2\end{bmatrix} \\
&=\theta^2\begin{bmatrix}{f'(a)(a-b)-f(a)+f(b)\over(a-b)^2}\,x_{12}^2&0\\
0&{f(a)-f(b)-f'(b)(a-b)\over(a-b)^2}\,x_{12}^2\end{bmatrix}.
\end{align*}
Hence expression \eqref{F-3.2} holds with \eqref{F-3.3}.
\end{proof}

Under the weaker assumption that $f$ is a $C^1$-function, the reduced version of \eqref{F-3.2} up to
$o(\theta)$ (meaning ${o(\theta)\over\theta}\to0$ as $\theta\to0$) is given as
\begin{align}\label{F-3.4}
f(X_\theta)=\begin{bmatrix}f(a)&\theta y_{12}\\\theta y_{12}&f(b)\end{bmatrix}+o(\theta)
\quad\mbox{as $\theta\to0$}
\end{align}
with $y_{12}$ in \eqref{F-3.3}.

\begin{example}\label{E-3.2}\rm
The following special cases of Lemma \ref{L-3.1} will be used below.
\begin{itemize}
\item[(1)] Let $a=1\ne b>0$ in \eqref{F-3.1} and $f(x)=x^r$ on $(0,\infty)$ with $r>0$. Then
\[
X_\theta^r=\begin{bmatrix}1+\theta^2y_{11}&\theta y_{12}\\\theta y_{12}&b^r+\theta^2y_{22}
\end{bmatrix}+o(\theta^2)\quad\mbox{as $\theta\to0$},
\]
where
\[
\begin{cases}y_{11}:=rx_{11}+{r-1-rb+b^r\over(1-b)^2}\,x_{12}^2, \\
y_{22}:=rb^{r-1}x_{22}+{1-rb^{r-1}+(r-1)b^r\over(1-b)^2}\,x_{12}^2, \\
y_{12}:={1-b^r\over1-b}\,x_{12}.
\end{cases}
\]

\item[(2)] Let $a=0\ne b\in\bR$ in \eqref{F-3.1} and $f(x)=e^x$ on $\bR$. Then
\[
e^{X_\theta}=\begin{bmatrix}1+\theta^2y_{11}&\theta y_{12}\\\theta y_{12}&e^b+\theta^2y_{22}
\end{bmatrix}+o(\theta^2)\quad\mbox{as $\theta\to0$},
\]
where
\[
\begin{cases}y_{11}:=x_{11}-\bigl({1\over b}+{1-e^b\over b^2}\bigr)x_{12}^2, \\
y_{22}:=e^bx_{22}+\bigl({e^b\over b}+{1-e^b\over b^2}\bigr)x_{12}^2, \\
y_{12}:=-{1-e^b\over b}\,x_{12}.
\end{cases}
\]

\item[(3)] Let $a=1$, $b>0$ in \eqref{F-3.1} and $f(x)=\log x$ on $(0,\infty)$. Then
\[
\log X_\theta=\begin{bmatrix}\theta^2y_{11}&\theta y_{12}\\\theta y_{12}&\log b+\theta^2y_{22}
\end{bmatrix}+o(\theta^2)\quad\mbox{as $\theta\to0$},
\]
where
\[
\begin{cases}y_{11}:=x_{11}+{1-b+\log b\over(1-b)^2}\,x_{12}^2, \\
y_{22}:={bx_{22}\over b}-{1-b+b\log b\over b(1-b)^2}\,x_{12}^2, \\
y_{12}:={-\log b\over1-b}\,x_{12}.
\end{cases}
\]
\end{itemize}
\end{example}

\medskip
We will repeatedly utilize the $2\times2$ positive definite matrices given as follows:
\begin{align}\label{F-3.5}
A_0:=\begin{bmatrix}1&0\\0&x\end{bmatrix},\qquad
B_\theta:=\begin{bmatrix}\cos\theta&-\sin\theta\\\sin\theta&\cos\theta\end{bmatrix}
\begin{bmatrix}1&0\\0&y\end{bmatrix}
\begin{bmatrix}\cos\theta&\sin\theta\\-\sin\theta&\cos\theta\end{bmatrix}
\end{align}
for $x,y>0$ and $\theta\in\bR$. As immediately verified, the approximate form of $B_\theta$ up to $o(\theta^2)$
is
\begin{align}\label{F-3.6}
B_\theta=\begin{bmatrix}1-\theta^2(1-y)&\theta(1-y)\\\theta(1-y)&y+\theta^2(1-y)\end{bmatrix}
+o(\theta^2)\quad\mbox{as $\theta\to0$}.
\end{align}

\begin{lemma}\label{L-3.3}
Let $0<\alpha<1$, $p,q>0$, and $x,y>0$ be such that $(1-\alpha)x^p+\alpha y^p\ne1$ and
$(1-\alpha)x^q+\alpha y^q\ne1$. Let $A_0$ and $B_\theta$ be given in \eqref{F-3.5}. Then we have
\begin{align}\label{F-3.7}
\cA_{\alpha,p}(A_0,B_\theta)
=\begin{bmatrix}1+\theta^2u_{11}(p)&\theta u_{12}(p)\\
\theta u_{12}(p)&((1-\alpha)x^p+\alpha y^p)^{1/p}+\theta^2u_{22}(p)\end{bmatrix}
+o(\theta^2)\quad\mbox{as $\theta\to0$},
\end{align}
where
\begin{align}\label{F-3.8}
\begin{cases}
u_{11}(p):=-{\alpha(1-\alpha)(1-x^p)(1-y^p)\over p(1-(1-\alpha)x^p-\alpha y^p)}
-{\alpha^2(1-y^p)^2\bigl\{1-((1-\alpha)x^p+\alpha y^p)^{1/p}\bigr\}\over(1-(1-\alpha)x^p-\alpha y^p)^2}, \\
u_{22}(p):={\alpha(1-\alpha)(1-x^p)(1-y^p)((1-\alpha)x^p+\alpha y^p)^{{1\over p}-1}
\over p(1-(1-\alpha)x^p-\alpha y^p)}
+{\alpha^2(1-y^p)^2\bigl\{1-((1-\alpha)x^p+\alpha y^p)^{1/p}\bigr\}\over(1-(1-\alpha)x^p-\alpha y^p)^2}, \\
u_{12}(p):={\alpha(1-y^p)\bigl\{1-((1-\alpha)x^p+\alpha y^p)^{1/p}\bigr\}\over1-(1-\alpha)x^p-\alpha y^p}.
\end{cases}
\end{align}

Furthermore, we have
\begin{equation}\label{F-3.9}
\begin{aligned}
&\det\{\cA_{\alpha,q}(A_0,B_\theta)-\cA_{\alpha,p}(A_0,B_\theta)\} \\
&\quad=\theta^2\Biggl[\alpha(1-\alpha)\biggl\{{(1-x^p)(1-y^p)\over p(1-(1-\alpha)x^p-\alpha y^p)}
-{(1-x^q)(1-y^q)\over q(1-(1-\alpha)x^q-\alpha y^q)}\biggr\} \\
&\qquad\qquad\qquad
\times\bigl\{((1-\alpha)x^q+\alpha y^q)^{1/q}-((1-\alpha)x^p+\alpha y^p)^{1/p}\bigr\} \\
&\qquad\qquad
-\alpha^2\biggl\{{1-y^p\over 1-(1-\alpha)x^p-\alpha y^p}-{1-y^q\over1-(1-\alpha)x^q-\alpha y^q}\biggr\}^2 \\
&\qquad\qquad\qquad\times\bigl\{1-((1-\alpha)x^p+\alpha y^p)^{1/p}\bigr\}
\bigl\{1-((1-\alpha)x^q+\alpha y^q)^{1/q}\bigr\}\Biggr] \\
&\qquad+o(\theta^2)\quad\mbox{as $\theta\to0$}.
\end{aligned}
\end{equation}
\end{lemma}

\begin{proof}
Since $B_\theta^p$ is written as in \eqref{F-3.6} with $y^p$ in place of $y$, it follows that
$X_\theta:=(1-\alpha)A_0^p+\alpha B_\theta^p$ is of the form \eqref{F-3.1} with $a=1$,
$b=(1-\alpha)x^p+\alpha y^p$, $x_{11}=-\alpha(1-y^p)$, $x_{22}=\alpha(1-y^p)$ and $x_{12}=\alpha(1-y^p)$.
Since $(1-\alpha)x^p+\alpha y^p\ne1$ by assumption, we see by Example \ref{E-3.2}(1) that
$\cA_{\alpha,p}(A_0,B_\theta)=X_\theta^{1/p}$ is given in the form \eqref{F-3.7}, where
\begin{align*}
u_{11}(p)&=-{\alpha\over p}(1-y^p)+{{1\over p}-1-{1\over p}((1-\alpha)x^p+\alpha y^p)
+((1-\alpha)x^p+\alpha y^p)^{1/p}\over(1-(1-\alpha)x^p-\alpha y^p)^2}\,\alpha^2(1-y^p)^2 \\
&=-{\alpha\over p}(1-y^p)+{\alpha^2(1-y^p)^2\over p(1-(1-\alpha)x^p-\alpha y^p)}
-{\alpha^2(1-y^p)^2\bigl\{1-((1-\alpha)x^p+\alpha y^p)^{1/p}\bigr\}\over(1-(1-\alpha)x^p-\alpha y^p)^2} \\
&=-{\alpha(1-\alpha)(1-x^p)(1-y^p)\over p(1-(1-\alpha)x^p-\alpha y^p)}
-{\alpha^2(1-y^p)^2\bigl\{1-((1-\alpha)x^p+\alpha y^p)^{1/p}\bigr\}\over(1-(1-\alpha)x^p-\alpha y^p)^2}, \\
u_{22}(p)&={\alpha\over p}(1-y^p)((1-\alpha)x^p+\alpha y^p)^{{1\over p}-1} \\
&\quad
+{1-{1\over p}((1-\alpha)x^p+\alpha y^p)^{{1\over p}-1}+({1\over p}-1)((1-\alpha)x^p+\alpha y^p)^{1/p}
\over(1-(1-\alpha)x^p-\alpha y^p)^2}\,\alpha^2(1-y^p)^2 \\
&={\alpha\over p}(1-y^p)((1-\alpha)x^p+\alpha y^p)^{{1\over p}-1} \\
&\quad
-{\alpha^2(1-y^p)^2((1-\alpha)x^p+\alpha y^p)^{{1\over p}-1}\over p(1-(1-\alpha)x^p-\alpha y^p)}
+{\alpha^2(1-y^p)^2\bigl\{1-((1-\alpha)x^p+\alpha y^p)^{1/p}\bigr\}\over(1-(1-\alpha)x^p-\alpha y^p)^2} \\
&={\alpha(1-\alpha)(1-x^p)(1-y^p)((1-\alpha)x^p+\alpha y^p)^{{1\over p}-1}\over p(1-(1-\alpha)x^p-\alpha y^p)}
+{\alpha^2(1-y^p)^2\bigl\{1-((1-\alpha)x^p+\alpha y^p)^{1/p}\bigr\}\over(1-(1-\alpha)x^p-\alpha y^p)^2}, \\
u_{12}(p)&={\alpha(1-y^p)\bigl\{1-((1-\alpha)x^p+\alpha y^p)^{1/p}\bigr\}\over1-(1-\alpha)x^p-\alpha y^p}.
\end{align*}
Hence \eqref{F-3.7} holds with \eqref{F-3.8}.

Moreover, since
\begin{align*}
&\cA_{\alpha,q}(A_0,B_\theta)-\cA_{\alpha,p}(A_0,B_\theta) \\
&\quad=\small{\begin{bmatrix}\theta^2(u_{11}(q)-u_{11}(p))&\theta(u_{12}(q)-u_{12}(p))\\
\theta(u_{12}(q)-u_{12}(p))&((1-\alpha)x^q+\alpha y^q)^{1/q}-((1-\alpha)x^p+\alpha y^p)^{1/p}
+\theta^2(u_{22}(q)-u_{22}(p))\end{bmatrix}}+o(\theta^2),
\end{align*}
it follows that
\begin{align}
&\det\{\cA_{\alpha,q}(A_0,B_\theta)-\cA_{\alpha,p}(A_0,B_\theta)\} \nonumber\\
&\quad=\theta^2\Bigl[(u_{11}(q)-u_{11}(p))
\bigl\{((1-\alpha)x^q+\alpha y^q)^{1/q}-((1-\alpha)x^p+\alpha y^p)^{1/p}\bigr\} \nonumber\\
&\qquad\qquad-(u_{12}(q)-u_{12}(p))^2\Bigr]+o(\theta^2) \nonumber\\
&\quad=\theta^2\Biggl[\alpha(1-\alpha)\biggl\{{(1-x^p)(1-y^p)\over p(1-(1-\alpha)x^p-\alpha y^p)}
-{(1-x^q)(1-y^q)\over q(1-(1-\alpha)x^q-\alpha y^q)}\biggr\} \nonumber\\
&\qquad\qquad\qquad\quad
\times\bigl\{((1-\alpha)x^q+\alpha y^q)^{1/q}-((1-\alpha)x^p+\alpha y^p)^{1/p}\bigr\} \nonumber\\
&\qquad\qquad\small{+\alpha^2\biggl\{\biggl({(1-y^p)^2w_p\over\zeta_p^2}
-{(1-y^q)^2w_q\over\zeta_q^2}\biggr)(w_p-w_q)
-\biggl({(1-y^q)w_q\over\zeta_q}-{(1-y^p)w_p\over\zeta_p}\biggr)^2\biggr\}\Biggr]} \label{F-3.10}\\
&\qquad+o(\theta^2), \nonumber
\end{align}
where
\begin{align}\label{F-3.11}
\zeta_p:=1-(1-\alpha)x^p-\alpha y^p,\qquad w_p:=1-((1-\alpha)x^p+\alpha y^p)^{1/p}.
\end{align}
Since the term $\alpha^2\{\cdots\}$ in \eqref{F-3.10} is equal to
\[
-\alpha^2\biggl({1-y^p\over\zeta_p}-{1-y^q\over\zeta_q}\biggr)^2w_pw_q,
\]
we have \eqref{F-3.9} as well.
\end{proof}

\begin{lemma}\label{L-3.4}
Let $0<\alpha<1$ and $x,y>0$ be such that $x^{1-\alpha}y^\alpha\ne1$. Let $A_0$ and $B_\theta$ be given
in \eqref{F-3.5}. Then we have
\begin{align}\label{F-3.12}
LE_\alpha(A_0,B_\theta)
=\begin{bmatrix}1+\theta^2v_{11}&\theta v_{12}\\
\theta v_{12}&x^{1-\alpha}y^\alpha+\theta^2v_{22}\end{bmatrix}+o(\theta^2)
\quad\mbox{as $\theta\to0$},
\end{align}
where
\begin{align}\label{F-3.13}
\begin{cases}
v_{11}:={\alpha(1-\alpha)\log x\cdot\log y\over\log x^{1-\alpha}y^\alpha}
-{\alpha^2(1-x^{1-\alpha}y^\alpha)\log^2y\over\log^2x^{1-\alpha}y^\alpha}, \\
v_{22}:=-{\alpha(1-\alpha)x^{1-\alpha}y^\alpha\log x\cdot\log y\over\log x^{1-\alpha}y^\alpha}
+{\alpha^2(1-x^{1-\alpha}y^\alpha)\log^2y\over\log^2x^{1-\alpha}y^\alpha}, \\
v_{12}:={\alpha(1-x^{1-\alpha}y^\alpha)\log y\over\log x^{1-\alpha}y^\alpha }.
\end{cases}
\end{align}

Furthermore, when $p>0$ and $(1-\alpha)x^p+\alpha y^p\ne1$, we have
\begin{equation}\label{F-3.14}
\begin{aligned}
&\det\{\cA_{\alpha,p}(A_0,B_\theta)-LE_\alpha(A_0,B_\theta)\} \\
&\quad=\theta^2\Biggl[-\alpha(1-\alpha)\biggl\{{(1-x^p)(1-y^p)\over p(1-(1-\alpha)x^p-\alpha y^p)}
+{\log x\cdot\log y\over\log x^{1-\alpha}y^\alpha}\biggr\} \\
&\qquad\qquad\qquad\times\bigl\{((1-\alpha)x^p+\alpha y^p)^{1/p}-x^{1-\alpha}y^\alpha\bigr\} \\
&\qquad\qquad-\alpha^2\biggl\{{1-y^p\over1-(1-\alpha)x^p-\alpha y^p}
-{\log y\over\log x^{1-\alpha}y^\alpha}\biggr\}^2(1-x^{1-\alpha}y^\alpha) \\
&\qquad\qquad\qquad\times\bigl\{1-((1-\alpha)x^p+\alpha y^p)^{1/p}\bigr\}\Biggr] \\
&\qquad+o(\theta^2)\quad\mbox{as $\theta\to0$}.
\end{aligned}
\end{equation}
\end{lemma}

\begin{proof}
We note that
\begin{align*}
X_\theta&:=(1-\alpha)\log A_0+\alpha\log B_\theta \\
&\ =(1-\alpha)\begin{bmatrix}0&0\\0&\log x\end{bmatrix}
+\alpha\begin{bmatrix}\theta^2\log y&-\theta\log y\\-\theta\log y&\log y-\theta^2\log y\end{bmatrix}
+o(\theta^2) \\
&\ =\begin{bmatrix}\theta^2\alpha\log y&-\theta\alpha\log y\\
-\theta\alpha\log y&\log x^{1-\alpha}y^\alpha-\theta^2\alpha\log y\end{bmatrix}+o(\theta^2)
\quad\mbox{as $\theta\to0$}.
\end{align*}
Since $\log x^\alpha y^{1-\alpha}\ne0$ by assumption, one can apply Example \ref{E-3.2}(2) with
$b=\log x^{1-\alpha}y^\alpha$, $x_{11}=\alpha\log y$ and $x_{22}=x_{12}=-\alpha\log y$. Hence it
follows that $LE_\alpha(A_0,B_\theta)=e^{X_\theta}$ is given in the form \eqref{F-3.12}, where
\begin{align*}
v_{11}&=\alpha\log y-\biggl({1\over\log x^{1-\alpha}y^\alpha}
+{1-x^{1-\alpha}y^\alpha\over\log^2x^{1-\alpha}y^\alpha}\biggr)(\alpha^2\log^2y) \\
&={\alpha(1-\alpha)\log x\cdot\log y\over\log x^{1-\alpha}y^\alpha}
-{\alpha^2(1-x^{1-\alpha}y^\alpha)\log^2y\over\log^2x^{1-\alpha}y^\alpha}, \\
v_{22}&=-\alpha x^{1-\alpha}y^\alpha\log y+\biggl({x^{1-\alpha}y^\alpha\over\log x^{1-\alpha}y^\alpha}
+{1-x^{1-\alpha}y^\alpha\over\log^2x^{1-\alpha}y^\alpha}\biggr)(\alpha^2\log^2y) \\
&=-{\alpha(1-\alpha)x^{1-\alpha}y^\alpha\log x\cdot\log y\over\log x^{1-\alpha}y^\alpha}
+{\alpha^2(1-x^{1-\alpha}y^\alpha )\log^2y\over\log^2x^{1-\alpha}y^\alpha}, \\
v_{12}&={\alpha(1-x^{1-\alpha}y^\alpha)\log y\over\log x^{1-\alpha}y^\alpha}.
\end{align*}
Hence \eqref{F-3.12} holds with \eqref{F-3.13}.

Moreover, since $(1-\alpha)x^p+\alpha y^p\ne1$ by assumption, it follows from \eqref{F-3.7} and
\eqref{F-3.12} that
\begin{align*}
&\det\{\cA_{\alpha,p}(A_0,B_\theta)-LE_\alpha(A_0,B_\theta)\} \\
&\quad=\det\begin{bmatrix}\theta^2(u_{11}(p)-v_{11})&\theta(u_{12}(p)-v_{12})\\
\theta(u_{12}(p)-v_{12})&((1-\alpha)x^p+\alpha y^p)^{1/p}-x^{1-\alpha}y^\alpha
+\theta^2(u_{22}(p)-v_{22})\end{bmatrix}+o(\theta^2) \\
&\quad=\theta^2\Bigl[(u_{11}(p)-v_{11})
\bigl\{((1-\alpha)x^p+\alpha y^p)^{1/p}-x^{1-\alpha}y^\alpha\bigr\}
-(u_{12}(p)-v_{12})^2\Bigr]+o(\theta^2) \\
&\quad=\theta^2\Biggl[-\alpha(1-\alpha)\biggl\{{(1-x^p)(1-y^p)\over p(1-(1-\alpha)x^p-\alpha y^p)}
+{\log x\cdot\log y\over\log x^{1-\alpha}y^\alpha}\biggr\}
\bigl\{((1-\alpha)x^p+\alpha y^p)^{1/p}-x^{1-\alpha}y^\alpha\bigr\} \\
&\qquad\qquad+\alpha^2\Biggl\{\biggl(-{(1-y^p)^2w_p\over(1-(1-\alpha)x^p-\alpha y^p)^2}
+{w_0\log^2y\over\log^2x^{1-\alpha}y^\alpha}\biggr)(-w_p+w_0) \\
&\qquad\qquad\qquad\quad-\biggl({(1-y^p)w_p\over1-(1-\alpha)x^p-\alpha y^p}
-{w_0\log y\over\log x^{1-\alpha}y^\alpha }\biggr)^2\Biggr\}\Biggr]+o(\theta^2),
\end{align*}
where $w_0:=1-x^{1-\alpha}y^\alpha$ and $w_p$ is given in \eqref{F-3.11}. Since $\alpha^2\{\cdots\}$ in
the last expression is equal to
\[
-\alpha^2\biggl\{{1-y^p\over1-(1-\alpha)x^p-\alpha y^p}
-{\log y\over\log x^{1-\alpha}y^\alpha}\biggr\}^2w_0w_p,
\]
we have \eqref{F-3.14}.
\end{proof}

\begin{lemma}\label{L-3.5}
Let $0<\alpha<1$, $p>0$, and $A_0,B_\theta\in\bM_2^{++}$ be given in \eqref{F-3.5} with $y=x>0$, $x\ne1$.
Then we have
\begin{align}\label{F-3.15}
R_{\alpha,p}(A_0,B_\theta)
=\begin{bmatrix}1+\theta^2z_{11}(p)&\theta z_{12}(p)\\
\theta z_{12}(p)&x+\theta^2z_{22}(p)\end{bmatrix}+o(\theta^2)\quad\mbox{as $\theta\to0$},
\end{align}
where
\begin{align}\label{F-3.16}
\begin{cases}
z_{11}(p):=-{1-x^{\alpha p}\over p}
+{(1-p+px-x^p)\bigl(x^{{1-\alpha\over2}p}-x^{{1+\alpha\over2}p}\bigr)^2\over p(1-x^p)^2}, \\
z_{22}(p):={x^{1-\alpha p}-x\over p}
+{(p+(1-p)x-x^{1-p})\bigl(x^{{1-\alpha\over2}p}-x^{{1+\alpha\over2}p}\bigr)^2\over p(1-x^p)^2}, \\
z_{12}(p):={(1-x)\bigl(x^{{1-\alpha\over2}p}-x^{{1+\alpha\over2}p}\bigr)\over1-x^p}.
\end{cases}
\end{align}
\end{lemma}

\begin{proof}
Since $B_\theta^{\alpha p}$ is written as in \eqref{F-3.6} with $x^{\alpha p}$ in place of $y$, it is easy to
compute
\[
A_0^{{1-\alpha\over2}p}B_\theta^{\alpha p}A_0^{{1-\alpha\over2}p}
=\begin{bmatrix}1-\theta^2(1-x^{\alpha p})&
\theta\bigl(x^{{1-\alpha\over2}p}-x^{{1+\alpha\over2}p}\bigr)\\
\theta\bigl(x^{{1-\alpha\over2}p}-x^{{1+\alpha\over2}p}\bigr)&
x^p+\theta^2(x^{(1-\alpha)p}-x^p)\end{bmatrix}
+o(\theta^2).
\]
Then one can apply Example \ref{E-3.2}(1) to show \eqref{F-3.15} with \eqref{F-3.16} (as in the first
paragraph of the proof of Lemma \ref{L-3.3}), whose details are left to the reader.
\end{proof}

In the last lemma of this section we explain the way of computing the eigenvalues of $X_\theta$ given in
\eqref{F-3.1}.

\begin{lemma}\label{L-3.6}
Let $X_\theta$ be given in \eqref{F-3.1} with $a\ne b$. Then the eigenvalues of $X_\theta$ is given in the
approximate form as
\begin{align}\label{F-3.17}
\lambda(X_\theta)=\begin{cases}
\Bigl(a+\theta^2\bigl(x_{11}+{x_{12}^2\over a-b}\bigr),b+\theta^2\bigl(x_{22}-{x_{12}^2\over a-b}\bigr)\Bigr)
+o(\theta^2)\quad\mbox{if $a>b$}, \\
\Bigl(b+\theta^2\bigl(x_{22}-{x_{12}^2\over a-b}\bigr),a+\theta^2\bigl(x_{11}+{x_{12}^2\over a-b}\bigr)\Bigr)
+o(\theta^2)\quad\mbox{if $a<b$}.\end{cases}
\end{align}
\end{lemma}

\begin{proof}
Since
\[
\det(\lambda I_2-X_\theta)
=\lambda^2-\{a+b+\theta^2(x_{11}+x_{22})\}\lambda
+\{ab+\theta^2(bx_{11}+ax_{22}-x_{12}^2)\}+o(\theta^2),
\]
the eigenvalues $\lambda_\pm(\theta)$ of $X_\theta$ are computed up to $o(\theta^2)$ as
\begin{align*}
2\lambda_\pm(\theta)&=a+b+\theta^2(x_{11}+x_{22}) \\
&\qquad\pm\sqrt{\{a+b+\theta^2(x_{11}+x_{22})\}^2-4\{ab+\theta^2(bx_{11}+ax_{22}-x_{12}^2)\}}+o(\theta^2) \\
&=a+b+\theta^2(x_{11}+x_{22})
\pm\sqrt{(a-b)^2+2\theta^2(a-b)(x_{11}-x_{22})+4\theta^2x_{12}^2}+o(\theta^2) \\
&=a+b+\theta^2(x_{11}+x_{22})
\pm|a-b|\biggl\{1+\theta^2{x_{11}-x_{22}\over a-b}+2\theta^2{x_{12}^2\over(a-b)^2}\biggr\}+o(\theta^2).
\end{align*}
Hence it is immediate to verify that $\lambda(X_\theta)=(\lambda_+(\theta),\lambda_-(\theta))$ has expression
\eqref{F-3.17}.
\end{proof}

\section{Various quasi-arithmetic-geometric inequalities}\label{Sec-4}

Throughout the section let $0<\alpha<1$ and $p,q>0$. For each quasi matrix mean $\cM_{\alpha,p}$ from
$\cA_{\alpha,p}$, $LE_\alpha$, $R_{\alpha,p}$, $G_{\alpha,p}$, $SG_{\alpha,p}$, $\widetilde SG_{\alpha,p}$
and for each matrix ordering $\triangleleft$ from $\le$, $\le_\chao$, $\le_\near$, $\le_\lambda$, $\prec_w$,
$\le_\Tr$, we will examine the inequality $\cM_{\alpha,p}(A,B)\triangleleft\cA_{\alpha,q}(A,B)$. Our goal is to
find the necessary and sufficient condition of $p,q,\alpha$ for the inequality to hold for all $A,B>0$, though
we have not succeeded it for all cases. In this paper we do not deal with the inequality
$\cH_{\alpha,q}\triangledown\cM_{\alpha,p}$ (see Remark \ref{R-2.7}(3) and remark (4) of Section \ref{Sec-5}).
Also we do not include the ordering $\prec_{w\log}$ in our considerations, though it would be meaningful to
examine the differences between $\le_\lambda$, $\prec_{w\log}$, $\prec_w$ (see Proposition \ref{P-2.4}(3))
in quasi-arithmetic-geometric mean inequalities.

The section is divided into six subsections.

\subsection{$\cA_{\alpha,p}$ vs.\ $\cA_{\alpha,q}$}\label{Sec-4.1}

The next theorem characterizes the inequality $\cA_{\alpha,p}\le\cA_{\alpha,q}$, which was formerly shown
in \cite{AuHi} with restriction to $\alpha=1/2$.

\begin{theorem}\label{T-4.1}
Let $0<\alpha<1$ and $p,q>0$. Then the following conditions are equivalent:
\begin{itemize}
\item[(i)] $\cA_{\alpha,p}(A,B)\le\cA_{\alpha,q}(A,B)$ for all $A,B\ge0$;
\item[(ii)] $\cA_{\alpha,p}(A,B)\le\cA_{\alpha,q}(A,B)$ for all $A,B\in\bM_2^{++}$;
\item[(iii)] $p=q$ or $1\le p<q$ or $1/2\le p<1\le q$.
\end{itemize}
\end{theorem}

To prove the theorem, in addition to Lemma \ref{L-3.3} we need the following lemma.

\begin{lemma}\label{L-4.2}
Define $\tilde A_0,\tilde B_\theta\in\bM_2^+$ by
\begin{align}\label{F-4.1}
\tilde A_0:=\begin{bmatrix}2&0\\0&0\end{bmatrix},\qquad
\tilde B_\theta:=\begin{bmatrix}\cos^2\theta&\cos\theta\sin\theta\\
\cos\theta\sin\theta&\sin^2\theta\end{bmatrix},
\end{align}
for $\theta\in\bR$. Let $0<\alpha<1$ and $0<p,q<1$. Then we have
\begin{equation}\label{F-4.2}
\begin{aligned}
&\det\{\cA_{\alpha,q}(\tilde A_0,\tilde B_\theta)-\cA_{\alpha,p}(\tilde A_0,\tilde B_\theta)\} \\
&\quad=-\theta^2(\alpha+(1-\alpha)2^p)^{1/p}(\alpha+(1-\alpha)2^q)^{1/q}
\biggl({\alpha\over\alpha+(1-\alpha)2^p}-{\alpha\over\alpha+(1-\alpha)2^q}\biggr)^2 \\
&\qquad+o(\theta^2)\quad\mbox{as $\theta\to0$}.
\end{aligned}
\end{equation}
\end{lemma}

\begin{proof}
The proof is a modification of that of \cite[Lemma 3.3]{AuHi} where the case $\alpha=1/2$ was treated. It is easy to
check that
\begin{align}
(1-\alpha)\tilde A_0^p+\alpha\tilde B_\theta^p
&=\begin{bmatrix}(1-\alpha)2^p+\alpha(1-\sin^2\theta)&{\alpha\sin2\theta\over2}\\
{\alpha\sin2\theta\over2}&\alpha\sin^2\theta\end{bmatrix} \label{F-4.3}\\
&={\alpha+(1-\alpha)2^p\over2}\begin{bmatrix}1+a&b\\b&1-a\end{bmatrix}, \nonumber
\end{align}
where
\[
a:=1-{2\alpha\sin^2\theta\over\alpha+(1-\alpha)2^p},\qquad
b:={\alpha\sin2\theta\over\alpha+(1-\alpha)2^p}.
\]
Letting $c:=\sqrt{a^2+b^2}$ ($\le1$), as in the proof of \cite[Lemma 3.3]{AuHi} one can compute
\begin{align*}
&\bigl((1-\alpha)\tilde A_0^p+\alpha\tilde B_\theta^p\bigr)^{1/p} \\
&\quad=\biggl({\alpha+(1-\alpha)2^p\over2}\biggr)^{1/p} \\
&\qquad\times\begin{bmatrix}{(1+c)^{1/p}+(1-c)^{1/p}\over2}+{(1+c)^{1/p}-(1-c)^{1/p}\over2c}a&
{(1+c)^{1/p}-(1-c)^{1/p}\over2c}b\\{(1+c)^{1/p}-(1-c)^{1/p}\over2c}b&
{(1+c)^{1/p}+(1-c)^{1/p}\over2}-{(1+c)^{1/p}-(1-c)^{1/p}\over2c}a\end{bmatrix}.
\end{align*}
As $\theta\to0$ we estimate
\[
a=1-{2\alpha\theta^2\over\alpha+(1-\alpha)2^p}+o(\theta^2),\qquad
b={2\alpha\theta\over\alpha+(1-\alpha)2^p}+o(\theta),
\]
so that
\[
c^2=1-{4\alpha\theta^2\over\alpha+(1-\alpha)2^p}
+{4\alpha^2\theta^2\over(\alpha+(1-\alpha)2^p)^2}+o(\theta^2)
=1-{\alpha(1-\alpha)2^{p+2}\theta^2\over(\alpha+(1-\alpha)2^p)^2}+o(\theta^2).
\]
Hence we have
\[
c=1-{\alpha(1-\alpha)2^{p+1}\theta^2\over(\alpha+(1-\alpha)2^p)^2}+o(\theta^2),\qquad
{1\over c}=1+{\alpha(1-\alpha)2^{p+1}\theta^2\over(\alpha+(1-\alpha)2^p)^2}+o(\theta^2),
\]
and moreover
\[
(1+c)^{1/p}=2^{1/p}\biggl\{1-{\alpha(1-\alpha)2^p\theta^2\over p(\alpha+(1-\alpha)2^p)^2}\biggr\}+o(\theta^2),
\qquad(1-c)^{1/p}=o(\theta^2)
\]
thanks to $0<p<1$. Therefore, we arrive at
\[
\cA_{\alpha,p}(\tilde A_0,\tilde B_\theta)
=\begin{bmatrix}s_{11}(p)&s_{12}(p)\\s_{12}(p)&s_{22}(p)\end{bmatrix},
\]
where
\begin{align*}
s_{11}(p)&=\biggl({\alpha+(1-\alpha)2^p\over2}\biggr)^{1/p}
\biggl[2^{{1\over p}-1}\biggl\{1-{\alpha(1-\alpha)2^p\theta^2\over p(\alpha+(1-\alpha)2^p)^2}\biggr\} \\
&\quad+2^{{1\over p}-1}\biggl\{1-{\alpha(1-\alpha)2^p\theta^2\over p(\alpha+(1-\alpha)2^p)^2}\biggr\}
\biggl\{1+{\alpha(1-\alpha)2^{p+1}\theta^2\over(\alpha+(1-\alpha)2^p)^2}\biggr\}
\biggl\{1-{2\alpha\theta^2\over\alpha+(1-\alpha)2^p}\biggr\}\biggr]+o(\theta^2) \\
&=(\alpha+(1-\alpha)2^p)^{1/p}\biggl\{1-{\alpha(1-\alpha)2^p\theta^2\over p(\alpha+(1-\alpha)2^p)^2}
+{\alpha(1-\alpha)2^p\theta^2\over(\alpha+(1-\alpha)2^p)^2}-{\alpha\theta^2\over\alpha+(1-\alpha)2^p}
\biggr\}+o(\theta^2) \\
&=(\alpha+(1-\alpha)2^p)^{1/p}\biggl\{1-{\alpha(p\alpha+(1-\alpha)2^p)\over p(\alpha+(1-\alpha)2^p)^2}
\theta^2\biggr\}+o(\theta^2), \\
s_{22}(p)&=\biggl({\alpha+(1-\alpha)2^p\over2}\biggr)^{1/p}
\Biggl[2^{{1\over p}-1}\biggl\{1-{\alpha(1-\alpha)2^p\theta^2\over p(\alpha+(1-\alpha)2^p)^2}\biggr\} \\
&\quad-2^{{1\over p}-1}\biggl\{1-{\alpha(1-\alpha)2^p\theta^2\over p(\alpha+(1-\alpha)2^p)^2}\biggr\}
\biggl\{1+{\alpha(1-\alpha)2^{p+1}\theta^2\over(\alpha+(1-\alpha)2^p)^2}\biggr\}
\biggl\{1-{2\alpha\theta^2\over\alpha+(1-\alpha)2^p}\biggr\}\Biggr]+o(\theta^2) \\
&=(\alpha+(1-\alpha)2^p)^{1/p}\biggl\{-{\alpha(1-\alpha)2^p\theta^2\over(\alpha+(1-\alpha)2^p)^2}
+{\alpha\theta^2\over\alpha+(1-\alpha)2^p}\biggr\}+o(\theta^2) \\
&=\alpha^2(\alpha+(1-\alpha)2^p)^{{1\over p}-2}\theta^2+o(\theta^2), \\
s_{12}(p)&=\biggl({\alpha+(1-\alpha)2^p\over2}\biggr)^{1/p}
2^{{1\over p}-1}\biggl\{1-{\alpha(1-\alpha)2^p\theta^2\over p(\alpha+(1-\alpha)2^p)^2}\biggr\}
\biggl\{1+{\alpha(1-\alpha)2^{p+1}\theta^2\over(\alpha+(1-\alpha)2^p)^2}\biggr\} \\
&\quad\times\biggl\{{2\alpha\theta\over\alpha+(1-\alpha)2^p}+o(\theta)\biggr\}+o(\theta^2) \\
&=\alpha(\alpha+(1-\alpha)2^p)^{{1\over p}-1}\theta+o(\theta).
\end{align*}
Therefore, we obtain
\begin{align*}
&\det\{\cA_{\alpha,q}(\tilde A_0,\tilde B_\theta)-\cA_{\alpha,p}(\tilde A_0,\tilde B_\theta)\} \\
&\quad=(s_{11}(q)-s_{11}(p))(s_{22}(q)-s_{22}(p))-(s_{12}(q)-s_{12}(p))^2 \\
&\quad=\scriptsize{\theta^2\alpha^2\biggl[\Bigl\{(\alpha+(1-\alpha)2^q)^{1/q}-(\alpha+(1-\alpha)2^p)^{1/p}\Bigr\}
\Bigl\{(\alpha+(1-\alpha)2^q)^{{1\over q}-2}-(\alpha+(1-\alpha)2^p)^{{1\over p}-2}\Bigr\}} \\
&\qquad\qquad\quad
-\scriptsize{\Bigl\{(\alpha+(1-\alpha)2^q)^{{1\over q}-1}-(\alpha+(1-\alpha)2^p)^{{1\over p}-1}\Bigr\}^2
\biggr]+o(\theta^2)} \\
&\quad=\small{-\theta^2(\alpha+(1-\alpha)2^p)^{1/p}(\alpha+(1-\alpha)2^q)^{1/q}
\biggl({\alpha\over\alpha+(1-\alpha)2^p}-{\alpha\over\alpha+(1-\alpha)2^q}\biggr)^2+o(\theta^2)},
\end{align*}
as asserted.
\end{proof}

Note that although \eqref{F-4.3} is written in the form of \eqref{F-3.1} with $a=\alpha+(1-\alpha)2^p$ and
$b=0$, we cannot apply Lemma \ref{L-3.1} to prove Lemma \ref{L-4.2} because $x^{1/p}$ is not twice
differentiable at $x=0$ when $p>1/2$.

Now we are ready to prove Theorem \ref{T-4.1}.

\begin{proof}[Proof of Theorem \ref{T-4.1}]\enspace
(iii)$\implies$(i) is seen from \cite[Theorem 2.1]{AuHi}. Indeed, the proof is easy as follows. If $1\le p<q$
then
\[
((1-\alpha)A^q+\alpha B^q)^{1/q}=\bigl(((1-\alpha)A^q+\alpha B^q)^{p/q}\bigr)^{1/p}
\ge((1-\alpha)A^p+\alpha B^p)^{1/p},
\]
since $x^{p/q}$ is operator concave and $x^{1/p}$ is operator monotone on $[0,\infty)$. If $1/2\le p<1\le q$
then
\[
((1-\alpha)A^p+\alpha B^p)^{1/p}\le(1-\alpha)A+\alpha B\le((1-\alpha)A^q+\alpha B^q)^{1/q},
\]
since $x^{1/p}$ is operator convex and $x^{1/q}$ is operator concave on $[0,\infty)$.

(i)$\implies$(ii) is trivial.

(ii)$\implies$(iii).\enspace
By considering the inequality for $A=xI_2$ and $B=yI_2$, $x,y>0$, (ii) implies that $p\le q$. Hence it suffices
to show that (ii) fails to hold when $0<p<1/2$ and $q>p$ and when $0<p<q<1$.

First, assume that $0<p<1/2$ and $q>p$. Consider $A_0$ and $B_\theta$ in \eqref{F-3.5} with $0<y<1$ and
$x=y^2$. Since $(1-\alpha)y^{2p}+\alpha y^p<1$ and $(1-\alpha)y^{2q}+\alpha y^q<1$ clearly, we can apply
Lemma \ref{L-3.3}. As $y\searrow0$ we estimate
\begin{align*}
&((1-\alpha)y^{2q}+\alpha y^q)^{1/q}-((1-\alpha)y^{2p}+\alpha y^p)^{1/p} \\
&\quad=y((1-\alpha)y^q+\alpha)^{1/q}-y((1-\alpha)y^p+\alpha)^{1/p}
\approx(\alpha^{1/q}-\alpha^{1/p})y
\end{align*}
and
\begin{align*}
&{1-y^p\over1-(1-\alpha)y^{2p}-\alpha y^p}-{1-y^q\over1-(1-\alpha)y^{2q}-\alpha y^q} \\
&\quad=(1-\alpha){-y^p+y^q+y^{2p}-y^{2q}+y^{p+2q}-y^{2p+q}\over
(1-(1-\alpha)y^{2p}-\alpha y^p)(1-(1-\alpha)y^{2q}-\alpha y^q)}\approx-(1-\alpha)y^p.
\end{align*}
Therefore, the dominant term of the big bracket $[\cdots]$ of the RHS of \eqref{F-3.9} is
\[
\alpha(1-\alpha)\biggl({1\over p}-{1\over q}\biggr)(\alpha^{1/q}-\alpha^{1/p})y
-\alpha^2(1-\alpha)^2y^{2p}<0
\]
thanks to $2p<1$ when $y>0$ is sufficiently small. For such a $y>0$ the RHS of \eqref{F-3.9} $<0$
if $\theta$ is small enough, which implies that
$\cA_{\alpha,p}(A_0,B_\theta)\not\le\cA_{\alpha,q}(A_0,B_\theta)$.

Next, assume that $0<p<q<1$, and apply Lemma \ref{L-4.2}. Let $\tilde A_0$ and $\tilde B_\theta$ be
given in \eqref{F-4.1}. From expression \eqref{F-4.2} it follows that
$\det\{\cA_{\alpha,q}(\tilde A_0,\tilde B_\theta)-\cA_{\alpha,p}(\tilde A_0,\tilde B_\theta)\}<0$ so that
$\cA_{\alpha,p}(\tilde A_0,\tilde B_\theta)\not\le\cA_{\alpha,q}(\tilde A_0,\tilde B_\theta)$ if $\theta$ is
small enough. By continuity there are $A,B\in\bM_2^{++}$ such that
$\cA_{\alpha,p}(A,B)\not\le\cA_{\alpha,q}(A,B)$.
\end{proof}

The inequalities between $\cA_{\alpha,p}$ and $\cA_{\alpha,q}$ with respect to other weaker orderings are
simple as stated in the next proposition.

\begin{proposition}\label{P-4.3}
Let $0<\alpha<1$ and $p,q>0$. Then the following conditions are equivalent:
\begin{itemize}
\item[(i)] $\cA_{\alpha,p}(A,B)\le_\chao\cA_{\alpha,q}(A,B)$ for all $A,B\ge0$;
\item[(ii)] $\cA_{\alpha,p}(a,b)\le\cA_{\alpha,q}(a,b)$ for all scalars $a,b>0$;
\item[(iii)] $p\le q$.
\end{itemize}

Hence, for any $\triangleleft\in\{\le_\chao,\le_\near,\le_\lambda,\prec_w,\le_\Tr\}$,
$\cA_{\alpha,p}\triangleleft\cA_{\alpha,q}$ holds for all $A,B\ge0$ if and only if $p\le q$.
\end{proposition}

\begin{proof}
(i)$\implies$(ii) is obvious, and (ii)$\implies$(iii) is easy since (iii) means that $x^{q/p}$ is convex on
$(0,\infty)$. Finally, let us show (iii)$\implies$(i). Assume that $p\le q$. Let $A,B\ge0$ be arbitrary. Note that
\[
s(\cA_{\alpha,p}(A,B))=s((1-\alpha)A^p+\alpha B^p)=s(A)\vee s(B)=:P
\]
and similarly $s(\cA_{\alpha,q}(A,B))=P$. Since $x^{p/q}$ is operator concave on $[0,\infty)$, one has
\[
(1-\alpha)A^p+\alpha B^p\le((1-\alpha)A^q+\alpha B^q)^{p/q},
\]
so that
\begin{align*}
P\{\log\cA_{\alpha,p}(A,B)\}P&={1\over p}P\{\log((1-\alpha)A^p+\alpha B^p)\}P \\
&\le{1\over q}P\{\log((1-\alpha)A^q+\alpha B^q)\}P=P\{\log\cA_{\alpha,q}(A,B)\}P.
\end{align*}
Hence (i) holds.
\end{proof}

The next proposition says that $p>0\mapsto\Tr\,\cA_{\alpha,p}(A,B)$ is strictly increasing unless $A=B$.

\begin{proposition}\label{P-4.4}
Let $0<\alpha<1$ and $0<p<q$. Then for every $A,B\ge0$ the following conditions are equivalent:
\begin{itemize}
\item[(i)] $\Tr\,\cA_{\alpha,p}(A,B)=\Tr\,\cA_{\alpha,q}(A,B)$;
\item[(ii)] $\cA_{\alpha,p}(A,B)=\cA_{\alpha,q}(A,B)$;
\item[(iii)] $A=B$.
\end{itemize}
\end{proposition}

\begin{proof}
It is obvious that (iii)$\implies$(ii)$\implies$(i). To show (i)$\implies$(iii), assume (i) so that
$\Tr\,\cA_{\alpha,t}(A,B)$ is constant for $t\in[p,q]$ by Proposition \ref{P-4.3}. Since the function
$t\mapsto\Tr\,\cA_{\alpha,t}(A,B)$ is real analytic in $t>0$, it follows that $\Tr\,\cA_{\alpha,t}(A,B)$ is
constant for all $t>0$. In particular, $\Tr\,\cA_{\alpha,1/2}(A,B)=\Tr\,\cA_{\alpha,1}(A,B)$, which gives
$\Tr(A^{1/2}-B^{1/2})^2=0$ so that $A=B$.
\end{proof}

The results of this subsection are summarized as follows:

\medskip
\begin{table}[htb]
\centering
\begin{tabular}{|c|l|l|} \hline
& Sufficient cond. & Necessary cond. \\ \hline

$\cA_{\alpha,q}\le\cA_{\alpha,q}$ & \small{$p=q$ or $1\le p<q$ or $1/2\le p<1\le q$} &
\small{$p=q$ or $1\le p<q$ or $1/2\le p<1\le q$} \\ \hline

$\begin{array}{ccccc}\cA_{\alpha,p}\le_\chao\cA_{\alpha,q} \\ \cA_{\alpha,p}\le_\near\cA_{\alpha,q} \\
\cA_{\alpha,p}\le_\lambda\cA_{\alpha,q} \\ \cA_{\alpha,p}\prec_w\cA_{\alpha,q} \\
\cA_{\alpha,p}\le_\Tr\cA_{\alpha,q}\end{array}$ & $p\le q$ & $p\le q$ \\ \hline
\end{tabular}
\end{table}

\subsection{$LE_\alpha$ vs.\ $\cA_{\alpha,q}$}\label{Sec-4.2}

The order relations between $LE_\alpha$ and $\cA_{\alpha,p}$ are clarified in the next theorem and
proposition. Theorem \ref{T-4.5} was formerly shown in \cite{AuHi} with restriction to $\alpha=1/2$.

\begin{theorem}\label{T-4.5}
For any $\alpha\in(0,1)$ and any $q>0$ there exist $A,B\in\bM_2^{++}$ such that
$LE_\alpha(A,B)\not\le\cA_{\alpha,q}(A,B)$.
\end{theorem}

\begin{proof}
Let $0<\alpha<1$ and $q>0$. Assume by contradiction that $LE_\alpha(A,B)\le\cA_{\alpha,p}(A,B)$ for all
$A,B\in\bM_n^{++}$. Then $LE_\alpha(A_0,B_\theta)\le\cA_{\alpha,p}(A_0,B_\theta)$ holds for
$A_0,B_\theta$ in \eqref{F-3.5} with any $x,y>0$ and $\theta\in\bR$. Let $0<y<1$ and $x=y^2$. Since
$x^{1-\alpha}y^\alpha \ne1$ and $(1-\alpha)x^q+\alpha y^q\ne1$, by Lemma \ref{L-3.4} we must have
\begin{equation}\label{F-4.4}
\begin{aligned}
&-{1-\alpha\over\alpha}\biggl\{{(1-y^{2p})(1-y^p)\over p(1-(1-\alpha)y^{2p}-\alpha y^p)}
+{2\over2-\alpha}\log y\biggr\}\bigl\{((1-\alpha)y^{2p}+\alpha y^p)^{1/p}-y^{2-\alpha}\bigr\} \\
&\quad-\biggl\{{1-y^p\over1-(1-\alpha)y^{2p}-\alpha y^p}-{1\over2-\alpha}\biggr\}^2
(1-y^{2-\alpha})\bigl\{1-((1-\alpha)y^{2p}+\alpha y^p)^{1/p}\bigr\}\ge0.
\end{aligned}
\end{equation}
As $y\searrow0$, since we have the approximations
\begin{align*}
{(1-y^{2p})(1-y^p)\over p(1-(1-\alpha)y^{2p}-\alpha y^p)}+{2\over2-\alpha}\log y
&\approx{2\over2-\alpha}\log y, \\
((1-\alpha)y^{2p}+\alpha y^p)^{1/p}-y^{2-\alpha}&\approx\alpha^{1/p}y,
\end{align*}
the LHS of \eqref{F-4.4} can be estimated to
\[
-{2(1-\alpha)\alpha^{{1\over p}-1}\over2-\alpha}\,y\log y-\biggl(1-{1\over2-\alpha}\biggr)^2,
\]
which is negative when $y>0$ is sufficiently small. This contradicts \eqref{F-4.4}.
\end{proof}

\begin{proposition}\label{P-4.6}
Let $0<\alpha<1$ and $q>0$. Then for every $A,B\ge0$ we have
$LE_\alpha(A,B)\le_\chao\cA_{\alpha,q}(A,B)$. Hence $LE_\alpha(A,B)\triangleleft\cA_{\alpha,q}(A,B)$
holds for any $\triangleleft\in\{\le_\chao,\le_\near,\le_\lambda,\prec_w,\le_\Tr\}$ and $A,B\ge0$.
\end{proposition}

\begin{proof}
When $0<p<q$, by Proposition \ref{P-4.3} we have $\cA_{\alpha,p}(A,B)\le_\chao\cA_{\alpha,q}(A,B)$ for
any $A,B\ge0$. Letting $p\searrow0$ gives $LE_\alpha(A,B)\le_\chao\cA_{\alpha,q}(A,B)$ by
Theorem \ref{T-2.3} and Lemma \ref{L-2.5}.
\end{proof}

In the above proof, by Theorem \ref{T-2.6} we may assume that $A,B>0$. Thus, the simpler version of
Theorem \ref{T-2.3} for $A,B>0$ is enough to prove the above proposition. However, we cannot
completely avoid use of Theorem \ref{T-A.1} because it is necessary in proving Lemma \ref{L-2.5} via
Proposition \ref{P-2.4}(2).

\begin{proposition}\label{P-4.7}
Let $0<\alpha<1$ and $q>0$. Then for every $A,B\ge0$ the following conditions are equivalent:
\begin{itemize}
\item[(i)] $\Tr\,LE_\alpha(A,B)=\Tr\,\cA_{\alpha,q}(A,B)$;
\item[(ii)] $LE_\alpha(A,B)=\cA_{\alpha,q}(A,B)$;
\item[(iii)] $A=B$.
\end{itemize}
\end{proposition}

\begin{proof}
It is obvious that (iii)$\implies$(ii)$\implies$(1). If (i) holds, then we have
$\Tr\,\cA_{\alpha,p}(A,B)=\Tr\,\cA_{\alpha,q}(A,B)$ for any $p\in(0,q]$ by Propositions \ref{P-4.3} and
\ref{P-4.6}. Hence (iii) follows by Proposition \ref{P-4.4}.
\end{proof}

The results of this subsection are summarized as follows:

\medskip
\begin{table}[htb]
\centering
\begin{tabular}{|c|l|l|} \hline
& Sufficient cond. & Necessary cond. \\ \hline

$LE_\alpha\le\cA_{\alpha,q}$ & none & \\ \hline

$\begin{array}{ccccc}LE_\alpha\le_\chao\cA_{\alpha,q} \\ LE_\alpha\le_\near\cA_{\alpha,q} \\
LE_\alpha\le_\lambda\cA_{\alpha,q} \\ LE_\alpha\prec_w\cA_{\alpha,q} \\
LE_\alpha\le_\Tr\cA_{\alpha,q}\end{array}$ & arbitrary $q$ & \\ \hline
\end{tabular}
\end{table}

\subsection{$R_{\alpha,p}$ vs.\ $\cA_{\alpha,q}$}\label{Sec-4.3}

In this subsection we discuss inequalities between $R_{\alpha,p}$ and $\cA_{\alpha,q}$. The next theorem
says that the near order inequality $R_{\alpha,p}(A,B)\le_\near\cA_{\alpha,q}(A,B)$ fails to hold for any
$p,q>0$.

\begin{theorem}\label{T-4.8}
For any $\alpha\in(0,1)$ and any $p,q>0$ there exist $A,B\in\bM_2^{++}$ such that
$R_{\alpha,p}(A,B)\not\le_\near\cA_{\alpha,q}(A,B)$ (hence $R_{\alpha,p}(A,B)\not\le\cA_{\alpha,q}(A,B)$
and $R_{\alpha,p}(A,B)\not\le_\chao\cA_{\alpha,q}(A,B)$).
\end{theorem}

\begin{proof}
Consider $A_0,B_\theta\in\bM_2^{++}$ given in \eqref{F-3.5} with $y=x\ne1$. Let $0<\alpha<1$ and
$p,q>0$ be arbitrarily fixed. The estimates of $R_{\alpha,p}(A_0,B_\theta)$ and
$\cA_{\alpha,q}(A_0,B_\theta)$ up to $o(\theta)$ are enough to show the theorem.  Assume that
$R_{\alpha,p}(A_0,B_\theta)\le_\near\cA_{\alpha,q}(A_0,B_\theta)$ for any $x>0$ with $x\ne1$ and any
$\theta\in\bR$. Now write $X:=R_{\alpha,p}(A_0,B_\theta)$ and $Y:=\cA_{\alpha,q}(A_0,B_\theta)$.
Reducing \eqref{F-3.15} to $o(\theta)$ (see \eqref{F-3.4}) one has
$X=\begin{bmatrix}1&\theta z_{12}\\\theta z_{12}&x\end{bmatrix}+o(\theta)$ with $z_{12}$ in
\eqref{F-3.16} (without variable $p$ for simplicity). On the other hand, reducing \eqref{F-3.7} and \eqref{F-3.8}
with $y=x\ne1$ one has $Y=\begin{bmatrix}1&\theta\alpha(1-x)\\\theta\alpha(1-x)&x\end{bmatrix}+o(\theta)$,
and applies \eqref{F-3.4} to have
$Y^{1/2}=\begin{bmatrix}1&\theta\alpha(1-x^{1/2})\\\theta\alpha(1-x^{1/2})&x^{1/2}\end{bmatrix}+o(\theta)$
as $\theta\to0$.
Hence we have
\[
Y^{1/2}XY^{1/2}=\begin{bmatrix}1&\theta u_{12}\\\theta u_{12}&x^2\end{bmatrix}+o(\theta)
\quad\mbox{as $\theta\to0$},
\]
where
\[
u_{12}:=x^{1/2}z_{12}+\alpha(1-x^{1/2})(1+x^{3/2}),
\]
so that by \eqref{F-3.4} again we can further compute 
\[
(Y^{1/2}XY^{1/2})^{1/2}
=\begin{bmatrix}1&\theta{u_{12}\over1+x}\\\theta{u_{12}\over1+x}&x\end{bmatrix}+o(\theta).
\]
Then, since $X\le_\near Y$ is equivalent to $(Y^{1/2}XY^{1/2})^{1/2}\le Y$,
we must have
\begin{align*}
0&\le\det\bigl(Y-(Y^{1/2}XY^{1/2})^{1/2}\bigr)
=\det\biggl(\theta\begin{bmatrix}0&\alpha(1-x)-{u_{12}\over1+x}\\
\alpha(1-x)-{u_{12}\over1+x}&0\end{bmatrix}+o(\theta)\biggr) \\
&=-\theta^2\Bigl\{\alpha(1-x)-{u_{12}\over1+x}\Bigr\}^2+o(\theta^2)\quad\mbox{as $\theta\to0$}.
\end{align*}
Indeed, the last equality is seen as, with $v\in\bR$,
\[
\det\begin{bmatrix}o(\theta)&\theta v+o(\theta)\\\theta v+o(\theta)&o(\theta)\end{bmatrix}
=o(\theta)^2-(\theta v+o(\theta))^2=o(\theta)^2-\theta^2v^2+o(\theta)\theta v=-\theta^2v^2+o(\theta^2).
\]
Therefore, $\alpha(1-x)={u_{12}\over1+x}$ must hold, that is,
$x^{1/2}z_{12}+\alpha(1-x^{1/2})(1+x^{3/2})=\alpha(1-x^2)$. Inserting the form of $z_{12}$ in \eqref{F-3.16}
gives
\[
{x^{1/2}(1-x)\bigl(x^{{1-\alpha\over2}p}-x^{{1+\alpha\over2}p}\bigr)\over1-x^p}
=\alpha\{(1-x^2)-(1-x^{1/2})(1+x^{3/2})\}=\alpha x^{1/2}(1-x),
\]
so that $\bigl(x^{{1-\alpha\over2}p}-x^{{1+\alpha\over2}p}\bigr)/(1-x^p)=\alpha$ for all $x>0$, $x\ne1$. But
this is impossible since the LHS goes to $0$ ($\ne\alpha$) as $x\searrow0$.
\end{proof}

\begin{proposition}\label{P-4.9}
Let $0<\alpha<1$ and $p,q>0$. If $p/2\le q$, then we have
$R_{\alpha,p}(A,B))\le_\lambda\cA_{\alpha,q}(A,B)$ (hence
$R_{\alpha,p}(A,B)\prec_w\cA_{\alpha,q}(A,B)$) for all $A,B\ge0$.
\end{proposition}

\begin{proof}
Ando's matrix Young inequality \cite{An} says that for any $\alpha\in(0,1)$ and $A,B\in\bM_n^+$,
\[
\lambda(|A^{1-\alpha}B^\alpha|)\le\lambda((1-\alpha)A+\alpha B),
\]
equivalently, $\lambda^{1/2}(A^{1-\alpha}B^{2\alpha}A^{1-\alpha})\le\lambda((1-\alpha)A+\alpha B)$. For
any $p>0$, by replacing $A,B$ with $A^{p/2},B^{p/2}$ respectively and taking the $2/p$-power, one has
\[
\lambda^{1/p}\bigl(B^{{1-\alpha\over2}p}A^{\alpha p}B^{{1-\alpha\over2}p}\bigr)
\le\lambda^{2/p}(\alpha A^{p/2}+(1-\alpha)B^{p/2}),
\]
that is, $R_{\alpha,p}(A,B)\le_\lambda\cA_{\alpha,p/2}(A,B)$. Hence the assertion follows from
Proposition \ref{P-4.3}.
\end{proof}

\begin{theorem}\label{T-4.10}
Let $0<\alpha<1$ and $p,q>0$. If $R_{\alpha,p}(A,B)\le_\lambda\cA_{\alpha,q}(A,B)$ holds for all
$A,B\in\bM_2^{++}$, then $\alpha(1-\alpha)p\le q$.
\end{theorem}

\begin{proof}
Let $A_0$ and $B_\theta$ be given in \eqref{F-3.5} with $y=x>0$ with $x\ne1$. Then we have expression
\eqref{F-3.15} with \eqref{F-3.16}. On the other hand, in the present case (where $y=x$), expression
\eqref{F-3.7} with \eqref{F-3.8} is simplified as
\begin{align}\label{F-4.5}
\cA_{\alpha,q}(A_0,B_\theta)
=\begin{bmatrix}1+\theta^2w_{11}(q)&\theta w_{12}(q)\\
\theta w_{12}(q)&x+\theta^2w_{22}(q)\end{bmatrix}
+o(\theta^2),
\end{align}
where
\begin{align}\label{F-4.6}
\begin{cases}
w_{11}(q):={\alpha(1-\alpha)\over q}(x^q-1)-\alpha^2(1-x), \\
w_{22}(q):={\alpha(1-\alpha)\over q}(x^{1-q}-x)+\alpha^2(1-x), \\
w_{12}(q):=\alpha(1-x).
\end{cases}
\end{align}
By assumption we must have $R_{\alpha,p}(A_0,B_\theta)\le_\lambda\cA_{\alpha,q}(A_0,B_\theta)$ for all
$\theta\in\bR$, which implies by Lemma \ref{L-3.6} that
\begin{align}
z_{11}(p)+{z_{12}(p)^2\over1-x}&\le w_{11}(q)+{w_{12}(q)^2\over1-x}, \label{F-4.7}\\
z_{22}(p)-{z_{12}(p)^2\over1-x}&\le w_{22}(q)-{w_{12}(q)^2\over1-x} \label{F-4.8}
\end{align}
for all $x>0$, $x\ne1$. Now, we compute from \eqref{F-3.16}
\begin{align*}
z_{11}(p)+{z_{12}(p)^2\over1-x}
&={-1-x^p+x^{\alpha p}+x^{(1-\alpha)p}\over p(1-x^p)}, \\
z_{22}(p)-{z_{12}(p)^2\over1-x}
&={x\bigl(1+x^p-x^{\alpha p}-x^{(1-\alpha)p}\bigr)\over p(1-x^p)},
\end{align*}
and from \eqref{F-4.6}
\begin{align*}
w_{11}(q)+{w_{12}(q)^2\over1-x}
&={\alpha(1-\alpha)\over q}(x^q-1), \\
w_{22}(q)-{w_{12}(q)^2\over1-x}
&={\alpha(1-\alpha)\over q}(x^{1-q}-x).
\end{align*}
Note that \eqref{F-4.7} and \eqref{F-4.8} are equivalent; indeed, dividing both sides of \eqref{F-4.7} by $x$
and then replacing $x$ with $x^{-1}$ transform \eqref{F-4.7} into \eqref{F-4.8}. Letting $x\searrow0$ in
\eqref{F-4.7} gives $-{1\over p}\le-{\alpha(1-\alpha)\over q}$, so that $\alpha(1-\alpha)p\le q$ follows.
\end{proof}

It seems that we can find no further condition other than $\alpha(1-\alpha)p\le q$ from \eqref{F-4.7}
(equivalently, \eqref{F-4.8}) for all $x\ne1$.

\begin{proposition}\label{P-4.11}
Let $0<\alpha<1$ and $p,q>0$. If $\min\{1,p/2\}\le q$, then $\Tr\,R_{\alpha,p}(A,B)\le\Tr\,\cA_{\alpha,q}(A,B)$
holds for all $A,B\ge0$.
\end{proposition}

\begin{proof}
When $p/2\le q$, the asserted trace inequality follows from Proposition \ref{P-4.9}. Assume that $1\le q$
and $p>0$ is arbitrary. For every $A,B\in\bM_n^+$, by Horn's log-majorization (see, e.g., \cite{Bh,Hi}) one has
\begin{align*}
\lambda(R_{\alpha,p}(A,B))&=\lambda^{2/p}\bigl(|A^{(1-\alpha)p/2}B^{\alpha p/2}|\bigr) \\
&\prec_{\log}\Bigl(\bigl(\lambda_i(A^{(1-\alpha)p/2})\lambda_i(B^{\alpha p/2})\bigr)^{2/p}\Bigr)_{i=1}^n
=\bigl(\lambda_i(A)^{1-\alpha}\lambda_i(B)^\alpha\bigr)_{i=1}^n,
\end{align*}
so that
\begin{align*}
\Tr\,R_{\alpha,p}(A,B)&\le\sum_{i=1}^n\lambda_i(A)^{1-\alpha}\lambda_i(B)^\alpha
\le\sum_{i=1}^n\{(1-\alpha)\lambda_i(A)+\alpha\lambda_i(B)\} \\
&=\Tr((1-\alpha)A+\alpha B)=\Tr\,\cA_{\alpha,1}(A,B)\le\Tr\,\cA_{\alpha,q}(A,B),
\end{align*}
where the last inequality is due to Proposition \ref{P-4.3} since $1\le q$.
\end{proof}
 
\begin{theorem}\label{T-4.12}
Let $0<\alpha<1$ and $p,q>0$. If $\Tr\,R_{\alpha,p}(A,B)\le\Tr\,\cA_{\alpha,q}(A,B)$ holds for all
$A,B\in\bM_2^{++}$, then $\min\{1,\alpha(1-\alpha)p\}\le q$.
\end{theorem}

\begin{proof}
It suffices to show that  we must have $\alpha(1-\alpha)p\le q$ when $q<1$. Consider again
$A_0,B_\theta\in\bM_2^{++}$ in \eqref{F-3.5} with $y=x>0$, $x\ne1$. From Lemma \ref{L-3.5} and
\eqref{F-4.5}, \eqref{F-4.6} we compute
\begin{align}
\Tr\,R_{\alpha,p}(A_0,B_\theta)
&=1+x+\theta^2{(1-x)(-1-x^p+x^{\alpha p}+x^{(1-\alpha)p})\over p(1-x^p)}+o(\theta^2), \nonumber\\
\Tr\,\cA_{\alpha,q}(A_0,B_\theta)
&=1+x+\theta^2{\alpha(1-\alpha)\over q}(-1-x+x^q+x^{1-q})+o(\theta^2). \label{F-4.9}
\end{align}
Hence, if $\Tr\,R_{\alpha,p}(A,B)\le\Tr\,\cA_{\alpha,q}(A,B)$ for all $A,B\in\bM_2^{++}$, then we must have
\[
{(1-x)(-1-x^p+x^{\alpha p}+x^{(1-\alpha)p})\over p(1-x^p)}
\le{\alpha(1-\alpha)\over q}(-1-x+x^q+x^{1-q})
\]
for all $x>0$, $x\ne1$. Since $q<1$ by assumption, letting $x\searrow0$ gives $-1/p\le-\alpha(1-\alpha)/q$
and hence $\alpha(1-\alpha)p\le q$ follows.
\end{proof}

In particular, when $\alpha=1/2$ and $p=1$, the next example provides the exact characterization for
$\Tr\,R_{1/2,1}(A,B)\le\Tr\cA_{1/2,q}(A,B)$.

\begin{example}\label{E-4.13}\rm
Let $\alpha=1/2$ and $p=1$. Then the following conditions are equivalent:
\begin{itemize}
\item[(i)] $\Tr\,R_{1/2,1}(A,B)\le\Tr\,\cA_{1/2,q}(A,B)$ for all $A,B\ge0$;
\item[(ii)] $\Tr\,R_{1/2,1}(A,B)\le\Tr\,\cA_{1/2,q}(A,B)$ for all $A,B\in\bM_2^{++}$;
\item[(iii)] $q\ge1/4$.
\end{itemize}
\end{example}

\begin{proof}
(i)$\implies$(ii) is trivial, and (ii)$\implies$(iii) follows from Theorem \ref{T-4.12}. Next let us show that
(iii)$\implies$(i) holds. For this, by Proposition \ref{P-4.3} it suffices to show that
$\Tr\,R_{1/2,1}(A,B)\le\Tr\,\cA_{1/2,1/4}(A,B)$ for all $A,B\ge0$. We note that
$\Tr\,R_{1/2,1}(A,B)=\Tr\,A^{1/2}B^{1/2}$ and
\begin{equation} \label{F-4.10}
\begin{aligned}
&\Tr\,\cA_{1/2,1/4}(A,B)=\Tr\biggl({A^{1/4}+B^{1/4}\over2}\biggr)^4 \\
&\qquad={1\over16}\Tr\bigl(A+B+4A^{1/2}B^{1/2}+4A^{3/4}B^{1/4}+4A^{1/4}B^{3/4}
+2A^{1/4}B^{1/4}A^{1/4}B^{1/4}\bigr).
\end{aligned}
\end{equation}
Since
\begin{align*}
0&\le\Tr(A^{1/4}-B^{1/4})^4 \\
&=\Tr\bigl(A+B+4A^{1/2}B^{1/2}-4A^{3/4}B^{1/4}-4A^{1/4}B^{3/4}+2A^{1/4}B^{1/4}A^{1/4}B^{1/4}\bigr),
\end{align*}
we have
\begin{align}\label{F-4.11}
\Tr\bigl(A+B+4A^{1/2}B^{1/2}+2A^{1/4}B^{1/4}A^{1/4}B^{1/4}\bigr)
\ge4\Tr\bigl(A^{3/4}B^{1/4}+A^{1/4}B^{3/4}\bigr).
\end{align}
It follows from \eqref{F-4.10} and \eqref{F-4.11} that
$\Tr\,\cA_{1/2,1/4}(A,B)\ge{1\over2}\Tr\bigl(A^{3/4}B^{1/4}+A^{1/4}B^{3/4}\bigr)$.
Moreover, by the matrix norm inequality for the Heinz-type means (see \cite{HiKo}) we have
\begin{align*}
\Tr\,A^{1/2}B^{1/2}&=\|A^{1/4}B^{1/4}\|_2^2\le\bigg\|{(A^{1/2})^{3\over4}(B^{1/2})^{1\over4}
+(A^{1/2})^{1\over4}(B^{1/2})^{3\over4}\over2}\bigg\|_2^2 \\
&={1\over4}\Tr(A^{3/8}B^{1/8}+A^{1/8}B^{3/8})^*(A^{3/8}B^{1/8}+A^{1/8}B^{3/8}) \\
&={1\over4}\Tr(B^{1/8}A^{3/8}+B^{3/8}A^{1/8})(A^{3/8}B^{1/8}+A^{1/8}B^{3/8}) \\
&={1\over4}\Tr(2A^{1/2}B^{1/2}+A^{3/4}B^{1/4}+A^{1/4}B^{3/4})
\end{align*}
so that $\Tr\,A^{1/2}B^{1/2}\le{1\over2}\Tr(A^{3/4}B^{1/4}+A^{1/4}B^{3/4})$. Hence (i) follows.
\end{proof}

\begin{proposition}\label{P-4.14}
Let $0<\alpha<1$ and $p,q>0$ be such that $\min\{1,p/2\}\le q$. Then for every $A,B\ge0$ the following
conditions are equivalent:
\begin{itemize}
\item[(i)] $\Tr\,R_{\alpha,p}(A,B)=\Tr\,\cA_{\alpha,q}(A,B)$;
\item[(ii)] $R_{\alpha,p}(A,B)=\cA_{\alpha,q}(A,B)$;
\item[(iii)] $A=B$.
\end{itemize}
\end{proposition}

\begin{proof}
It is obvious that (iii)$\implies$(ii)$\implies$(i). Assume (i). If $p/2<q$, then we have
$\Tr\,\cA_{\alpha,t}(A,B)=\Tr\,\cA_{\alpha,q}(A,B)$ for any $t\in[p/2,q]$ by Propositions 4.3 and 4.11. Hence
(iii) follows from Proposition 4.4. If $p/2=q$, then Proposition \ref{P-4.9} yields
$\lambda(R_{\alpha,p}(A,B))=\lambda(\cA_{\alpha,p/2}(A,B))$ and hence
\[
\Tr\,\big|(A^{p/2})^{1-\alpha}(B^{p/2})^\alpha\big|
=\Tr\,R_{\alpha,p}(A,B)^{p/2}=\Tr\,\cA_{\alpha,p/2}(A,B)^{p/2}
=\Tr\bigl((1-\alpha)A^{p/2}+\alpha B^{p/2}\bigr).
\]
Hence It follows from \cite[Corollary 2.3]{Man} that $A^{p/2}=B^{p/2}$ so that $A=B$. If $q\ge1$, then for
any $t\ge p$ we have $\Tr\,R_{\alpha,p}(A,B)\le\Tr\,R_{\alpha,t}(A,B)\le\Tr\,\cA_{\alpha,q}(A,B)$ by
Araki's log-majorization \cite{Ar} and Proposition \ref{P-4.11}. Hence
$\Tr\,R_{\alpha,t}(A,B)=\Tr\,\cA_{\alpha,q}(A,B)$ for $t\ge p$, which extends to all $t>0$ thanks to real
analyticity of $t\mapsto\Tr\,R_{\alpha,t}(A,B)$ in $t>0$, so we can apply the above case (where $p/2<q$).
\end{proof}

The results of this subsection are summarized as follows:

\medskip
\begin{table}[htb]
\centering
\begin{tabular}{|c|l|l|} \hline
& Sufficient cond. & Necessary cond. \\ \hline

$\begin{array}{ccc}R_{\alpha,p}\le\cA_{\alpha,q} \\ R_{\alpha,p}\le_\chao\cA_{\alpha,q} \\
R_{\alpha,p}\le_\near\cA_{\alpha,q}\end{array}$

& none & \\ \hline

$\begin{array}{cc}R_{\alpha,p}\le_\lambda\cA_{\alpha,q} \\
R_{\alpha,p}\prec_w\cA_{\alpha,q}\end{array}$ & $p/2\le q$ & $\alpha(1-\alpha)p\le q$ \\ \hline

$R_{\alpha,p}\le_\Tr\cA_{\alpha,q}$ & $q\ge1$ or $p/2\le q$ &
$q\ge1$ or $\alpha(1-\alpha)p\le q$\\ \hline

$R_{1/2,1}\le_\Tr\cA_{1/2,q}$ & $q\ge1/4$ & $q\ge1/4$ \\ \hline
\end{tabular}
\end{table}

\newpage
\begin{problem}\label{Q-4.15}\rm
There is a gap between the sufficient condition and the necessary condition for
$R_{\alpha,p}\le_\lambda\cA_{\alpha,q}$, $R_{\alpha,p}\prec_w\cA_{\alpha,q}$ and
$R_{\alpha,p}\le_\Tr\cA_{\alpha,q}$. It is also unknown whether $R_{\alpha,p}\le_\lambda\cA_{\alpha,q}$
is strictly stronger than $R_{\alpha,p}\prec_w\cA_{\alpha,q}$, or they are equivalent. Example \ref{E-4.13}
says that the sufficient condition in Proposition \ref{P-4.11} is not sharp when $\alpha=1/2$ and $p=1$.
This suggests us that the complete characterization of $R_{\alpha,p}\le_\Tr\cA_{\alpha,q}$ is a complicated
problem.
\end{problem}

\subsection{$G_{\alpha,p}$ vs.\ $\cA_{\alpha,q}$}\label{Sec-4.4}

In this subsection we discuss inequalities between $G_{\alpha,p}$ and $\cA_{\alpha,q}$. In theory of
operator means in Kubo and Ando's sense \cite{KA} it is well known that
$A\#_\alpha B\le A\triangledown_\alpha B$ for any $\alpha\in(0,1)$ and $A,B\ge0$. The next theorem
characterizes the Loewner inequality $G_{\alpha,p}\le\cA_{\alpha,q}$, extending
$\#_\alpha\le\triangledown_\alpha$ (for $p=q=1$).

\begin{theorem}\label{T-4.16}
Let $0<\alpha<1$ and $p,q>0$. Then the following conditions are equivalent:
\begin{itemize}
\item[(i)] $G_{\alpha,p}(A,B)\le\cA_{\alpha,q}(A,B)$ for all $A,B\ge0$;
\item[(ii)] $G_{\alpha,p}(A,B)\le\cA_{\alpha,q}(A,B)$ for all $A,B\in\bM_2^{++}$;
\item[(iii)] $1\le p\le q$.
\end{itemize}
\end{theorem}

\begin{lemma}\label{L-4.17}
Let $\alpha>0$ and $p,q>0$. Let $A_0$ and $B_\theta$ be given in \eqref{F-3.5} with $y=x>0$ and $x\ne1$.
Then we have
\begin{align}\label{F-4.12}
G_{\alpha,p}(A_0,B_\theta)=\begin{bmatrix}1+\theta^2z_{11}(p)&\theta z_{12}(p)\\
\theta z_{12}(p)&x+\theta^2z_{22}(p)\end{bmatrix}+o(\theta^2)
\quad\mbox{as $\theta\to0$},
\end{align}
where
\begin{align}\label{F-4.13}
\begin{cases}
z_{11}(p):={\alpha(1-\alpha)\over2p}(x^p-x^{-p})-\alpha^2(1-x), \\
z_{22}(p):={\alpha(1-\alpha)\over2p}(x^{1-p}-x^{1+p})+\alpha^2(1-x), \\
z_{12}(p):=\alpha(1-x).
\end{cases}
\end{align}
\end{lemma}

\begin{proof}
Similarly to the proof of Lemma \ref{L-3.5} it is easy to compute
\[
A_0^{-p/2}B_\theta^pA_0^{-p/2}=\begin{bmatrix}1-\theta^2(1-x^p)&\theta(x^{-p/2}-x^{p/2})\\
\theta(x^{-p/2}-x^{p/2})&1+\theta^2(x^{-p}-1)\end{bmatrix}+o(\theta^2).
\]
Then Example \ref{E-3.2}(1) is applied to compute $\bigl(A_0^{-p/2}B_\theta^pA_0^{-p/2}\bigr)^\alpha$
and hence
\[
A_0^p\#B_\theta^p=A_0^{p/2}\bigl(A_0^{-p/2}B_\theta^pA_0^{-p/2}\bigr)^\alpha A_0^{p/2}
=\begin{bmatrix}1+\theta^2v_{11}(p)&\theta v_{12}(p)\\
\theta v_{12}(p)&x^p+\theta^2v_{22}(p)\end{bmatrix}+o(\theta^2),
\]
where
\[
\begin{cases}
v_{11}(p):=-\alpha^2+{\alpha(\alpha+1)\over2}x^p+{\alpha(\alpha-1)\over2}x^{-p}, \\
v_{22}(p):=x^p\bigl(-\alpha^2+{\alpha(\alpha-1)\over2}x^p+{\alpha(\alpha+1)\over2}x^{-p}\bigr), \\
v_{12}(p):=\alpha(1-x^p).\end{cases}
\]
Furthermore, we apply Example \ref{E-3.2}(1) again to show that \eqref{F-4.12} holds with \eqref{F-4.13},
whose details are left to the reader.
\end{proof}

\begin{proof}[Proof of Theorem \ref{T-4.16}]
(iii)$\implies$(i).\enspace
Since $A^p\#_\alpha B^p\le A^p\triangledown_\alpha B^p$ and $x^{1/p}$ is operator monotone on
$[0,\infty)$, we have
\[
G_{\alpha,p}(A,B)=(A^p\#_\alpha B^p)^{1/p}\le((1-\alpha)A^p+\alpha B^p)^{1/p}
=\cA_{\alpha,p}(A,B).
\]
Since $\cA_{\alpha,p}(A,B)\le\cA_{\alpha,q}(A,B)$ by Theorem \ref{T-4.1}, the result follows.

(i)$\implies$(ii) is trivial.

(ii)$\implies$(iii).\enspace
Assume that (ii) holds. Note that this extends to all $A,B\in\bM_2^+$ by continuity. First, let
$A_0,B_\theta\in\bM_2^{++}$ be given as in Lemma \ref{L-4.17}. Then by \eqref{F-4.12} and \eqref{F-4.13}
as well as \eqref{F-4.5} and \eqref{F-4.6} we have
\begin{align*}
\cA_{\alpha,q}(A_0,B_\theta)-G_{\alpha,p}(A_0,B_\theta)
=\theta^2\alpha(1-\alpha)\begin{bmatrix}{x^q-1\over q}-{x^p-x^{-p}\over2p}&0\\
0&{x^{1-q}-x\over q}-{x^{1-p}-x^{1+p}\over2p}\end{bmatrix}+o(\theta^2).
\end{align*}
Hence we must have
\begin{align}
{x^q-1\over q}-{x^p-x^{-p}\over2p}&\ge0,\qquad x>0, \label{F-4.14}\\
{x^{1-q}-x\over q}-{x^{1-p}-x^{1+p}\over2p}&\ge0,\qquad x>0. \label{F-4.15}
\end{align}
(Note that \eqref{F-4.15} is equivalent to \eqref{F-4.14}.) It is clear that $p\le q$ follows from \eqref{F-4.14}.

Now assume that $p\le q$, and let $A_0$ be given in \eqref{F-3.5} with $x>0$ satisfying
$(1-\alpha)x^q\ne1$, and $\tilde B_\theta\in\bM_2^+$ be in \eqref{F-4.1}. Then by Example \ref{E-3.2}(1)
we compute
\begin{align}\label{F-4.16}
\cA_{\alpha,q}(A_0,\tilde B_\theta)
=\begin{bmatrix}1+\theta^2w_{11}&\theta w_{12}\\\theta w_{12}&(1-\alpha)^{1/q}x+\theta^2w_{22}
\end{bmatrix}+o(\theta^2)\quad\mbox{as $\theta\to0$},
\end{align}
where
\begin{align}\label{F-4.17}
\begin{cases}w_{11}=-{\alpha(1-\alpha)(1-x^q)\over q(1-(1-\alpha)x^q)}
-{\alpha^2(1-(1-\alpha)^{1/q}x)\over(1-(1-\alpha)x^q)^2}, \\
w_{22}={\alpha(1-\alpha)^{1/q}x^{1-q}(1-x^q)\over q(1-(1-\alpha)x^q)}
+{\alpha^2(1-(1-\alpha)^{1/q}x)\over(1-(1-\alpha)x^q)^2}, \\
w_{12}={\alpha(1-(1-\alpha)^{1/q}x)\over1-(1-\alpha)x^q}.
\end{cases}
\end{align}
On the other hand, since $\tilde B_\theta^p=\tilde B_\theta=|\phi\>\<\phi|$ where
$\phi:=\begin{bmatrix}\cos\theta\\\sin\theta\end{bmatrix}$, one has
\[
\bigl(A_0^{-p/2}\tilde B_\theta^pA_0^{-p/2}\bigr)^\alpha
=\|A_0^{-p/2}\phi\|^{2(\alpha-1)}|A_0^{-p/2}\phi\>\<A_0^{-p/2}\phi|
\]
so that $G_\alpha(A_0^p,\tilde B_\theta^p)=\|A_0^{-p/2}\phi\|^{2(\alpha-1)}|\phi\>\<\phi|$. Hence one writes
\begin{equation}\label{F-4.18}
\begin{aligned}
G_{\alpha,p}(A_0,\tilde B_\theta)&=\|A_0^{-p/2}\phi\|^{2(\alpha-1)\over p}|\phi\>\<\phi| \\
&=(\cos^2\theta+x^{-p}\sin^2\theta)^{\alpha-1\over p}
\begin{bmatrix}\cos^2\theta&\cos\theta\sin\theta\\\cos\theta\sin\theta&\sin^2\theta\end{bmatrix} \\
&=\{1+(x^{-p}-1)\theta^2\}^{\alpha-1\over p}
\begin{bmatrix}1-\theta^2&\theta\\\theta&\theta^2\end{bmatrix}+o(\theta^2) \\
&=\biggl\{1+{\alpha-1\over p}(x^{-p}-1)\theta^2\biggr\}
\begin{bmatrix}1-\theta^2&\theta\\\theta&\theta^2\end{bmatrix}+o(\theta^2) \\
&=\begin{bmatrix}1+\bigl\{{\alpha-1\over p}(x^{-p}-1)-1\bigr\}\theta^2&\theta\\
\theta&\theta^2\end{bmatrix}+o(\theta^2).
\end{aligned}
\end{equation}
By \eqref{F-4.16} and \eqref{F-4.18} we write
\[
\cA_{\alpha,q}(A_0,\tilde B_\theta)-G_{\alpha,p}(A_0,\tilde B_\theta)
=\begin{bmatrix}\theta^2\bigl\{w_{11}-{\alpha-1\over p}(x^{-p}-1)+1\bigr\}&\theta(w_{12}-1)\\
\theta(w_{12}-1)&(1-\alpha)^{1/q}x+\theta^2(w_{22}-1)\end{bmatrix}+o(\theta^2).
\]
Therefore, we must have
\begin{align}
0&\le w_{11}-{\alpha-1\over p}(x^{-p}-1)+1, \label{F-4.19}\\
0&\le(1-\alpha)^{1/q}x\biggl\{w_{11}-{\alpha-1\over p}(x^{-p}-1)+1\biggr\}-(w_{12}-1)^2. \label{F-4.20}
\end{align}

We estimate \eqref{F-4.19} and \eqref{F-4.20} as $x\searrow0$ as follows:
\begin{align*}
\mbox{\eqref{F-4.19}}
&\approx-{\alpha(1-\alpha)\over q}-\alpha^2-{\alpha-1\over p}(x^{-p}-1)+1
\approx{1-\alpha\over p}\,x^{-p}>0. \\
\mbox{\eqref{F-4.20}}
&\approx(1-\alpha)^{1/q}x\biggl\{-{\alpha(1-\alpha)\over q}-\alpha^2-{\alpha-1\over p}(x^{-p}-1)+1\biggr\}
-(\alpha-1)^2 \\
&\approx{(1-\alpha)^{{1\over q}+1}\over p}\,x^{1-p}-(1-\alpha)^2.
\end{align*}
If $p<1$, then \eqref{F-4.20}\,$\to-(1-\alpha)^2<0$ as $x\searrow0$, which is impossible since
\eqref{F-4.20}\,$\ge0$ in the limit $x\searrow0$. Hence $p\ge1$ must hold.
\end{proof}

The next corollary is a particular case of Theorem \ref{T-4.16} when $q=1$.

\begin{corollary}\label{C-4.18}
Let $0<\alpha<1$ and $p>0$. Then the following conditions are equivalent:
\begin{itemize}
\item[(i)] $G_{\alpha,p}(A,B)\le\alpha A+(1-\alpha)B$ for all $A,B\ge0$;
\item[(ii)] $G_{\alpha,p}(A,B)\le\alpha A+(1-\alpha)B$ for all $A,B\in\bM_2^{++}$;
\item[(iii)] $p=1$.
\end{itemize}
\end{corollary}

\begin{proposition}\label{P-4.19}
Let $0<\alpha<1$. If $0<p\le q$, then we have $G_{\alpha,p}(A,B)\le_\chao\cA_{\alpha,q}(A,B)$ for all
$A,B\ge0$.
\end{proposition}

\begin{proof}
Assume that $0<p\le q$. Since $1\le q/p$, it follows from Theorem \ref{T-4.16} that
$G_{\alpha,1}(A^p,B^p)\le\cA_{\alpha,q/p}(A^p,B^p)$, i.e.,
\[
A^p\#_\alpha B^p\le\bigl((1-\alpha)A^q+\alpha B^q\bigr)^{p/q},
\]
which implies that
\[
P:=s(G_{\alpha,p}(A,B))=s(A^p\#_\alpha B^p)\le s((1-\alpha)A^q+\alpha B^q)=s(\cA_{\alpha,q}(A,B))
\]
and
\[
P\{\log(A^p\#_\alpha B^p)\}P\le P\{\log((1-\alpha)A^q+\alpha B^q)^{p/q}\}P
=pP\{\log\cA_{\alpha,q}(A,B)\}P
\]
so that
\[
P\{\log G_{\alpha,p}(A,B)\}P\le P\{\log\cA_{\alpha,q}(A,B)\}P,
\]
that is, $G_{\alpha,p}(A,B)\le_\chao\cA_{\alpha,q}(A,B)$.
\end{proof}

\begin{theorem}\label{T-4.20}
Let $0<\alpha<1$ and $p,q>0$. If $G_{\alpha,p}(A,B))\le_\lambda\cA_{\alpha,q}(A,B)$ holds for all
$A,B\in\bM_2^{++}$, then $p\le q$.
\end{theorem}

\begin{proof}
We use $A_0,B_\theta\in\bM_2^{++}$ in \eqref{F-3.5} with $y=x\ne1$ once again, and argue in the same
way as in the proof of Theorem \ref{T-4.10}. Then by assumption we must have the same inequalities as
in \eqref{F-4.7} and \eqref{F-4.8} with $z_{11}(p),z_{22}(p),z_{12}(p)$ in \eqref{F-4.13} instead of
\eqref{F-3.16}. These are specified in the present case as
\begin{align*}
{\alpha(1-\alpha)\over2p}(x^p-x^{-p})&\le{\alpha(1-\alpha)\over q}(x^q-1), \\
{\alpha(1-\alpha)\over2p}(x^{1-p}-x^{1+p})&\le{\alpha(1-\alpha)\over q}(x^{1-q}-x)
\end{align*}
for all $x>0$, $x\ne1$. Hence we have $p\le q$.
\end{proof}

\begin{proposition}\label{P-4.21}
Let $0<\alpha<1$ and $p,q>0$ be arbitrary. Then for every $A,B\ge0$ we have
$G_{\alpha,p}(A,B)\prec_{w\log}\cA_{\alpha,q}(A,B)$ and hence
$G_{\alpha,p}(A,B)\prec_w\cA_{\alpha,q}(A,B)$.
\end{proposition}

\begin{proof}
When $A,B>0$, the result follows from the log-majorization $G_{\alpha,p}(A,B)\prec_{\log}LE_\alpha(A,B)$
(see \cite[Corollary 2.3]{AH}) and Proposition \ref{P-4.6}. For general $A,B\ge0$ one can take the limit from
the result for $A+\eps I$ and $B+\eps I$ as $\eps\searrow0$.
\end{proof}

\begin{proposition}\label{P-4.22}
Let $0<\alpha<1$ and $p,q>0$. Then for every $A,B\ge0$ the following conditions are equivalent:
\begin{itemize}
\item[(i)] $\Tr\,G_{\alpha,p}(A,B)=\Tr\,\cA_{\alpha,q}(A,B)$;
\item[(ii)] $G_{\alpha,p}(A,B)=\cA_{\alpha,q}(A,B)$;
\item[(iii)] $A=B$.
\end{itemize}
\end{proposition}

\begin{proof}
It suffices to show (i)$\implies$(iii). If (i) holds, then $\Tr\,\cA_{\alpha,t}(A,B)=\Tr\,\cA_{\alpha,q}(A,B)$
for any $t\in(0,q]$ by Propositions \ref{P-4.3} and \ref{P-4.21}. Hence we have (iii) by Proposition \ref{P-4.4}.
\end{proof}

The results of this subsection are summarized as follows:

\medskip
\begin{table}[htb]
\centering
\begin{tabular}{|c|l|l|} \hline
& Sufficient cond. & Necessary cond. \\ \hline

$G_{\alpha,p}\le\cA_{\alpha,q}$ & $1\le p\le q$ & $1\le p\le q$ \\ \hline

$\begin{array}{ccc}G_{\alpha,p}\le_\chao\cA_{\alpha,q} \\ G_{\alpha,p}\le_\near\cA_{\alpha,q} \\
G_{\alpha,p}\le_\lambda\cA_{\alpha,q}\end{array}$ & $p\le q$ & $p\le q$ \\ \hline

$\begin{array}{cc}G_{\alpha,p}\prec_w\cA_{\alpha,q} \\
G_{\alpha,p}\le_\Tr\cA_{\alpha,q}\end{array}$ & arbitrary $p,q$ & \\ \hline
\end{tabular}
\end{table}

\subsection{$SG_{\alpha,p}$ vs.\ $\cA_{\alpha,q}$}\label{Sec-4.5}

In this subsection we discuss inequalities between $SG_{\alpha,p}$ and $\cA_{\alpha,q}$. The next theorem
says that $SG_{\alpha,p}(A,B)\le_\chao\cA_{\alpha,q}(A,B)$ fails to hold for any $p,q>0$.

\begin{theorem}\label{T-4.23}
For any $\alpha\in(0,1)$ and any $p,q>0$ there exist $A,B\in\bM_2^{++}$ such that
$SG_{\alpha,p}(A,B)\allowbreak\not\le_\chao\cA_{\alpha,q}(A,B)$ (hence
$SG_{\alpha,p}(A,B)\not\le\cA_{\alpha,q}(A,B)$).
\end{theorem}

\begin{proof}
For $A,B>0$ set $Y:=A^p$ and $X:=A^{-p}\#B^p$; then $X=Y^{-1}\#B^p$ and hence $B^p=XYX$ by the
Riccati lemma, so that $B=(XYX)^{1/p}$. Therefore, $SG_{\alpha,p}(A,B)\le_\chao\cA_{\alpha,q}(A,B)$ is
equivalently written as
\[
{q\over p}\log(X^\alpha YX^\alpha)\le\log\{(1-\alpha)Y^{q/p}+\alpha(XYX)^{q/p}\}.
\]
Moreover, for any $X,Y>0$, if we set $A:=Y^{1/p}$ and $B:=(XYX)^{1/p}$, then $Y=A^p$ and
$X=A^{-p}\#B^p$. Hence, by replacing $q/p$ with $r$, it suffices to show that for any $r>0$ there exist
$X,Y\in\bM_2^{++}$ such that
\begin{align}\label{F-4.21}
r\log(X^\alpha YX^\alpha)\not\le\log\{(1-\alpha)Y^r+\alpha(XYX)^r\}.
\end{align}

To do this, let $X:=A_0$ and $Y:=B_\theta$ in \eqref{F-3.5} for $x,y>0$ with $x^2y\ne1$. Since
\[
XYX=\begin{bmatrix}1-\theta^2(1-y)&\theta x(1-y)\\\theta x(1-y)&x^2y+\theta^2x^2(1-y)\end{bmatrix}
+o(\theta^2)\quad\mbox{as $\theta\to0$},
\]
one can compute $(XYX)^r$ by Example \ref{E-3.2}(1) as
\[
(XYX)^r=\begin{bmatrix}1+\theta^2a_{11}&\theta a_{12}\\
\theta a_{12}&x^{2r}y^r+\theta^2 a_{22}\end{bmatrix}+o(\theta^2),
\]
where
\begin{align}\label{F-4.22}
\begin{cases}
a_{11}:=-r(1-y)+{x^2(1-y)^2(r-1-rx^2y+x^{2r}y^r)\over(1-x^2y)^2}, \\
a_{12}:={x(1-y)(1-x^{2r}y^r)\over1-x^2y}.
\end{cases}
\end{align}
Hence one can write
\begin{align}\label{F-4.23}
(1-\alpha)Y^r+\alpha(XYX)^r=\begin{bmatrix}1+\theta^2 b_{11}&\theta b_{12}\\
\theta b_{12}&b+\theta^2b_{22}\end{bmatrix}+o(\theta^2),
\end{align}
where
\begin{align}\label{F-4.24}
\begin{cases}
\ \,b\ :=y^r(\alpha x^{2r}+1-\alpha), \\
b_{11}:=\alpha a_{11}-(1-\alpha)(1-y^r), \\
b_{12}:=\alpha a_{12}+(1-\alpha)(1-y^r).
\end{cases}
\end{align}
From this, by Example \ref{E-3.2}(3) one can write
\begin{align}\label{F-4.25}
\log\{(1-\alpha)Y^r+\alpha(XYX)^r\}=\begin{bmatrix}\theta^2c_{11}&\theta c_{12}\\
\theta c_{12}&\log b+\theta^2c_{22}\end{bmatrix}+o(\theta^2),
\end{align}
where
\begin{align}\label{F-4.26}
c_{11}:=b_{11}+{1-b+\log b\over(1-b)^2}\,b_{12}^2.
\end{align}

On the other hand, for $x,y>0$ with $x^{2\alpha}y\ne1$, since
\begin{align}\label{F-4.27}
X^\alpha YX^\alpha=\begin{bmatrix}1-\theta^2(1-y)&\theta x^\alpha(1-y)\\
\theta x^\alpha(1-y)&x^{2\alpha}y+\theta^2x^{2\alpha}(1-y)\end{bmatrix}+o(\theta^2),
\end{align}
one can write by Example \ref{E-3.2}(3),
\begin{align}\label{F-4.28}
\log(X^\alpha YX^\alpha)=\begin{bmatrix}\theta^2 d_{11}&\theta d_{12}\\
\theta d_{12}&2\alpha\log x+\log y+\theta^2d_{22}\end{bmatrix}+o(\theta^2),
\end{align}
where
\begin{align}\label{F-4.29}
d_{11}:=y-1+{x^{2\alpha}(1-y)^2(1-x^{2\alpha}y+2\alpha\log x+\log y)\over(1-x^{2\alpha}y)^2}.
\end{align}
From \eqref{F-4.25}, \eqref{F-4.28} and \eqref{F-4.24} one can write
\begin{align*}
&\det\bigl[\log\{(1-\alpha)Y^r+\alpha(XYX)^r\}-r\log(X^\alpha YX^\alpha)\bigr] \\
&\quad=\det\biggl(\begin{bmatrix}\theta^2(c_{11}-rd_{11})&\theta(c_{12}-rd_{12})\\
\theta(c_{12}-rd_{12})&\log b-r(2\alpha\log x+\log y)+\theta^2(c_{22}-rd_{22})\end{bmatrix}
+o(\theta^2)\biggr) \\
&\quad=\theta^2\bigl[(c_{11}-rd_{11})\{\log(\alpha x^{2r}+1-\alpha)-2r\alpha\log x\}
-(c_{12}-rd_{12})^2\bigr]+o(\theta^2).
\end{align*}
Now assume that $r\log(X^\alpha YX^\alpha)\le\log\{(1-\alpha)Y^r+\alpha(XYX)^r\}$ for all
$X,Y\in\bM_2^{++}$. Then we must have
\begin{align}\label{F-4.30}
(c_{11}-rd_{11})\{\log(\alpha x^{2r}+1-\alpha)-2r\alpha\log x\}\ge(c_{12}-rd_{12})^2\ge0
\end{align}
for all $x,y>0$ with $x^2y\ne1$ and $x^{2\alpha}y\ne1$. Let $x>0$ with $x\ne1$ be fixed, so that $x^2y<1$
and $x^{2\alpha}y<1$ for sufficiently small $y>0$. Since $(x^{2r})^\alpha<\alpha x^{2r}+1-\alpha$ so that
$\log(\alpha x^{2r}+1-\alpha)-2r\alpha\log x>0$, it follows from \eqref{F-4.30} that $c_{11}-rd_{11}\ge0$.
Let us estimate $c_{11}$ and $d_{11}$ when $y\searrow0$. As $y\searrow0$ we have $b\to0$ and
$\log b\approx r\log y$ by \eqref{F-4.24}. Since $a_{11}\to-r+(r-1)x^2$ and $a_{12}\to x$ by \eqref{F-4.22},
we have by \eqref{F-4.24}
\begin{align}\label{F-4.31}
b_{11}\to\alpha(r-1)(x^2-1)-1,\qquad b_{12}\to\alpha x+1-\alpha,
\end{align}
so that by \eqref{F-4.26},
\[
c_{11}\approx\alpha(r-1)(x^2-1)-1+r(\alpha x+1-\alpha)^2\log y
\approx r(\alpha x+1-\alpha)^2\log y.
\]
Moreover, by \eqref{F-4.29},
\[
d_{11}\approx-1+x^{2\alpha}\log y\approx x^{2\alpha}\log y.
\]
Hence it follows that
\[
0\le c_{11}-rd_{11}\approx r(\alpha x+1-\alpha)^2\log y-rx^{2\alpha}\log y
=r\{(\alpha x+1-\alpha)^2-x^{2\alpha}\}\log y,
\]
which is impossible since $(\alpha x+1-\alpha)^2-x^{2\alpha}>0$ and $\log y<0$ as $y\searrow0$. This implies
that \eqref{F-4.21} holds for some $X,Y\in\bM_2^{++}$.
\end{proof}

When $p=q=1$, Theorem \ref{T-4.23} says that for any $\alpha\in(0,1)$, the Loewner inequality
$F_\alpha(A,B)\le A\triangledown_\alpha B$ fails to hold. In particular, $F(A,B)\le{A+B\over2}$ fails
to hold.

\begin{theorem}\label{T-4.24}
Let $0<\alpha<1$ and $p,q>0$. If $SG_{\alpha,p}(A,B)\le_\near\cA_{\alpha,q}(A,B)$ for all $A,B\in\bM_2^{++}$,
then $q/p\ge1+\max\{\alpha(q-1),(1-\alpha)(q-1)\}$; hence
\[
{q\over p}\ge\begin{cases}
1+\max\{\alpha,1-\alpha\}(q-1)\ge1 & \text{if $q\ge1$}, \\
1-\min\{\alpha,1-\alpha\}(1-q)\ge1-\min\{\alpha,1-\alpha\}=\max\{\alpha,1-\alpha\} & \text{if $0<q\le1$}.
\end{cases}
\]
\end{theorem}

\begin{proof}
As in the first paragraph of the proof of Theorem \ref{T-4.23} we see that
$SG_{\alpha,p}(A,B)\le_\near\cA_{\alpha,q}(A,B)$ for all $A,B\in\bM_2^{++}$ if and only if
\begin{align}\label{F-4.32}
(X^\alpha YX^\alpha)^{r/q}\le_\near\{(1-\alpha)Y^r+\alpha(XYX)^r\}^{1/q},\qquad X,Y\in\bM_2^{++},
\end{align}
where $r:=q/p$. Let $X:=A_0$ and $Y:=B_\theta$ in \eqref{F-3.5} for $x,y>0$ with $x^2y\ne1$,
$x^{2\alpha}y\ne1$ and $b:=y^r(\alpha x^{2r}+1-\alpha)\ne1$. By \eqref{F-4.23} with \eqref{F-4.24} one has
by Example \ref{E-3.2}(1),
\[
M:=\{(1-\alpha)Y^r+\alpha(XYX)^r\}^{1/q}=\begin{bmatrix}1+\theta^2u_{11}&\theta u_{12}\\
\theta u_{12}&b^{1/q}+\theta^2u_{22}\end{bmatrix}+o(\theta^2),
\]
where
\begin{align}\label{F-4.33}
u_{11}:={b_{11}\over q}+{1-b-q\bigl(1-b^{1\over q}\bigr)\over q(1-b)^2}\,b_{12}^2,\qquad
u_{12}:={1-b^{1\over q}\over1-b}\,b_{12}
\end{align}
with $b_{11},b_{12}$ in \eqref{F-4.24}. Moreover,
\[
M^{1/2}=\begin{bmatrix}1+\theta^2v_{11}&\theta v_{12}\\
\theta v_{12}&b^{1\over2q}+\theta^2v_{22}\end{bmatrix}+o(\theta^2),
\]
where
\begin{align}\label{F-4.34}
v_{11}:={1\over2}\Biggl\{u_{11}-{u_{12}^2\over\bigl(1+b^{1\over2q}\bigr)^2}\Biggr\},\qquad
v_{12}:={u_{12}\over1+b^{1\over2q}}.
\end{align}
On the other hand, by \eqref{F-4.27} one has by Example \ref{E-3.2}(1) again,
\begin{align}\label{F-4.35}
L:=(X^\alpha YX^\alpha)^{r/q}=\begin{bmatrix}1+\theta^2s_{11}&\theta s_{12}\\
\theta s_{12}&x^{2\alpha r\over q}y^{r\over q}+\theta^2s_{22}\end{bmatrix}+o(\theta^2),
\end{align}
where
\begin{align}\label{F-4.36}
\begin{cases}
s_{11}:=-{r\over q}(1-y)+{x^{2\alpha}(1-y)^2\bigl({r\over q}-1-{r\over q}x^{2\alpha}y
+x^{2\alpha r\over q}y^{r\over q}\bigr)\over(1-x^{2\alpha}y)^2}, \\
s_{12}:={x^\alpha(1-y)\bigl(1-x^{2\alpha r\over q}y^{r\over q}\bigr)\over1-x^{2\alpha}y}.
\end{cases}
\end{align}
Moreover,
\[
M^{1/2}LM^{1/2}=\begin{bmatrix}1+\theta^2t_{11}&\theta t_{12}\\
\theta t_{12}&a+\theta^2t_{22}\end{bmatrix}+o(\theta^2),
\]
where
\begin{align}\label{F-4.37}
\begin{cases}
a:=x^{2\alpha r\over q}y^{r\over q}b^{1\over q}, \\
t_{11}:=s_{11}+2v_{11}+2s_{12}v_{12}+x^{2\alpha r\over q}y^{r\over q}v_{12}^2, \\
t_{12}:=b^{1\over2q}s_{12}+\bigl(x^{2\alpha r\over q}y^{r\over q}b^{1\over2q}+1\bigr)v_{12}.
\end{cases}
\end{align}
Hence one has
\[
(M^{1/2}LM^{1/2})^{1/2}=\begin{bmatrix}1+\theta^2w_{11}&\theta w_{12}\\
\theta w_{12}&a^{1/2}+\theta^2w_{22}\end{bmatrix}+o(\theta^2),
\]
where
\begin{align}\label{F-4.38}
w_{11}:={1\over2}\Bigl\{t_{11}-{t_{12}^2\over(1+a^{1/2})^2}\Bigr\},\qquad
w_{12}:={t_{12}\over1+a^{1/2}}.
\end{align}

Now assume \eqref{F-4.32}, which means that $(M^{1/2}LM^{1/2})^{1/2}\le M$ so that
\begin{align*}
0\le\det\{M-(M^{1/2}LM^{1/2})^{1/2}\}
&=\det\Biggl(\begin{bmatrix}\theta^2(u_{11}-w_{11})&\theta(u_{12}-w_{12})\\
\theta(u_{12}-w_{12})&b^{1\over q}-a^{1\over2}+\theta^2(u_{22}-w_{22})\end{bmatrix}+o(\theta^2)\Biggr) \\
&=\theta^2\bigl\{\bigl(b^{1\over q}-a^{1\over2}\bigr)(u_{11}-w_{11})-(u_{12}-w_{12})^2\bigr\}+o(\theta^2).
\end{align*}
Therefore,
\begin{align}\label{F-4.39}
\bigl(b^{1\over q}-a^{1\over2}\bigr)(u_{11}-w_{11})-(u_{12}-w_{12})^2\ge0
\end{align}
for all $x,y>0$ with $x^2y\ne1$, $x^{2\alpha}y\ne1$ and $b\ne1$. By \eqref{F-4.24} and \eqref{F-4.37}
note that
\begin{equation}\label{F-4.40}
\begin{aligned}
b^{1\over q}-a^{1\over2}&=y^{r\over q}(\alpha x^{2r}+1-\alpha)^{1\over q}
-x^{\alpha r\over q}y^{r\over q}(\alpha x^{2r}+1-\alpha)^{1\over2q} \\
&=y^{r\over q}(\alpha x^{2r}+1-\alpha)^{1\over2q}\{(\alpha x^{2r}+1-\alpha)^{1\over2q}-x^{\alpha r\over q}\}.
\end{aligned}
\end{equation}
Furthermore, when $y\searrow0$ with $x>0$ fixed, we see that
\begin{align*}
u_{12}-w_{12}&=u_{12}-t_{12}+O(y^{r\over q})
\quad\mbox{$\Bigl($by \eqref{F-4.38} and $a^{1/2}=O(y^{r\over2q}b^{1\over2q})=O(y^{r\over q})$$\Bigr)$} \\
&=u_{12}-b^{1\over2q}s_{12}-v_{12}+O(y^{r\over q})\quad\mbox{(by \eqref{F-4.37})} \\
&=u_{12}-b^{1\over2q}{1-y\over1-x^{2\alpha}y}\,x^\alpha-{u_{12}\over1+b^{1\over2q}}+O(y^{r\over q})
\quad\mbox{(by \eqref{F-4.36} and \eqref{F-4.34})} \\
&={b^{1\over2q}\over1+b^{1\over2q}}\,u_{12}-b^{1\over2q}{1-y\over1-x^{2\alpha}y}\,x^\alpha+O(y^{r\over q}) \\
&=b^{1\over2q}(u_{12}-x^\alpha)+O(y^{{r\over2q}+1})+O(y^{r\over q}) \\
&\quad\mbox{$\Bigl($since ${b^{1\over2q}\over1+b^{1\over2q}}-b^{1\over2q}
={-b^{1\over q}\over1+b^{1\over2q}}=O(y^{r\over q})$ and
$b^{1\over2q}{1-y\over1-x^{2\alpha}y}-b^{1\over2q}
=b^{1\over2q}y\,{x^{2\alpha}-1\over1-x^{2\alpha}y}=O(y^{{r\over2q}+1})$$\Bigr)$} \\
&=y^{r\over2q}(\alpha x^{2r}+1-\alpha)^{1\over2q}(u_{12}-x^\alpha)+O(y^{{r\over2q}+1})+O(y^{r\over q}),
\end{align*}
which implies that
\begin{align}\label{F-4.41}
(u_{12}-w_{12})^2=y^{r\over q}(\alpha x^{2r}+1-\alpha)^{1\over q}(u_{12}-x^\alpha)^2+o(y^{r\over q}).
\end{align}
From \eqref{F-4.39}--\eqref{F-4.41} it follows that when $y\searrow0$,
\[
(\alpha x^{2r}+1-\alpha)^{1\over2q}\bigl\{(\alpha x^{2r}+1-\alpha)^{1\over2q}-x^{\alpha r\over q}\bigr\}
(u_{11}-w_{11})-(\alpha x^{2r}+1-\alpha)^{1\over q}(u_{12}-x^\alpha)^2\ge o(1)
\]
so that
\begin{align}\label{F-4.42}
\bigl\{(\alpha x^{2r}+1-\alpha)^{1\over2q}-x^{\alpha r\over q}\bigr\}(u_{11}-w_{11})
-(\alpha x^{2r}+1-\alpha)^{1\over2q}(u_{12}-x^\alpha)^2\ge o(1).
\end{align}
As $y\searrow0$, since $b\to0$, from \eqref{F-4.33}, \eqref{F-4.31} and \eqref{F-4.34} we see that
\begin{align*}
u_{11}&\to{\alpha(r-1)\over q}(x^2-1)-{1\over q}+{1-q\over q}(\alpha x+1-\alpha)^2,\qquad
u_{12}\to\alpha x+1-\alpha, \\
v_{11}&\to{\alpha(r-1)\over2q}(x^2-1)-{1\over2q}+{1-2q\over2q}(\alpha x+1-\alpha)^2,\qquad
v_{12}\to\alpha x+1-\alpha.
\end{align*}
Moreover, from \eqref{F-4.36} and \eqref{F-4.37},
\begin{align*}
s_{11}&\to-{r\over q}+\biggl({r\over q}-1\biggr)x^{2\alpha},\qquad s_{12}\to x^\alpha,\qquad a\to0, \\
t_{11}&\to-{r+1\over q}+\biggl({r\over q}-1\biggr)x^{2\alpha}+{\alpha(r-1)\over q}(x^2-1)
+{1-2q\over q}(\alpha x+1-\alpha)^2+2x^\alpha(\alpha x+1-\alpha), \\
t_{12}&\to\alpha x+1-\alpha.
\end{align*}
Finally, from \eqref{F-4.38},
\begin{align*}
w_{11}&\to-{r+1\over2q}+{1\over2}\biggl({r\over q}-1\biggr)x^{2\alpha}
+{\alpha(r-1)\over2q}(x^2-1)+{1-3q\over2q}(\alpha x+1-\alpha)^2+x^\alpha(\alpha x+1-\alpha).
\end{align*}
Therefore, letting $y\searrow0$ in \eqref{F-4.42} gives
\begin{equation}\label{F-4.43}
\begin{aligned}
&\bigl\{(\alpha x^{2r}+1-\alpha)^{1\over2q}-x^{\alpha r\over q}\bigr\}
\biggl\{{\alpha(r-1)\over2q}(x^2-1)+{r-1\over2q}-{1\over2}\biggl({r\over q}-1\biggr)x^{2\alpha} \\
&+{q+1\over2q}(\alpha x+1-\alpha)^2-x^\alpha(\alpha x+1-\alpha)\biggr\}
-(\alpha x^{2r}+1-\alpha)^{1\over2q}(\alpha x+1-\alpha-x^\alpha)^2\ge0
\end{aligned}
\end{equation}
for all $x>0$. Letting $x\searrow0$ in \eqref{F-4.43} gives
\[
(1-\alpha)^{1\over2q}\biggl\{-{\alpha(r-1)\over2q}+{r-1\over2q}+{q+1\over2q}(1-\alpha)^2\biggr\}
-(1-\alpha)^{1\over2q}(1-\alpha)^2\ge0
\]
so that $(1-\alpha)(r-1)-(1-\alpha)^2(q-1)\ge0$, implying that $r\ge1+(1-\alpha)(q-1)$. Examining the
coefficient of the maximal order term $x^{{r\over q}+2}$ of \eqref{F-4.43} we also find that
\[
\alpha^{1\over2q}\biggl\{{\alpha(r-1)\over2q}+{\alpha^2(q+1)\over2q}\biggr\}-\alpha^{1\over2q}\alpha^2\ge0
\]
so that $\alpha(r-1)+\alpha^2(q+1)-2q\alpha^2\ge0$, implying that $r\ge1+\alpha(q-1)$. This is immediate
since the assumption of the theorem holds for $1-\alpha$ in place of $\alpha$ from symmetry of
$SG_{\alpha,p}$ and $\cA_{\alpha,q}$ (see Remark \ref{R-2.1}(2)). Thus the result follows.
\end{proof}

\begin{proposition}\label{P-4.25}
Let $0<\alpha<1$ and $p,q>0$.
\begin{itemize}
\item[(1)] If ${p/q}\le\min\{\alpha,1-\alpha\}$, then we have
$SG_{\alpha,p}(A,B)\prec_{\log}R_{\alpha,q}(A,B)$ for every $A,B\ge0$ with $s(A)\ge s(B)$.
\item[(2)] If ${p/q}\le2\min\{\alpha,1-\alpha\}$, then we have
$SG_{\alpha,p}(A,B)\prec_{w\log}\cA_{\alpha,q}(A,B)$ (hence
$SG_{\alpha,p}(A,B)\prec_w\cA_{\alpha,q}(A,B)$) for every $A,B\ge0$ with $s(A)\ge s(B)$.
\end{itemize}
\end{proposition}

\begin{proof}
By continuity (Proposition \ref{P-2.2}) and symmetry mentioned in Remark \ref{R-2.1}(2) we may assume
that $A,B>0$ and $0<\alpha\le1/2$.

(1)\enspace
For $0<\alpha\le1/2$ and $A,B>0$, by Araki's log-majorization \cite{Ar} we have
\begin{align*}
\lambda((A^{-1}\#B)^\alpha A(A^{-1}\#B)^\alpha)
&\prec_{\log}\lambda^{2\alpha}\bigl((A^{-1}\#B)^{1\over2}A^{1\over2\alpha}(A^{-1}\#B)^{1\over2}\bigr) \\
&=\lambda^{2\alpha}\bigl(A^{1\over4\alpha}(A^{-1}\#B)A^{1\over4\alpha}\bigr) \\
&=\lambda^{2\alpha}\bigl(A^{1-2\alpha\over4\alpha}A^{1/2}(A^{-1}\#B)A^{1/2}
A^{1-2\alpha\over4\alpha}\bigr) \\
&=\lambda^{2\alpha}\bigl(A^{1-2\alpha\over4\alpha}(A^{1/2}BA^{1/2})^{1/2}
A^{1-2\alpha\over4\alpha}\bigr) \\
&\prec_{\log}\lambda^\alpha\bigl(A^{1-2\alpha\over2\alpha}A^{1/2}BA^{1/2}
A^{1-2\alpha\over2\alpha}\bigr) \\
&=\lambda^\alpha\bigl(A^{1-\alpha\over2\alpha}BA^{1-\alpha\over2\alpha}\bigr)
=\lambda(R_{\alpha,1/\alpha}(A,B)).
\end{align*}
Replacing $A,B$ with $A^p,B^p$ gives
$SG_{\alpha,p}(A,B)\prec_{\log}R_{\alpha,1/\alpha}(A^p,B^p)^{1/p}=R_{\alpha,p/\alpha}(A,B)$. The result
follows since $R_{\alpha,p/\alpha}(A,B)\prec_{\log}R_{\alpha,q}(A,B)$ if $p/\alpha\le q$ by Araki's
log-majorization again.

(2)\enspace
If $p/q\le2\alpha$ and so $p/\alpha\le2q$, then we have
$R_{\alpha,p/\alpha}(A,B)\le_\lambda\cA_{\alpha,q}(A,B)$ by Proposition \ref{P-4.9}. Combining this with (1)
gives the result.
\end{proof}

\begin{remark}\label{R-4.26}\rm
As for log-majorization in (1) above, for $0<\alpha<1$ and $p,q>0$, it is indeed known that
$SG_{\alpha,p}(A,B)\prec_{\log}R_{\alpha,q}(A,B)$ for all $A,B>0$ if and only if $p/q\le\min\{\alpha,1-\alpha\}$,
and that $R_{\alpha,q}(A,B)\prec_{\log}SG_{\alpha,p}(A,B)$ for all $A,B>0$ if and only if
$p/q\ge\max\{\alpha,1-\alpha\}$. The details on these facts will be provided in \cite{Hi5}, while the `if' part
of the latter was indeed shown in \cite{GT}.
\end{remark}

\begin{theorem}\label{T-4.27}
Let $0<\alpha<1$ and $p,q>0$. If $\lambda_1(SG_{\alpha,p}(A,B))\le\lambda_1(\cA_{\alpha,q}(A,B))$
holds for all $A,B\in\bM_2^{++}$, then $q/p\ge\max\{\alpha,1-\alpha\}$. Hence, if
$SG_{\alpha,p}(A,B)\prec_w\cA_{\alpha,q}(A,B)$ for all $A,B\in\bM_2^{++}$, then
$q/p\ge\max\{\alpha,1-\alpha\}$.
\end{theorem}

\begin{proof}
Assume that $\lambda_1(SG_{\alpha,p}(A,B))\le\lambda_1(\cA_{\alpha,q}(A,B))$ holds for all
$A,B\in\bM_2^{++}$. Then as in the first paragraph of the proof of Theorem \ref{T-4.23} it follows that
\[
\lambda_1((X^\alpha YX^\alpha)^r)\le\lambda_1((1-\alpha)Y^r+\alpha(XYX)^r)
\]
holds for all $X,Y\in\bM_2^{++}$, where $r:=q/p$. Now set $X:=A_0$ and $Y:=B_\theta$ in \eqref{F-3.5}
for $x,y>0$ with $x^2y\ne1$. Apply Lemma \ref{L-3.6} to \eqref{F-4.23} and \eqref{F-4.35} with $q=1$ to
find that
\begin{align*}
&\lambda_1((1-\alpha)Y^r+\alpha(XYX)^r)
=1+\theta^2\biggl(b_{11}+{b_{12}^2\over1-b}\biggr)+o(\theta^2)
\quad\mbox{if $0<b<1$}, \\
&\lambda_1((X^\alpha YX^\alpha)^r)
=1+\theta^2\biggl(\hat s_{11}+{\hat s_{12}^2\over1-x^{2\alpha r}y^r}\biggr)+o(\theta^2)
\qquad\mbox{if $0<x^{2\alpha r}y^r<1$},
\end{align*}
where $b,b_{11},b_{12}$ are in \eqref{F-4.24} via \eqref{F-4.22} and $\hat s_{11},\hat s_{12}$ are in
\eqref{F-4.36} with $q=1$. Hence, for any $x>0$, if $y>0$ is sufficiently small, then we must have
\[
\hat s_{11}+{\hat s_{12}^2\over1-x^{2\alpha r}y^r}\le b_{11}+{b_{12}^2\over1-b}.
\]
When $y\searrow0$, the above inequality becomes
\[
-r+(r-1)x^{2\alpha}+x^{2\alpha}
\le-\alpha r+\alpha(r-x)x^2-(1-\alpha)+(\alpha x+1-\alpha)^2.
\]
Letting $x\searrow0$ in the above gives $-r\le-\alpha r-(1-\alpha)+(1-\alpha)^2$, so that $r\ge\alpha$.
Moreover, from symmetry we have $r\ge1-\alpha$ too, so the result follows.
\end{proof}

Note that the sufficient condition in Proposition \ref{P-4.25} is indeed stricter than the necessary condition in
Theorem \ref{T-4.27}, because $2\min\{\alpha,1-\alpha\}<1/\max\{\alpha,1-\alpha\}$.

\begin{proposition}\label{P-4.28}
Let $0<\alpha<1$ and $p,q>0$. If $q\ge1$ or ${p/q}\le2\min\{\alpha,1-\alpha\}$, then
$\Tr\,SG_{\alpha,p}(A,B)\le\Tr\,\cA_{\alpha,q}(A,B)$ holds for all $A,B\ge0$ with $s(A)\ge s(B)$.
\end{proposition}

\begin{proof}
If ${p/q}\le2\min\{\alpha,1-\alpha\}$, then the result follows from Proposition \ref{P-4.25}(2). If $q\ge1$, then
we have, with $r>0$ satisfying $p/r\le\min\{\alpha,1-\alpha\}$,
\[
\Tr\,SG_{\alpha,p}(A,B)\le\Tr\,R_{\alpha,r}(A,B)\le\Tr\,\cA_{\alpha,q}(A,B)
\]
by Propositions \ref{P-4.25}(1) and \ref{P-4.11}.
\end{proof}

\begin{theorem}\label{T-4.29}
Let $0<\alpha<1$ and $p,q>0$. If $\Tr\,SG_{\alpha,p}(A,B)\le\Tr\,\cA_{\alpha,q}(A,B)$ holds for all
$A,B\in\bM_2^{++}$, then $\min\{1,p/2\}\le q$.
\end{theorem}

\begin{proof}
It suffices to show that  we must have $p/2\le q$ when $q<1$. Consider once again
$A_0,B_\theta\in\bM_2^{++}$ in \eqref{F-3.5} with $y=x>0$, $x\ne1$. Note that
\[
\Tr\,SG_{\alpha,p}(A_0,B_\theta)
=\Tr\bigl\{(A_0^{-p}\#B_\theta^p)^\alpha A_0^p(A_0^{-p}\#B_\theta^p)^\alpha\bigr\}^{1/p}
=\Tr\bigl\{A_0^{p/2}(A_0^{-p}\#B_\theta^p)^{2\alpha}A_0^{p/2}\bigr\}^{1/p}.
\]
We compute
\[
A_0^{-p}\#B_\theta^p=\begin{bmatrix}1+\theta^2\xi_{11}&\theta\xi_{12}\\
\theta\xi_{12}&1+\theta^2\xi_{22}\end{bmatrix}+o(\theta^2),
\]
where
\[
\begin{cases}
\xi_{11}:=-{(1-x^p)(1+3x^p)\over2(1+x^p)^2}, \\
\xi_{22}:={(1-x^p)(3+x^p)\over2(1+x^p)^2}, \\
\xi_{12}:={1-x^p\over1+x^p}.
\end{cases}
\]
This computation is similar to the previous ones based on Example \ref{E-3.2}(1), so the details are left to
the reader. Furthermore, with $H:=\begin{bmatrix}0&\xi_{12}\\\xi_{12}&0\end{bmatrix}$ and
$K:=\begin{bmatrix}\xi_{11}&0\\0&\xi_{22}\end{bmatrix}$ we have by the Taylor expansion,
\begin{align*}
(A_0^{-p}\#B_\theta^p)^{2\alpha}
&=(I_2+\theta H+\theta^2K+o(\theta^2))^{2\alpha} \\
&=I_2+2\alpha(\theta H+\theta^2K)+\alpha(2\alpha-1)(\theta H)^2+o(\theta^2) \\
&=I_2+2\theta\alpha H+\theta^2\alpha(2K+(2\alpha-1)H^2)+o(\theta^2),
\end{align*}
from which it is easy to compute
\begin{align*}
A_0^{p/2}(A_0^{-p}\#B_\theta^p)^{2\alpha}A_0^{p/2}
&=\begin{bmatrix}1+\theta^2\eta_{11}(p)&\theta\eta_{12}(p)\\
\theta\eta_{12}(p)&x^p+\theta^2\eta_{22}(p)\end{bmatrix}+o(\theta^2),
\end{align*}
where
\[
\begin{cases}
\eta_{11}(p):=-{2\alpha(1-x^p)(1-\alpha+(1+\alpha)x^p)\over(1+x^p)^2}, \\
\eta_{22}(p):={2\alpha x^p(1-x^p)(1+\alpha+(1-\alpha)x^p)\over(1+x^p)^2}, \\
\eta_{12}(p):={2\alpha x^{p/2}(1-x^p)\over1+x^p}.
\end{cases}
\]
Then by Lemma \ref{L-3.6} the two eigenvalues of
$A_0^{p/2}(A_0^{-p}\#B_\theta^p)^{2\alpha}A_0^{p/2}$ are
\[
1+\theta^2\biggl(\eta_{11}(p)+{\eta_{12}(p)^2\over1-x^p}\biggr)+o(\theta^2),\qquad
x^p+\theta^2\biggl(\eta_{22}(p)-{\eta_{12}(p)^2\over1-x^p}\biggr)+o(\theta^2).
\]
Hence we find that
\begin{align*}
\Tr\,SG_{\alpha,p}(A_0,B_\theta)
&=\biggl\{1+\theta^2\biggl(\eta_{11}(p)+{\eta_{12}(p)^2\over1-x^p}\biggr)\biggr\}^{1/p}
+\biggl\{x^p+\theta^2\biggl(\eta_{22}(p)-{\eta_{12}(p)^2\over1-x^p}\biggr)\biggr\}^{1/p}+o(\theta^2) \\
&=1+x+{\theta^2\over p}\bigg\{\eta_{11}(p)+{\eta_{12}(p)^2\over1-x^p}
+x^{1-p}\biggl(\eta_{22}(p)+{\eta_{12}(p)^2\over1-x^p}\biggr)\biggr\}+o(\theta^2).
\end{align*}
On the other hand, we have estimated $\Tr\,\cA_{\alpha,q}(A_0,B_\theta)$ in \eqref{F-4.9}. Therefore, we
must have
\[
{1\over p}\bigg\{\eta_{11}(p)+{\eta_{12}(p)^2\over1-x^p}
+x^{1-p}\biggl(\eta_{22}(p)+{\eta_{12}(p)^2\over1-x^p}\biggr)\biggr\}
\le{\alpha(1-\alpha)\over q}(-1-x+x^q+x^{1-q})
\]
for all $x>0$, $x\ne1$. As $x\searrow0$ note that
\begin{align*}
&\eta_{11}(p)\to-2\alpha(1-\alpha),\qquad\eta_{12}(p)\to0, \\
&x^{1-p}\eta_{22}(p)={2\alpha x(1-x^p)(1+\alpha+(1-\alpha)x^p)\over(1+x^p)^2}\to0, \\
&x^{1-p}\,{\eta_{12}(p)^2\over1-x^p}={4\alpha^2 x(1-x^p)\over(1+x^p)^2}\to0, \\
&-1-x+x^q+x^{1-q}\to-1,
\end{align*}
where the last convergence is due to the assumption $q<1$. Therefore, we have
$-{2\alpha(1-\alpha)\over p}\le-{\alpha(1-\alpha)\over q}$, showing that $p/2\le q$.
\end{proof}

\begin{proposition}\label{P-4.30}
Let $0<\alpha<1$ and $p,q>0$ be such that $q\ge1$ or ${p/q}\le2\min\{\alpha,1-\alpha\}$. Then for every
$A,B\ge0$ with $s(A)\ge s(B)$ the following conditions are equivalent:
\begin{itemize}
\item[(i)] $\Tr\,SG_{\alpha,p}(A,B)=\Tr\,\cA_{\alpha,q}(A,B)$;
\item[(ii)] $SG_{\alpha,p}(A,B)=\cA_{\alpha,q}(A,B)$;
\item[(iii)] $A=B$.
\end{itemize}
\end{proposition}

\begin{proof}
It suffices to show (i)$\implies$(iii). Assume (i). Let $r:=\min\{\alpha,1-\alpha\}$; then $\min\{1,(p/r)/2\}\le q$.
By Propositions \ref{P-4.25}(1) and \ref{P-4.11} we have
\[
\Tr\,SG_{\alpha,p}(A,B)\le\Tr\,R_{\alpha,p/r}(A,B)\le\Tr\,\cA_{\alpha,q}(A,B)
\]
so that $\Tr\,R_{\alpha,p/r}(A,B)=\Tr\,\cA_{\alpha,q}(A,B)$. Hence (iii) follows from Proposition \ref{P-4.14}.
\end{proof}

\newpage
The results of this subsection are summarized as follows:

\medskip
\begin{table}[htb]
\centering
\begin{tabular}{|c|l|l|} \hline
& Sufficient cond. & Necessary cond. \\ \hline

$\begin{array}{cc}SG_{\alpha,p}\le\cA_{\alpha,q} \\ SG_{\alpha,p}\le_\chao\cA_{\alpha,q}\end{array}$
& none & \\ \hline

$SG_{\alpha,p}\le_\near\cA_{\alpha,q}$ &\ \ ?
& $q/p\ge1+\max\{\alpha(q-1),(1-\alpha)(q-1)\}$ \\ \hline

$SG_{\alpha,p}\le_\lambda\cA_{\alpha,q}$ &\ \ ?
& $q/p\ge\max\{\alpha,1-\alpha\}$ \\ \hline

$SG_{\alpha,p}\prec_w\cA_{\alpha,q}$ & ${p/q}\le2\min\{\alpha,1-\alpha\}$ &
$q/p\ge\max\{\alpha,1-\alpha\}$ \\ \hline

$SG_{\alpha,p}\le_\Tr\cA_{\alpha,q}$ & $q\ge1$ or ${p/q}\le2\min\{\alpha,1-\alpha\}$ &
$q\ge1$ or $p\le2q$ \\ \hline
\end{tabular}
\end{table}

\begin{problem}\label{Q-4.31}\rm
At the moment we find no sufficient condition for $SG_{\alpha,p}\le_\near\cA_{\alpha,q}$ to hold, while a
necessary condition is given in Theorem \ref{T-4.24}. It might happen that this inequality never holds for
any $p,q>0$. For example, when $\alpha=1/2$, since the necessary condition gives $p\le{2q\over q+1}$,
we notice that $SG_{1/2,p}\le_\near\cA_{1/2,q}$ fails to hold for any $q>0$ if $p\ge2$, and for any
$q<1$ if $p=1$. But it is still unknown if $F(A,B)\le_\near{A+B\over2}$ (the case $p=q=1$) holds for
all $A,B>0$. The situation is similar for $SG_{\alpha,p}\le_\lambda\cA_{\alpha,q}$ except the case
$\alpha\ne1/2$. When $\alpha=1/2$, since $\lambda(F(A,B))=\lambda\bigl((B^{1/2}AB^{1/2})^{1/2}\bigr)$
(see \cite[Theorem 3.2, Item 8]{FP}), it is easy to see that if $0<p\le q$, then
$SG_{1/2,p}(A,B)\le_\lambda\cA_{1/2,q}(A,B)$ for all $A,B>0$. As for $SG_{\alpha,p}\prec_w\cA_{\alpha,q}$
and $SG_{\alpha,p}\le_\Tr\cA_{\alpha,q}$, there is a rather big gap between the sufficient condition and the
necessary condition obtained.
\end{problem}

\subsection{$\widetilde SG_{\alpha,p}$ vs.\ $\cA_{\alpha,q}$}\label{Sec-4.6}

In this subsection we discuss inequalities between $\widetilde SG_{\alpha,p}$ and $\cA_{\alpha,q}$. The next
theorem says that $\widetilde SG_{\alpha,p}(A,B)\le_\near\cA_{\alpha,q}(A,B)$ fails to hold for any
$\alpha\in(0,1)\setminus\{1/2\}$ and any $p,q>0$.

\begin{theorem}\label{T-4.32}
For any $\alpha\in(0,1)\setminus\{1/2\}$ and any $p,q>0$ there exist $A,B\in\bM_2^{++}$ such that
$\widetilde SG_{\alpha,p}(A,B)\not\le_\near\cA_{\alpha,q}(A,B)$. Hence, for any $\alpha\in(0,1)$ and any
$p,q>0$ there exist $A,B\in\bM_2^{++}$ such that
$\widetilde SG_{\alpha,p}(A,B)\not\le_\chao\cA_{\alpha,q}(A,B)$.
\end{theorem}

\begin{proof}
First, note that the latter assertion follows from the first and Theorem \ref{T-4.23} since
$\widetilde SG_{1/2,p}=SG_{1/2,p}$. For $A,B>0$ set $Y:=A^p$ and $X:=A^{-p}\#_\alpha B^p$; then
$B^p=Y^{-1/2}(Y^{1/2}XY^{1/2})^{1/\alpha}Y^{-1/2}$. We will express this RHS as $Y^{-1}\#_{1/\alpha}X$,
though $1/\alpha>1$, in analogy of the geometric mean $\#_\alpha$. Then
$\widetilde SG_{\alpha,p}(A,B)\le_\near\cA_{\alpha,q}(A,B)$ is equivalently written as
\begin{align}\label{F-4.44}
(X^{1/2}Y^{2(1-\alpha)}X^{1/2})^{r/q}\le_\near
\bigl\{(1-\alpha)Y^r+\alpha(Y^{-1}\#_{1/\alpha}X)^r\bigr\}^{1/q},
\end{align}
where $r:=q/p$. Conversely, for any $X,Y>0$, if we set $A:=Y^{1/p}$ and $B:=(Y^{-1}\#_{1/\alpha}X)^{1/p}$,
then $Y=A^p$ and $X=A^{-p}\#_\alpha B^p$. Hence it suffices to show that for any
$\alpha\in(0,1)\setminus\{1/2\}$ and any $r,q>0$ there exist $X,Y\in\bM_2^{++}$ for which \eqref{F-4.44}
is violated.

Let $X:=B_\theta$ and $Y:=A_0$ in \eqref{F-3.5} with $y=x^{2\alpha-1}>0$ ($\alpha\ne1/2$) and $x\ne1$.
Set
\[
L:=(X^{1/2}Y^{2(1-\alpha)}X^{1/2})^{r/q},\qquad
M:=\bigl\{(1-\alpha)Y^r+\alpha(Y^{-1}\#_{1/\alpha}X)^r\bigr\}^{1/q}.
\]
In the following computations we will repeatedly apply the reduced version of Example \ref{E-3.2}(1) up to
$o(\theta)$. One computes
$X^{1/2}Y^{2(1-\alpha)}X^{1/2}=\begin{bmatrix}1&\theta a_{12}\\\theta a_{12}&x\end{bmatrix}+o(\theta)$,
where
\begin{align}\label{F-4.45}
a_{12}:=\bigl(1-x^{2\alpha-1\over2}\bigr)\bigl(1+x^{3-2\alpha\over2}\bigr),
\end{align}
and hence
\begin{align}\label{F-4.46}
L=\begin{bmatrix}1&\theta\,{1-x^{r\over q}\over1-x}\,a_{12}\\
\theta\,{1-x^{r\over q}\over1-x}\,a_{12}&x^{r\over q}\end{bmatrix}+o(\theta)\quad\mbox{as $\theta\to0$}.
\end{align}
On the other hand, one computes
$Y^{-1}\#_{1/\alpha}X=\begin{bmatrix}1&\theta\xi_{12}\\\theta\xi_{12}&x\end{bmatrix}+o(\theta)$,
where
\[
\xi_{12}:={(1-x^2)(1-x^{2\alpha-1})\over1-x^{2\alpha}},
\]
and hence
$(Y^{-1}\#_{1/\alpha}X)^r=\begin{bmatrix}1&\theta\zeta_{12}\\\theta\zeta_{12}&x^r\end{bmatrix}+o(\theta)$,
where
\begin{align}\label{F-4.47}
\zeta_{12}:={1-x^r\over1-x}\,\xi_{12}={(1+x)(1-x^r)(1-x^{2\alpha-1})\over1-x^{2\alpha}}.
\end{align}
Since $(1-\alpha)Y^r+\alpha(Y^{-1}\#_{1/\alpha}X)^r
=\begin{bmatrix}1&\theta\alpha\zeta_{12}\\\theta\alpha\zeta_{12}&x^r\end{bmatrix}+o(\theta)$,
one has
\begin{align}\label{F-4.48}
M=\begin{bmatrix}1&\theta\,{1-x^{r\over q}\over1-x^r}\,\alpha\zeta_{12}\\
\theta\,{1-x^{r\over q}\over1-x^r}\,\alpha\zeta_{12}&x^{r\over q}\end{bmatrix}+o(\theta)
\quad\mbox{as $\theta\to0$}.
\end{align}
From \eqref{F-4.48} and \eqref{F-4.46} one can further compute
\[
M^{1/2}=\begin{bmatrix}1&\theta{1-x^{r\over2q}\over1-x^r}\,\alpha\zeta_{12}\\
\theta{1-x^{r\over2q}\over1-x^r}\,\alpha\zeta_{12}&x^{r\over2q}\end{bmatrix}+o(\theta),\qquad
M^{1/2}LM^{1/2}=\begin{bmatrix}1&\theta u_{12}\\u_{12}&x^{2r\over q}\end{bmatrix}+o(\theta),
\]
where
\[
u_{12}:={\bigl(1-x^{r\over2q}\bigr)\bigl(1+x^{3r\over2q}\bigr)\over1-x^r}\,\alpha\zeta_{12}
+{x^{r\over2q}\bigl(1-x^{r\over q}\bigr)\over1-x}\,a_{12}.
\]
Furthermore,
\begin{align}\label{F-4.49}
(M^{1/2}LM^{1/2})^{1/2}=\begin{bmatrix}1&\theta v_{12}\\v_{12}&x^{r\over q}\end{bmatrix}+o(\theta),
\quad\mbox{as $\theta\to0$},
\end{align}
where
\[
v_{12}:={1-x^{r\over q}\over1-x^{2r\over q}}\,u_{12}
={\bigl(1-x^{r\over q}\bigr)\bigl(1-x^{r\over2q}\bigr)\bigl(1+x^{3r\over2q}\bigr)\over
(1-x^r)(1-x^{2r\over q}\bigr)}\,\alpha\zeta_{12}
+{x^{r\over2q}\bigl(1-x^{r\over q}\bigr)^2\over(1-x)\bigl(1-x^{2r\over q}\bigr)}\,a_{12}.
\]

Now assume that $L\le_\near M$ for all $X,Y\in\bM_2^{++}$. Then we must have
$(M^{1/2}LM^{1/2})^{1/2}\le M$ so that by \eqref{F-4.48} and \eqref{F-4.49},
\begin{align*}
0&\le\det\{M-(M^{1/2}LM^{1/2})^{1/2}\} \\
&=\det\Biggl(\begin{bmatrix}0&
\theta\Bigl({1-x^{r\over q}\over1-x^r}\,\alpha\zeta_{12}-v_{12}\Bigr)\\
\theta\Bigl({1-x^{r\over q}\over1-x^r}\,\alpha\zeta_{12}-v_{12}\Bigr)&0\end{bmatrix}+o(\theta)\Biggr) \\
&=-\theta^2\biggl({1-x^{r\over q}\over1-x^r}\,\alpha\zeta_{12}-v_{12}\biggr)^2+o(\theta^2).
\end{align*}
Therefore, we must have
\[
{1-x^{r\over q}\over1-x^r}\,\alpha\zeta_{12}=v_{12}
={\bigl(1-x^{r\over q}\bigr)\bigl(1-x^{r\over2q}\bigr)\bigl(1+x^{3r\over2q}\bigr)\over
(1-x^r)(1-x^{2r\over q}\bigr)}\,\alpha\zeta_{12}
+{x^{r\over2q}\bigl(1-x^{r\over q}\bigr)^2\over(1-x)\bigl(1-x^{2r\over q}\bigr)}\,a_{12},
\]
which becomes ${\alpha\zeta_{12}\over1-x^r}={a_{12}\over1-x}$. By \eqref{F-4.45} and \eqref{F-4.47}
this gives
\[
{(1+x)(1+x^{2\alpha-1\over2})\over1-x^{2\alpha}}\,\alpha={1+x^{3-2\alpha\over2}\over1-x}
\]
for all $x>0$ with $x\ne1$, which fails to hold as $x\searrow0$ in either case $0<\alpha<1/2$ or
$1/2<\alpha<1$.
\end{proof}

\begin{proposition}\label{P-4.33}
Let $0<\alpha<1$ and $p,q>0$.
\begin{itemize}
\item[(1)] If $p\le\alpha q$, then we have $\widetilde SG_{\alpha,p}(A,B)\prec_{\log}R_{\alpha,q}(A,B)$ for
every $A,B\ge0$ with $s(A)\ge s(B)$.
\item[(2)] If $p\le2\alpha q$, then we have $\widetilde SG_{\alpha,p}(A,B)\prec_{w\log}\cA_{\alpha,q}(A,B)$
(hence $\widetilde SG_{\alpha,p}(A,B)\prec_w\cA_{\alpha,q}(A,B)$) for every $A,B\ge0$ with $s(A)\ge s(B)$.
\end{itemize}
\end{proposition}

\begin{proof}
(1)\enspace
By continuity (Proposition \ref{P-2.2}) we may assume that $A,B>0$. Similarly to the proof of
Proposition \ref{P-4.25}(1) it suffices to show that
\[
\widetilde SG_{\alpha,p}(A,B)\prec_{\log}R_{\alpha,p/\alpha}(A,B),\qquad p>0.
\]
For $0<\alpha\le1/2$, by Araki's log-majorization \cite{Ar} we have
\begin{align*}
\lambda((B^{-1}\#_\alpha A)^{1/2}B^{2(1-\alpha)}(B^{-1}\#_\alpha A)^{1/2})
&=\lambda(B^{1-\alpha}(B^{-1}\#_\alpha A)B^{1-\alpha}) \\
&=\lambda\bigl(B^{{1\over2}-\alpha}(B^{1/2}AB^{1/2})^\alpha B^{{1\over2}-\alpha}\bigr) \\
&\prec_{\log}\lambda^\alpha\bigl(B^{{1\over2\alpha}-1}(B^{1/2}AB^{1/2})B^{{1\over2\alpha}-1}\bigr) \\
&=\lambda^\alpha\bigl(B^{1-\alpha\over2\alpha}AB^{1-\alpha\over2\alpha}\bigr)
=\lambda(R_{\alpha,1/\alpha}(A,B)).
\end{align*}
The remaining proof is the same as in the proof of Proposition \ref{P-4.25}(1).

(2) follows from (1) and Proposition \ref{P-4.9} as in the proof of Proposition \ref{P-4.25}(2).
\end{proof}

\begin{remark}\label{R-4.34}\rm
Although we have shown the log-majorizations in Propositions \ref{P-4.25}(1) and \ref{P-4.33}(1) from
Araki's log-majorization, they can also be shown by directly applying the familiar antisymmetric tensor
power technique (see \cite{AH}). Furthermore, it is known that
$\widetilde SG_{\alpha,p}(A,B)\prec_{\log}R_{\alpha,q}(A,B)$ for all $A,B>0$ if and only if $p/q\le\alpha$,
and that $R_{\alpha,q}(A,B)\prec_{\log}\widetilde SG_{\alpha,p}(A,B)$ for all $A,B>0$ if $\alpha\le1/2$
and $q\le p$, and the same holds only if $\alpha\le1/2$ and $p/q\ge1/2$. The details on these facts will be
provided in \cite{Hi5}.
\end{remark}

\begin{theorem}\label{T-4.35}
Let $0<\alpha<1$ and $p,q>0$. If
$\lambda_1\bigl(\widetilde SG_{\alpha,p}(A,B)\bigr)\le\lambda_1(\cA_{\alpha,q}(A,B))$ holds for all
$A,B\in\bM_2^{++}$, then $q/p\ge1-\alpha$. Hence, if
$\widetilde SG_{\alpha,p}(A,B)\prec_w\cA_{\alpha,q}(A,B)$ for all $A,B\in\bM_2^{++}$, then
$q/p\ge1-\alpha$.
\end{theorem}

\begin{proof}
Assume that $\lambda_1\bigl(\widetilde SG_{\alpha,p}(A,B)\bigr)\le\lambda_1(\cA_{\alpha,q}(A,B))$ holds
for all $A,B\in\bM_2^{++}$, and argue as in the first paragraph of the proof of Theorem \ref{T-4.32}. Then,
since $\lambda((X^{1/2}Y^{2(1-\alpha)}X^{1/2})^{r/q})=\lambda((Y^{1-\alpha}XY^{1-\alpha})^{r/q})$, we have
\[
\lambda_1\bigl((Y^{1-\alpha}XY^{1-\alpha})^{r/q}\bigr)\le
\lambda_1\bigl(\bigl\{(1-\alpha)Y^r+(Y^{-1}\#_{1/\alpha}X)^r\bigr\}^{1/q}\bigr)
\]
holds for all $X,Y\in\bM_2^{++}$, where $r:=q/p$. Now set $X:=B_\theta$ and $Y:=A_0$ in \eqref{F-3.5}
for $x,y>0$ with $xy\ne1$ and $x^{1-\alpha}y\ne1$. We compute
\[
Y^{1-\alpha}XY^{1-\alpha}=\begin{bmatrix}1-\theta^2(1-y)&\theta x^{1-\alpha}(1-y)\\
\theta x^{1-\alpha}(1-y)&x^{2(1-\alpha)}y+\theta^2x^{2(1-\alpha)}(1-y)\end{bmatrix}+o(\theta^2)
\quad\mbox{as $\theta\to0$}.
\]
On the other hand, using Example \ref{E-3.2}(1) we compute
\[
Y^{-1}\#_{1/\alpha}X=\begin{bmatrix}1+\theta^2 u_{11}&\theta u_{12}\\
\theta u_{12}&x^{1-\alpha\over\alpha}y^{1\over\alpha}+\theta^2u_{22}\end{bmatrix}+o(\theta^2),
\]
where
\[
\begin{cases}
u_{11}:=-{1-y\over\alpha}+{x(1-y)^2\bigl(1-\alpha-xy+\alpha x^{1\over\alpha}y^{1\over\alpha}\bigr)
\over\alpha(1-xy)^2}, \\
u_{12}:={(1-y)\bigl(1-x^{1\over\alpha}y^{1\over\alpha}\bigr)\over1-xy},
\end{cases}
\]
and hence
\[
(Y^{-1}\#_{1/\alpha}X)^r=\begin{bmatrix}1+\theta^2v_{11}&\theta v_{12}\\
\theta v_{12}&x^{(1-\alpha)r\over\alpha}y^{r\over\alpha}+\theta^2v_{22}\end{bmatrix}+o(\theta^2),
\]
where
\[
\begin{cases}
v_{11}:=ru_{11}+{r-1-rx^{1-\alpha\over\alpha}y^{1\over\alpha}
+x^{(1-\alpha)r\over\alpha}y^{r\over\alpha}\over
\bigl(1-x^{1-\alpha\over\alpha}y^{1\over\alpha}\bigr)^2}\,u_{12}^2, \\
v_{12}:={1-x^{(1-\alpha)r\over\alpha}y^{r\over\alpha}\over
1-x^{1-\alpha\over\alpha}y^{1\over\alpha}}\,u_{12}.
\end{cases}
\]
Therefore,
\[
(1-\alpha)Y^r+\alpha(Y^{-1}\#_{1/\alpha}X)^r
=\begin{bmatrix}1+\theta^2\alpha v_{11}&\theta\alpha v_{12}\\\theta\alpha v_{12}&
(1-\alpha)x^r+\alpha x^{(1-\alpha)r\over\alpha}y^{r\over\alpha}+\theta^2\alpha v_{22}\end{bmatrix}
+o(\theta^2)\quad\mbox{as $\theta\to0$}.
\]

Let $x,y<1$. Note that $x^{2(1-\alpha)}y<1$ and
$(1-\alpha)x^r+\alpha x^{(1-\alpha)r\over\alpha}y^{r\over\alpha}<1$. Then by Lemma \ref{L-3.6} we have
\begin{align*}
&\lambda_1(Y^{1-\alpha}XY^{1-\alpha}\bigr)
=1+\theta^2\biggl\{-(1-y)+{x^{2(1-\alpha)}(1-y)^2\over1-x^{2(1-\alpha)}y}\biggr\}+o(\theta^2), \\
&\lambda_1\bigl((1-\alpha)Y^r+\alpha(Y^{-1}\#_{1/\alpha}X)^r\bigr)
=1+\theta^2\Biggl\{\alpha v_{11}+{\alpha^2v_{12}^2\over
1-(1-\alpha)x^r-\alpha x^{(1-\alpha)r\over\alpha}y^{r\over\alpha}}\Biggr\}+o(\theta^2),
\end{align*}
so that
\begin{align*}
&\lambda_1\bigl((Y^{1-\alpha}XY^{1-\alpha})^{r/q}\bigr)
=1+{\theta^2r\over q}\biggl\{-1+y+{x^{2(1-\alpha)}(1-y)^2\over1-x^{2(1-\alpha)}y}\biggr\}+o(\theta^2), \\
&\lambda_1\bigl(\bigl\{(1-\alpha)Y^r+\alpha(Y^{-1}\#_{1/\alpha}X)^r\bigr\}^{1/q}\bigr)
=1+{\theta^2\over q}\Biggl\{\alpha v_{11}+{\alpha^2v_{12}^2\over
1-(1-\alpha)x^r-\alpha x^{(1-\alpha)r\over\alpha}y^{r\over\alpha}}\Biggr\}+o(\theta^2).
\end{align*}
This implies that
\[
r\biggl\{-1+y+{x^{2(1-\alpha)}(1-y)^2\over1-x^{2(1-\alpha)}y}\biggr\}
\le\alpha v_{11}+{\alpha^2v_{12}^2\over1-(1-\alpha)x^r-\alpha x^{(1-\alpha)r\over\alpha}y^{r\over\alpha}}.
\]
Letting $y\searrow0$ gives
\begin{align*}
&u_{11}\to-{1\over\alpha}+{1-\alpha\over\alpha}\,x,\qquad u_{12}\to1, \\
&v_{11}\to r\biggl(-{1\over\alpha}+{1-\alpha\over\alpha}\,x\biggr)+(r-1),\qquad v_{12}\to1.
\end{align*}
Hence we must have
\[
-r+rx^{2(1-\alpha)}\le-r+r(1-\alpha)x+\alpha(r-1)+{\alpha^2\over1-(1-\alpha)x^r}.
\]
Letting $x\searrow0$ further gives $0\le\alpha(r-1)+\alpha^2$, i.e., $r\ge1-\alpha$.
\end{proof}

\begin{proposition}\label{P-4.36}
Let $0<\alpha<1$ and $p,q>0$. If $\min\{1,{p\over2\alpha}\}\le q$, then
$\Tr\,\widetilde SG_{\alpha,p}(A,B)\le\Tr\,\cA_{\alpha,q}(A,B)$ holds for all $A,B\ge0$ with $s(A)\ge s(B)$.
\end{proposition}

\begin{proof}
If $p\le2\alpha q$, then the result follows from Proposition \ref{P-4.33}(2). If $q\ge1$, then we have, with
$r>0$ satisfying $p/r\le\alpha$,
\[
\Tr\,\widetilde SG_{\alpha,p}(A,B)\le\Tr\,R_{\alpha,r}(A,B)\le\Tr\,\cA_{\alpha,q}(A,B)
\]
by Propositions \ref{P-4.33}(1) and \ref{P-4.11}.
\end{proof}

\begin{theorem}\label{T-4.37}
Let $0<\alpha<1$ and $p,q>0$. If $\Tr\,\widetilde SG_{\alpha,p}(A,B)\le\Tr\,\cA_{\alpha,q}(A,B)$ holds for
all $A,B\in\bM_2^{++}$, then  $\min\{1,(1-\alpha)p\}\le q$.
\end{theorem}

\begin{proof}
It suffices to show that we must have $(1-\alpha)p\le q$ when $q<1$. Consider $A_0,B_\theta\in\bM_2^{++}$
in \eqref{F-3.5} with $y=x>0$, $x\ne1$. Note that
\begin{align*}
\Tr\,\widetilde SG_{\alpha,p}(A_0,B_\theta)
&=\Tr\bigl\{(A_0^{-p}\#_\alpha B_\theta^p)^{1/2}A_0^{2(1-\alpha)p}
(A_0^{-p}\#_\alpha B_\theta^p)^{1/2}\bigr\}^{1/p} \\
&=\Tr\bigl\{A_0^{(1-\alpha)p}(A_0^{-p}\#_\alpha B_\theta^p)A_0^{(1-\alpha)p}\bigr\}^{1/p} \\
&=\Tr\bigl\{A_0^{({1\over2}-\alpha)p}(A_0^{p/2}B_\theta^pA_0^{p/2})^\alpha
A_0^{({1\over2}-\alpha)p}\bigr\}^{1/p}.
\end{align*}
Using Example \ref{E-3.2}(1) as before we compute
\[
A_0^{({1\over2}-\alpha)p}(A_0^{p/2}B_\theta^pA_0^{p/2})^\alpha A_0^{({1\over2}-\alpha)p}
=\begin{bmatrix}1+\theta^2\eta_{11}(p)&\theta\eta_{12}(p)\\
\theta\eta_{12}(p)&x^p+\theta^2\eta_{22}(p)\end{bmatrix}+o(\theta^2)\quad\mbox{as $\theta\to0$},
\]
where
\[
\begin{cases}
\eta_{11}(p):=-\alpha(1-x^p)+{x^p(\alpha-1-\alpha x^{2p}+x^{2\alpha p})\over(1+x^p)^2}, \\
\eta_{22}(p):=\alpha(1-x^p)+{x^{2(1-\alpha)p}(1+(\alpha-1)x^{2\alpha p}-\alpha x^{2(\alpha-1)p})
\over(1+x^p)^2}, \\
\eta_{12}(p):={x^{(1-\alpha)p}(1-x^{2\alpha p})\over1+x^p}.
\end{cases}
\]
Hence we find by Lemma \ref{L-3.6} that
\begin{align*}
\Tr\,\widetilde SG_{\alpha,p}(A_0,B_\theta)
&=\biggl\{1+\theta^2\biggl(\eta_{11}(p)+{\eta_{12}(p)^2\over1-x^p}\biggr)\biggr\}^{1/p}
+\biggl\{x^p+\theta^2\biggl(\eta_{22}(p)-{\eta_{12}(p)^2\over1-x^p}\biggr)\biggr\}^{1/p}+o(\theta^2) \\
&=1+x+{\theta^2\over p}\biggl\{\eta_{11}(p)+{\eta_{12}(p)^2\over1-x^p}
+x^{1-p}\biggl(\eta_{22}(p)-{\eta_{12}(p)^2\over1-x^p}\biggr)\biggr\}+o(\theta^2).
\end{align*}
From this and \eqref{F-4.9} for $\Tr\,\cA_{\alpha,q}(A_0,B_\theta)$ we must have
\begin{align}\label{F-4.50}
{1\over p}\biggl\{\eta_{11}(p)+{\eta_{12}(p)^2\over1-x^p}
+x^{1-p}\biggl(\eta_{22}(p)-{\eta_{12}(p)^2\over1-x^p}\biggr)\biggr\}
\le{\alpha(1-\alpha)\over q}(-1-x+x^q+x^{1-q})
\end{align}
for all $x>0$, $x\ne1$. As $x\searrow0$, noting that $\eta_{11}(p)\to-\alpha$, $\eta_{12}(p)\to0$ and
\begin{align*}
&x^{1-p}\biggl(\eta_{22}(p)-{\eta_{12}(p)^2\over1-x^p}\biggr) \\
&\quad=x^{1-p}\biggl\{\alpha(1-x^p)+{x^{2(1-\alpha)p}(1+(\alpha-1)x^{2\alpha p}-\alpha x^{2(\alpha-1)p}
\over(1+x^p)^2}-{x^{2(1-\alpha)p}(1-x^{2\alpha p})^2\over(1+x^p)^2(1-x^p)}\biggr\} \\
&\quad=\alpha x^{1-p}\biggl\{1-{1\over(1+x^p)^2}\biggr\}-\alpha x
+x^{(1-2\alpha)p+1}\biggl\{{1+(\alpha-1)x^{2\alpha p}\over(1+x^p)^2}
-{(1-x^{2\alpha p})^2\over(1+x^p)^2(1-x^p)}\biggr\} \\
&\quad={\alpha x(2+x^p)\over(1+x^p)^2}-\alpha x
+{-x^{2(1-\alpha)p+1}+(\alpha+1)x^{p+1}-(\alpha-1)x^{2p+1}-x^{(1+2\alpha)p+1}\over(1+x^p)^2(1-x)} \\
&\quad\to0,
\end{align*}
we find that the LHS of \eqref{F-4.50} goes to $-{\alpha\over p}$, while the RHS goes to
$-{\alpha(1-\alpha)\over q}$ since $q<1$. Therefore, $-{\alpha\over p}\le-{\alpha(1-\alpha)\over q}$, showing
that $p\le{q\over1-\alpha}$.
\end{proof}

The following is seen similarly to Proposition \ref{P-4.30}.

\begin{proposition}\label{P-4.38}
Let $0<\alpha<1$ and $p,q>0$ be such that $\min\{1,{p\over2\alpha}\}\le q$. Then for every $A,B\ge0$ with
$s(A)\ge s(B)$ the following conditions are equivalent:
\begin{itemize}
\item[(i)] $\Tr\,\widetilde SG_{\alpha,p}(A,B)=\Tr\,\cA_{\alpha,p}(A,B)$;
\item[(ii)] $\widetilde SG_{\alpha,p}(A,B)=\cA_{\alpha,p}(A,B)$;
\item[(iii)] $A=B$.
\end{itemize}
\end{proposition}

The results of this subsection are summarized as follows. Note that the conditions here are not symmetric
under interchanging $\alpha$ and $1-\alpha$, unlike those in the previous subsections. This is a reflection
of the fact that $\widetilde SG_{\alpha,p}$ is not symmetric in the sense of Remark \ref{R-2.1}(2), unlike
other quasi matrix means.

\medskip
\begin{table}[htb]
\centering
\begin{tabular}{|c|l|l|} \hline
& Sufficient cond. & Necessary cond. \\ \hline

$\begin{array}{cc}\widetilde SG_{\alpha,p}\le\cA_{\alpha,q} \\
\widetilde SG_{\alpha,p}\le_\chao\cA_{\alpha,q}\end{array}$
& none & \\ \hline

$\widetilde SG_{\alpha,p}\le_\near\cA_{\alpha,q}$ & none for $\alpha\ne1/2$ & \ \ ? \\ \hline

$\widetilde SG_{\alpha,p}\le_\lambda\cA_{\alpha,q}$ & \ \ ? & $p\le{q\over1-\alpha}$ \\ \hline

$\widetilde SG_{\alpha,p}\prec_w\cA_{\alpha,q}$ & $p\le2\alpha q$ & $p\le{q\over1-\alpha}$ \\ \hline

$\widetilde SG_{\alpha,p}\le_\Tr\cA_{\alpha,q}$ & $q\ge1$ or $p\le2\alpha q$ &
$q\ge1$ or $p\le{q\over1-\alpha}$ \\ \hline
\end{tabular}
\end{table}

\newpage
\begin{problem}\label{Q-4.39}\rm
We find no necessary condition for $\widetilde SG_{\alpha,p}\le_\near\cA_{\alpha,q}$ and no sufficient
condition for $\widetilde SG_{\alpha,p}\le_\lambda\cA_{\alpha,q}$. The proof of Theorem \ref{T-4.32} cannot
apply to the case $\alpha=1/2$. Indeed, when $\alpha=1/2$, since $\widetilde SG_{1/2,p}=SG_{1/2,p}$,
the problem has already been pointed out in Problem \ref{Q-4.31}. Moreover, as for
$\widetilde SG_{\alpha,p}\prec_w\cA_{\alpha,q}$ and $\widetilde SG_{\alpha,p}\le_\Tr\cA_{\alpha,q}$,
there is a big gap between the sufficient condition and the necessary condition, similarly to those for
$SG_{\alpha,p}$ and $\cA_{\alpha,q}$.
\end{problem}

\section{Concluding remarks}\label{Sec-5}
(1)\enspace
For each quasi matrix mean $\cM_{\alpha,p}\in\{\cA_{\alpha,p},LE_\alpha,R_{\alpha,p},G_{\alpha,p},
SG_{\alpha,p},\widetilde SG_{\alpha,p}\}$ and for each matrix order
$\triangleleft\in\{\le,\le_\chao,\le_\near,\le_\lambda,\prec_w,\le_\Tr\}$, we have aimed at finding the
necessary and sufficient condition on $p,q,\alpha$ under which the inequality
$\cM_{\alpha,p}(A,B)\triangleleft\cA_{\alpha,q}(A,B)$ holds for all $A,B>0$. When
$\cM_{\alpha,p}=\cA_{\alpha,p},LE_\alpha,G_{\alpha,p}$, our objective has perfectly been achieved
as seen in the tables at the end of Sections \ref{Sec-4.1}, \ref{Sec-4.2} and \ref{Sec-4.4}. However, when
$\cM_{\alpha,p}=R_{\alpha,p},SG_{\alpha,p},\widetilde SG_{\alpha,p}$, that has not completely be done
as seen in the tables of Sections \ref{Sec-4.3}, \ref{Sec-4.5} and \ref{Sec-4.6}, where there is a gap
between the sufficient condition and the necessary condition for some of our target inequalities. Therefore,
the problem is still left open for those cases as explained in Problems \ref{Q-4.15}, \ref{Q-4.31} and
\ref{Q-4.39}. We are especially concerned with the question whether $p/2\le q$ is the necessary and
sufficient condition for $R_{\alpha,p}\le_\lambda\cA_{\alpha,q}$ to hold or not. This is indeed equivalent to
saying whether $|A^{\alpha p}B^{(1-\alpha)p}|^{1/p}\le_\lambda\alpha A+(1-\alpha)B$ holds for all $A,B>0$
only if $p\le1$ (in other words, whether Ando's matrix Young inequality
$|A^\alpha B^{1-\alpha}|\le_\lambda\alpha A+(1-\alpha)B$ is the best possible case or not).

(2)\enspace
We have considered an inequality $\cM_{\alpha,p}\triangleleft\cA_{\alpha,q}$ in the two directions of the
sufficiency part (to show the inequality under some condition) and the necessity part (to find a necessary
condition for the inequality to hold). The former direction is a more or less easy task by applying well known
facts or methods in matrix analysis. Thus we have presented the results in this direction as propositions. On
the other hand, the latter direction is computation-oriented, where we provide a counter-example with use of
a specific pair of $2\times2$ matrices, typically the pair $A_0,B_\theta$ given in \eqref{F-3.5}. We have
prepared in Section \ref{Sec-3} some technical computations on $A_0,B_\theta$ based on Taylor's theorem
for matrix functions, which are repeatedly used in the main Section \ref{Sec-4}. In this way, the results in the
latter direction are much involved with plenty of computations though elementary in nature, so we have
presented those as theorems. In order to settle the aforementioned open questions by improving the current
necessary conditions, we probably need to seek a more sophisticated example of $2\times2$ matrix pair,
or otherwise a matrix pair of higher degrees, though computations with matrix pairs of $3\times3$ or higher
seem difficult to perform by hand. Also, it seems that numerical computations are not so much helpful
because the problem is to find the best possible necessary condition.

(3)\enspace
It is intuitively clear that the implications stated in Proposition \ref{P-2.4}(3) are all strict. This fact has been
exemplified in our study of inequalities of quasi matrix means. In fact, the strictness of those implications
except for $\le_\chao\ \Rightarrow\ \le_\near$ is manifest in the tables at the end of Sections
\ref{Sec-4.1}--\ref{Sec-4.4}. As for $\le_\chao\ \Rightarrow\ \le_\near$, the following remark is worth noting:
Due to Proposition \ref{P-2.4}(2) this implication is equivalent to the so-called Ando--Hiai inequality \cite{AH}
(i.e., for $X,Y>0$, $X\#Y\le I\Rightarrow X^p\#Y^p\le I$ for $p\ge1$); see also \cite[Proposition 4]{DF}.
Therefore, its strictness corresponds to $X\#Y\le I\nRightarrow X^p\#Y^p\le I$ for $0<p<1$.

(4)\enspace
We can apply the method explained in Remark \ref{R-2.7}(3) to have the characterization of
$\cH_{\alpha,q}\triangleleft\cH_{\alpha,p}$ from that of $\cA_{\alpha,p}\triangleleft\cA_{\alpha,q}$ given in
Section \ref{Sec-4.1}. Furthermore, we have the sufficient condition (or the necessary condition) for
$\cH_{\alpha,q}\triangleleft\cM_{\alpha,p}$ from that for $\cM_{\alpha,p}\triangleleft\cA_{\alpha,q}$ given in
Sections \ref{Sec-4.2}--\ref{Sec-4.6} for any
$\cM_{\alpha,p}\in\{G_{\alpha,p},SG_{\alpha,p},\widetilde SG_{\alpha,p},R_{\alpha,p},L_\alpha\}$ and
$\triangleleft\in\{\le,\le_\chao,\le_\near,\le_\lambda\}$. But the inequalities
$\cH_{\alpha,q}\prec_w\cM_{\alpha,p}$ and $\cH_{\alpha,q}\le_\Tr\cM_{\alpha,p}$ are not touched in this
paper. As for $\cH_{\alpha,p}\triangleleft\cA_{\alpha,q}$, we have $\cH_{\alpha,p}\le_\chao\cA_{\alpha,q}$
for any $p,q>0$ since Proposition \ref{P-4.19} yields
$\cH_{\alpha,p}\le_\chao G_{\alpha,r}\le_\chao\cA_{\alpha,q}$ with $r=\min\{p,q\}$. We have also
$\cH_{\alpha,p}\le\cA_{\alpha,q}$ for any $p,q\ge1$ since Theorem \ref{T-4.16} yields
$\cH_{\alpha,p}\le G_{\alpha,1}\le\cA_{\alpha,q}$. An interesting problem left open is whether
$\cH_{\alpha,p}\le\cA_{\alpha,q}$ holds only if $p,q\ge1$ or not. Here note that for any $q>0$,
$\cH_{\alpha,p}\le\cA_{\alpha,q}$ fails to hold at least for sufficiently small $p>0$, because otherwise
letting $p\searrow0$ gives $LE_\alpha\le\cA_{\alpha,q}$ thanks to Theorem \ref{T-2.3} and it contradicts
Theorem \ref{T-4.5}.

(5)\enspace
The study of this paper was partly motivated by Milan Mosonyi's question to the author, asking whether
there exists,  for $0<\alpha<1$, a `reasonable' $\alpha$-weighted geometric-type mean $\cM(A,B)$
($A,B>0$) other than $\#_\alpha$, where `reasonable' is used in the sense that $\cM$ satisfies
(i)~$\cM(A,B)=A^{1-\alpha}B^\alpha$ if $AB=BA$,
(ii)~tensor multiplicative $\cM(A_1\otimes A_2,B_1\otimes B_2)=\cM(A_1,B_1)\otimes\cM(A_2,B_2)$,
(iii)~block additive $\cM(A_1\oplus A_2,B_1\oplus B_2)=\cM(A_1,B_1)\oplus\cM(A_2,B_2)$,
and (iv)~arithmetic-geometric inequality $\cM(A,B)\le(1-\alpha)A+\alpha B$.
We examined the possibility of the quasi-geometric type means discussed in this paper to satisfy condition
(iv) as it is obvious that they satisfy the other conditions (i)--(iii). But it turned out that none of them other
than $\#_\alpha$ satisfies (iv); see the tables of Sections \ref{Sec-4.2}--\ref{Sec-4.6} and
Corollary \ref{C-4.18} for $G_{\alpha,p}$. Meanwhile, Mosonyi and his coauthors settled the question as
follows in \cite{FMVW}: For any $\alpha\in(0,1)$, if an $\alpha$-weighted matrix mean
$\cM:\bM_n^{++}\times\bM_n^{++}\to\bM_n^{++}$ ($n\in\bN$) satisfies (i),
(ii$'$)~(weakly) tensor multiplicative $\cM(A^{\otimes n},B^{\otimes n})=\cM(A,B)^{\otimes n}$,
(ii$''$)~scalar tensor multiplicative $\cM(aA,bB)=\cM(a,b)\cM(A,B)$ for $a,b\in(0,\infty)$, (iii) and (iv), then
$\cM(A,B)=A\#_\alpha B$ for all $A,B>0$. This result establishes a remarkable new characterization of the
operator mean $\#_\alpha$.

(6)\enspace
In Theorem \ref{T-A.1} of the appendix below, we present the Lie--Trotter--Kato product formula for operator
means in the positive semidefinite matrix case, which we have used in Section \ref{Sec-2}. The author has
known Theorem \ref{T-A.1} for long years, without publication though it was briefly explained in \cite{AuHi2}
without proof. This product formula for operator means in the positive semidefinite case seems unfamiliar
even to matrix analysis experts, while that in the positive definite matrix case is rather well known (see, e.g.,
\cite[Lemma 3.3]{HP}, \cite[Sec.~4.3]{Hi0}). So it would be worthwhile for us to take this opportunity to
present its complete description.

\subsection*{Acknowledgements}

The author thanks Milan Mosonyi for invitation to the workshop at the Erd\H os Center in July, 2024 and for
valuable suggestions which helped to improve this paper.

\appendix

\section{The operator mean version of the Lie--Trotter--Kato product formula for positive semidefinite
matrices}\label{Sec-A}

The famous \emph{Lie--Trotter--Kato product formula} originally established in \cite{Tr,Kato} says that if $H$
and $K$ are lower bounded self-adjoint operators in a Hilbert space $\cH$, then $(e^{-H/m}e^{-K/m})^m$
converges in strong operator topology to $e^{-(H\dot+K)}P_0$ as $m\to\infty$, where $H\dot+K$ is the
\emph{form sum} (see \cite{Kato0}) and $P_0$ is the orthogonal projection onto the closure of the domain
of $H\dot+K$. According to the proof in \cite{Kato0} (see also \cite[Theorem 3.6]{Hi0}) the formula can be
modified in the symmetric form with a continuous parameter as
\begin{align}\label{F-A.1}
\mathrm{s}\mbox{-}\lim_{p\searrow0}(e^{-pH/2}e^{-pK}e^{-pH/2})^{1/p}=e^{-(H\dot+K)}P_0.
\end{align}
Furthermore, it is known \cite[Sec.~5]{Kato} that the formula is valid even if $H$ and $K$ have non-dense
domains.

For Hermitian matrices $H$ and $K$ the product formula simply becomes the Lie formula
\[
\lim_{k\to\infty}(e^{H/m}e^{K/m})^m=\lim_{p\searrow0}(e^{pH/2}e^{pK}e^{pH/2})^{1/p}=e^{H+K},
\]
which has plenty of applications in matrix analysis. The unitary orbital version of this (without limit) is also
worth noting \cite{So,GLT}. For positive semidefinite (not necessarily positive definite) matrices $A,B$ we
consider $H:=-\log A$ and $K:=-\log B$ defined under the restriction to the ranges of the support projections
$A^0:=s(A)$ and $B^0:=s(B)$, respectively. Applying \eqref{F-A.1} (for non-dense domains) to these $H,K$
we have\footnote{
There seems no literature which provides the proof of \eqref{F-A.2} in a way specialized to the matrix setting.}
\begin{align}\label{F-A.2}
\lim_{p\searrow0}(A^{p/2}B^pA^{p/2})^{1/p}
=P_0\exp\{P_0(\log A)P_0+P_0(\log B)P_0\},
\end{align}
where $P_0:=A^0\wedge B^0$.

This appendix is aimed to supply the operator mean version of \eqref{F-A.2} for matrices $A,B\ge0$.
Throughout the appendix let $A,B$ be $n\times n$ positive semidefinite matrices. Define $\log A$ in the
generalized sense as $\log A:=(\log A)A^0$ restricted on the range of $A^0$ (and zero on the range of
$A^{0\perp}=I-A^0$), and similarly $\log B:=(\log B)B^0$. We write $P_0:=A^0\wedge B^0$ as above.
Now, let $\sigma$ be a Kubo--Ando's operator mean with the representing operator monotone function $f$
on $(0,\infty)$, and let $\alpha:=f'(1)$. Note that $0\le\alpha\le1$ and if $\alpha=0$ (resp., $\alpha=1$) then
$A\sigma B=A$ (resp., $A\sigma B=B)$ so that $(A^p\sigma B^p)^{1/p}=A$ (resp.,
$(A^p\sigma B^p)^{1/p}=B$) for all $A,B\ge0$ and $p>0$. So in the rest we assume that $0<\alpha<1$.

\begin{theorem}\label{T-A.1}
With the above assumptions, for every $A,B\ge0$,
\begin{equation}\label{F-A.3}
\lim_{p\searrow0}(A^p\sigma B^p)^{1/p}
=P_0\exp\{(1-\alpha)P_0(\log A)P_0+\alpha P_0(\log B)P_0\}.
\end{equation}
\end{theorem}

\begin{remark}\label{R-A.2}\rm
Note \cite[Sect.~4]{HP} that the RHS of \eqref{F-A.3} is written as
\[
\lim_{\eps\searrow0}\exp\{(1-\alpha)\log(A+\eps I)+\alpha\log(B+\eps I)\},
\]
so that we may write
\begin{align*}
\lim_{p\searrow0}(A^p\sigma B^p)^{1/p}
&=\lim_{\eps\searrow0}\exp\{(1-\alpha)\log(A+\eps I)+\alpha\log(B+\eps I)\} \\
&=\lim_{\eps\searrow0}\lim_{p\searrow0}((A+\eps I)^p\sigma (B+\eps I)^p)^{1/p}.
\end{align*}
\end{remark}

The next lemma is essential to prove the theorem. The proof of the lemma is a slight
modification of that of \cite[Lemma 4.1]{HP}.

\begin{lemma}\label{L-A.3}
For each $p\in(0,p_0)$ with some $p_0>0$, a Hermitian matrix $Z(p)$ is given in the
$2\times2$ block form as
\[
Z(p)=\begin{bmatrix}Z_0(p)&Z_2(p)\\Z_2^*(p)&Z_1(p)\end{bmatrix},
\]
where $Z_0(p)$ is $m\times m$, $Z_1(p)$ is $l\times l$ and $Z_2(p)$ is $m\times l$.
Assume:
\begin{itemize}
\item[(a)] $Z_0(p)\to Z_0$ as $p\searrow0$,
\item[(b)] $\sup\{\|Z_2(p)\|_\infty:p\in(0,p_0)\}<\infty$,
\item[(c)] there is a $\delta>0$ such that $pZ_1(p)\le-\delta I_l$ for all $p\in(0,p_0)$.
\end{itemize}
Then
\[
e^{Z(p)}\to\begin{bmatrix}e^{Z_0}&0\\0&0\end{bmatrix}\quad\mbox{as $p\searrow0$}.
\]
\end{lemma}

\begin{proof}
We list the eigenvalues of $Z(p)$ in decreasing order (with multiplicities) as
\[
\lambda_1(p)\ge\dots\ge\lambda_m(p)\ge\lambda_{m+1}(p)\ge\dots\ge\lambda_{m+l}(p)
\]
together with the corresponding orthonormal eigenvectors
\[
u_1(p),\dots,u_m(p),u_{m+1}(p),\dots,u_{m+l}(p),
\]
so that we write
\begin{equation}\label{F-A.4}
e^{Z(p)}=\sum_{i=1}^{m+l}e^{\lambda_i(p)}|u_i(p)\>\<u_i(p)|.
\end{equation}
Furthermore, let $\mu_1(p)\ge\dots\ge\mu_m(p)$ be the eigenvalues of $Z_0(p)$ and
$\mu_1\ge\dots\ge\mu_m$ be the eigenvalues of $Z_0$ Then $\mu_i(p)\to\mu_i$ as $p\searrow0$
thanks to assumption (a). By the majorization result for eigenvalues in
\cite[Corollary 7.2]{An2} we have
\begin{equation}\label{F-A.5}
\sum_{i=1}^r\mu_i(p)\le\sum_{i=1}^r\lambda_i(p),\qquad1\le r\le m.
\end{equation}
Since
\[
pZ(p)\le\begin{bmatrix}pZ_0(p)&pZ_2(p)\\pZ_2^*(p)&-\delta I_l\end{bmatrix}
\to\begin{bmatrix}0&0\\0&-\delta I_l\end{bmatrix}\quad\mbox{as $p\searrow0$}
\]
thanks to assumptions (a)--(c), it follows that for $m<i\le m+l$, $p\lambda_i(p)<-\delta/2$ for any
$p>0$ sufficiently small so that
\begin{equation}\label{F-A.6}
\lim_{p\searrow0}\lambda_i(p)=-\infty,\qquad m<i\le m+l.
\end{equation}
Hence it suffices to prove that for any sequence $(p_0>)\ p_k\searrow0$ there exist a subsequence
$\{p_k'\}$ of $\{p_k\}$ and vectors $v_1,\dots,v_m\in\bC^m$ such that we have for $1\le i\le m$
\begin{align}
\lambda_i(p_k')&\to\mu_i\quad\mbox{as $k\to\infty$}, \label{F-A.7}\\
u_i(p_k')&\to v_i\oplus0\in\bC^m\oplus\bC^l\quad\mbox{as $k\to\infty$}, \label{F-A.8}\\
Z_0v_i&=\mu_iv_i. \label{F-A.9}
\end{align}
Indeed, it then follows that $v_1,\dots,v_m$ are orthonormal vectors in $\bC^m$, so from
\eqref{F-A.4} and \eqref{F-A.6} we obtain
$$
\lim_{k\to\infty}e^{Z(p_k')}
=\sum_{i=1}^me^{\mu_i}|v_i\>\<v_i|\oplus0=e^{Z_0}\oplus0.
$$

Now, replacing $\{p_k\}$ with a subsequence, we may assume that $u_i(p_k)$ itself converges to some
$u_i\in\bC^m\oplus\bC^l$ for $1\le i\le m$. Writing $u_i(p_k)=v_i^{(k)}\oplus w_i^{(k)}$ in
$\bC^m\oplus\bC^l$, we have
\begin{align}
\lambda_1(p_k)
&=\bigl\<v_i^{(k)}\oplus w_i^{(k)},Z(p_k)(v_i^{(k)}\oplus w_i^{(k)})\bigr\> \nonumber\\
&=\bigl\<v_i^{(k)},Z_0(p_k)v_i^{(k)}\bigr\>
+2\Re\bigl\<v_i^{(k)},Z_2(p_k)w_i^{(k)}\bigr\>
+\bigl\<w_i^{(k)},Z_1(p_k)w_i^{(k)}\bigr\> \nonumber\\
&\le\bigl\<v_i^{(k)},Z_0(p_k)v_i^{(k)}\bigr\>
+2\Re\bigl\<v_i^{(k)},Z_2(p_k)w_i^{(k)}\bigr\>
-{\delta\over p_k}\,\big\|w_i^{(k)}\big\|^2 \label{F-A.10}
\end{align}
by assumption (c). For $i=1$, since $\mu_1(p_k)\le\lambda_1(p_k)$ by \eqref{F-A.5} for $r=1$, it follows
from \eqref{F-A.10} that
\[
p_k\mu_1(p_k)\le p_k\|Z_0(p_k)\|_\infty+2p_k\|Z_2(p_k)\|_\infty-\delta\big\|w_1^{(k)}\big\|^2
\]
so that
\[
\delta\big\|w_1^{(k)}\big\|^2\le p_k\|Z_0(p_k)\|_\infty+2p_k\|Z_2(p_k)\|_\infty-p_k\mu_1(p_k)\to0
\]
as $k\to\infty$ ($p_k\searrow0$) due to assumptions (a) and (b). Hence we have $w_1^{(k)}\to0$ so that
$u_1(p_k)\to u_1=v_1\oplus0$ in $\bC^m\oplus\bC^l$ (hence $v_1^{(k)}\to v_1$) for some $v_1\in\bC^m$.
From \eqref{F-A.10} again we furthermore have
\begin{align*}
\limsup_{k\to\infty}\lambda_1(p_k)
&\le\limsup_{k\to\infty}\Bigl\{\bigl\<v_i^{(k)},Z_0(p_k)v_i^{(k)}\bigr\>
+2\|Z_2(p_k)\|_\infty\big\|w_i^{(k)}\big\|\Bigr\} \\
&\le\<v_1,Z_0v_1\>\le\mu_1=\lim_{k\to\infty}\mu_1(p_k)\le\liminf_{k\to\infty}\lambda_1(p_k).
\end{align*}
Therefore, $\lambda_1(p_k)\to\<v_1,Z_0v_1\>=\mu_1$ and hence $Z_0v_1=\mu_1v_1$ since $\mu_1$ is
the largest eigenvalue of $Z_0$. Next, when $k\ge2$ and $i=2$, since $\lambda_2(p_k)$ is bounded below
by \eqref{F-A.5} for $r=2$, it follows as above that $w_2^{(k)}\to0$ and hence $u_2(p_k)\to u_2=v_2\oplus0$
for some $v_2\in\bC^m$. Therefore,
\[
\limsup_{k\to\infty}\lambda_2(p_k)\le\<v_2,Z_0v_2\>
\le\mu_2\le\liminf_{k\to\infty}\lambda_2(p_k)
\]
so that $\lambda_2(p_k)\to\<v_2,Z_0v_2\>=\mu_2$ and $Z_0v_2=\mu_2v_2$ since $\mu_2$ is the largest
eigenvalue of $Z_0$ restricted to $\{v_1\}^\perp\cap\bC^m$. Repeating this argument we obtain
$v_1,\dots,v_m\in\bC^m$ for which \eqref{F-A.7}--\eqref{F-A.9} hold for $1\le i\le m$.
\end{proof}

%Note that the lemma and its proof hold true even when the assumption $Z_2(p)\to Z_2$ in (b) is relaxed
%into $pZ_2(p)\to0$ as $p\searrow0$.

\begin{proof}[Proof of Theorem \ref{T-A.1}]
Let us divide the proof into two steps. In the proof below we use the $\alpha$-weighted arithmetic mean
$\triangledown_\alpha$ and the $\alpha$-weighted harmonic mean $!_\alpha$. Note that
\[
A!_\alpha B\le A\sigma B\le A\triangledown_\alpha B,\qquad A,B\ge0.
\]

\noindent
{\it Step 1.}\enspace
First, we prove the theorem in the case where $P\sigma Q=P\wedge Q$ for all orthogonal projections
$P,Q$ (this is the case, for instance, when $\sigma=\ !_\alpha$ or $\#_\alpha$ the weighted geometric mean,
see \cite[Theorem 3.7]{KA}). Let $\cH_0$ be the range of $P_0$ ($=A^0!_\alpha B^0=A^0\sigma B^0$).
From the operator monotonicity of $\log x$ ($x>0$) it follows that for every $p>0$,
\begin{equation}\label{F-A.11}
{1\over p}\log(A^p!_\alpha B^p)\big|_{\cH_0}\le{1\over p}\log(A^p\sigma B^p)\big|_{\cH_0}
\le{1\over p}\log\bigl(P_0(A^p\triangledown_\alpha B^p)P_0\bigr)\big|_{\cH_0}.
\end{equation}
For every $\eps>0$ we have
\begin{align*}
(A+\eps A^{0\perp})^p!_\alpha(B+\eps B^{0\perp})^p
&=\bigl\{(A+\eps A^{0\perp})^{-p}\triangledown_\alpha(B+\eps B^{0\perp})^{-p}\bigr\}^{-1} \\
&=\bigl\{A^{-p}\triangledown_\alpha B^{-p}
+\eps^{-p}(A^{0\perp}\triangledown_\alpha B^{0\perp})\bigr\}^{-1},
\end{align*}
where $A^{-p}=(A^{-1})^p$ and $B^{-p}=(B^{-1})^p$ are defined via the generalized inverses. Therefore,
\begin{equation}\label{F-A.12}
\begin{aligned}
P_0\bigl\{(A+\eps A^{0\perp})^p!_\alpha(B+\eps B^{0\perp})^p\bigr\}P_0
&\ge\bigl\{P_0\bigl(A^{-p}\triangledown_\alpha B^{-p}
+\eps^{-p}(A^{0\perp}\triangledown_\alpha B^{0\perp})\bigr)P_0\bigr\}^{-1} \\
&=\bigl\{P_0(A^{-p}\triangledown_\alpha B^{-p})P_0\bigr\}^{-1},
\end{aligned}
\end{equation}
since the support projection of $A^{0\perp}\triangledown_\alpha B^{0\perp}$ is
$A^{0\perp}\vee B^{0\perp}=P_0^\perp$. In the above, $\{\cdots\}^{-1}$ is the generalized inverse (with
support $\cH_0$) and the inequality is due to \cite[Corollary 2.3]{Choi}. Letting $\eps\searrow0$ in
\eqref{F-A.12} gives
\begin{align}\label{F-A.13}
A^p!_\alpha B^p=P_0(A^p!_\alpha B^p)P_0
\ge\bigl\{P_0(A^{-p}\triangledown_\alpha B^{-p})P_0\bigr\}^{-1}
\end{align}
so that
\begin{equation}\label{F-A.14}
{1\over p}\log(A^p!_\alpha B^p)\big|_{\cH_0}
\ge-{1\over p}\log\bigl(P_0(A^{-p}\triangledown_\alpha B^{-p})P_0\bigr)\big|_{\cH_0}.
\end{equation}
Combining \eqref{F-A.11} and \eqref{F-A.14} yields
\begin{equation}\label{F-A.15}
-{1\over p}\log\bigl(P_0(A^{-p}\triangledown_\alpha B^{-p})P_0\bigr)\big|_{\cH_0}
\le{1\over p}\log(A^p\sigma B^p)\big|_{\cH_0}
\le{1\over p}\log\bigl(P_0(A^p\triangledown_\alpha B^p)P_0\bigr)\big|_{\cH_0}.
\end{equation}

Since $A^{-p}=e^{-p\log A}A^0=A^0-p\log A+o(p)$ and similarly $B^{-p}=B^0-p\log B+o(p)$ as
$p\searrow0$, we have
\[
A^{-p}\triangledown_\alpha B^{-p}
=A^0\triangledown_\alpha B^0-p((\log A)\triangledown_\alpha(\log B))+o(p)
\]
so that
\begin{align}\label{F-A.16}
P_0(A^{-p}\triangledown_\alpha B^{-p})P_0
=P_0-p\{(1-\alpha)P_0(\log A)P_0+\alpha P_0(\log B)P_0\}+o(p).
\end{align}
Therefore,
\begin{equation}\label{F-A.17}
-{1\over p}\log\bigl(P_0(A^{-p}\triangledown B^{-p})P_0\bigr)\big|_{\cH_0}
=\{(1-\alpha)P_0(\log A)P_0+\alpha P_0(\log B)P_0\}\big|_{\cH_0}+o(1).
\end{equation}
Similarly, we have
\begin{align}\label{F-A.18}
P_0(A^p\triangledown_\alpha B^p)P_0
=P_0+p\{(1-\alpha)P_0(\log A)P_0+\alpha P_0(\log B)P_0\}+o(p).
\end{align}
so that
\begin{equation}\label{F-A.19}
{1\over p}\log\bigl(P_0(A^p\triangledown B^p)P_0\bigr)\big|_{\cH_0}
=\{(1-\alpha)P_0(\log A)P_0+\alpha P_0(\log B)P_0\}\big|_{\cH_0}+o(1).
\end{equation}
From \eqref{F-A.15}, \ref{F-A.17} and \eqref{F-A.19} we obtain
\[
\lim_{p\searrow0}{1\over p}\log(A^p\sigma B^p)\big|_{\cH_0}
=\{(1-\alpha)P_0(\log A)P_0+\alpha P_0(\log B)P_0\}\big|_{\cH_0},
\]
which yields the required limit formula.

\medskip\noindent
{\it Step 2.}\enspace
For a general operator mean $\sigma$ the integral representation theorem \cite[Theorems 3.4, 3.7]{KA}
says that there are $\theta,\beta\in[0,1]$ and an operator mean $\tau$ such that
\[
\sigma=\theta\triangledown_\beta+(1-\theta)\tau
\]
and $P\tau Q=P\wedge Q$ for all orthogonal projections $P,Q$. Moreover, $\tau$ has the representing
operator monotone function $g$ on $(0,\infty)$ for which $\gamma:=g'(1)\in(0,1)$ and
\begin{align}\label{F-A.20}
\alpha=\theta\beta+(1-\theta)\gamma.
\end{align}
We may assume that $0<\theta\le1$ since the case $\theta=0$ was shown in Step 1. Moreover, when
$\theta=1$, we have $\beta=\alpha\in(0,1)$. At the moment, assume that $0<\theta\le1$ and $0<\beta<1$.
Let $A,B\ge0$ be given, and note that
$A^0\sigma B^0=\theta A^0\triangledown_\beta B^0+(1-\theta)(A^0\wedge B^0)$ has the support projection
$A^0\vee B^0$. Let $\cH$, $\cH_0$ and $\cH_1$ denote the ranges of $A^0\vee B^0$,
$P_0=A^0\wedge B^0$ and $A^0\vee B^0-P_0$, respectively, so that $\cH=\cH_0\oplus\cH_1$. Note that
the support of $A^p\sigma B^p$ for any $p>0$ is $\cH$. We will describe
${1\over p}\log(A^p\sigma B^p)\big|_\cH$ in the $2\times2$ block form with respect to the decomposition
$\cH=\cH_0\oplus\cH_1$. Let
\[
Y_0:=\{(1-\gamma)P_0(\log A)P_0+\gamma P_0(\log B)P_0\}\big|_{\cH_0}.
\]
It follows from Step 1 that $\lim_{p\searrow0}(A^p\tau B^p)^{1/p}=P_0e^{Y_0}P_0$ and hence
\begin{align*}
A^p\tau B^p&=P_0\bigl(e^{Y_0}+o(1)\bigr)^pP_0
=P_0\bigl[\exp\bigl\{p\log\bigl(e^{Y_0}+o(1)\bigr)\bigr\}\bigr]P_0 \\
&=P_0\bigl[I_{\cH_0}+p\log\bigl(e^{Y_0}+o(1)\bigr)+o(p)\bigr]P_0
=P_0\bigl(I_{\cH_0}+pY_0+o(p)\bigr)P_0 \\
&=P_0+pP_0Y_0P_0+o(p)\quad\mbox{as $p\searrow0$}.
\end{align*}
In the above, the fourth equality follows since $\log\bigl(e^{Y_0}+o(1)\bigr)=Y_0+o(1)$. On the other hand,
we have
\[
A^p\triangledown_\beta B^p
=A^0\triangledown_\beta B^0+p\{(\log A)\triangledown_\beta(\log B)\}+o(p).
\]
Therefore, we have
\begin{align*}
A^p\sigma B^p&=\theta(A^0\triangledown_\beta B^0)+(1-\theta)P_0
+p\theta\{(\log A)\triangledown_\beta(\log B)\}+p(1-\theta)P_0Y_0P_0+o(p)\quad\mbox{as $p\searrow0$}.
\end{align*}
Setting
\begin{align*}
C&:=\bigl\{\theta(A^0\triangledown_\beta B^0)+(1-\theta)P_0\bigr\}\big|_\cH, \\
H&:=\bigl[\theta\{(\log A)\triangledown_\beta(\log B)\}+(1-\theta)P_0Y_0P_0\bigr]\big|_\cH,
\end{align*}
we write
\begin{equation}\label{F-A.21}
{1\over p}\log(A^p\sigma B^p)\big|_\cH
={1\over p}\log(C+pH+o(p)),
\end{equation}
where $C$ is a positive definite contraction on $\cH$ and $H$ is a Hermitian operator on $\cH$. Note that
the eigenspace of $C$ for the eigenvalue $1$ is $\cH_0$. Hence, with a basis consisting of orthonormal
eigenvectors for $C$ we may assume that $C$ is diagonal so that $C=\diag(c_1,\dots,c_{m+l})$ with
\[
c_1=\dots=c_m=1>c_{m+1}\ge\dots\ge c_{m+l}>0,
\]
where $m=\dim\cH_0$ and $m+l=\dim\cH$.

Applying Taylor's theorem (see, e.g., \cite[Theorem 2.3.1]{Hi}) to $\log(C+pH+o(p))$ we have
\begin{equation}\label{F-A.22}
\log(C+pH+o(p))=\log C+pD(\log x)(C)(H)+o(p),
\end{equation}
where $D(\log x)(C)(\cdot)$ denotes the Fr\'echet derivative of the functional calculus by $\log x$ at $C$.
Daleckii and Krein's derivative formula (see, e.g., \cite[Theorem 2.3.1]{Hi}) says that
\begin{equation}\label{F-A.23}
D(\log x)(C)(H)=\Biggl[{\log c_i-\log c_j\over c_i-c_j}\Biggr]_{i,j=1}^{m+l}\circ H,
\end{equation}
where $\circ$ is the Schur product and $(\log c_i-\log c_j)/(c_i-c_j)$ is understood as $1/c_i$ when $c_i=c_j$.
We write $D(\log x)(C)(H)$ in the $2\times2$ block form on $\cH_0\oplus\cH_1$ as
$\begin{bmatrix}Z_0&Z_2\\Z_2^*&Z_1\end{bmatrix}$ where $Z_0:=P_0HP_0|_{\cH_0}$. By
\eqref{F-A.21}--\eqref{F-A.23} we can write
\begin{align}\label{F-A.24}
{1\over p}\log(A^p\sigma B^p)\big|_\cH={1\over p}\log C+D(\log x)(C)(H)+o(1)
=\begin{bmatrix}Z_0(p)&Z_2(p)\\Z_2^*(p)&Z_1(p)\end{bmatrix},
\end{align}
where
\begin{align*}
Z_0(p)&=Z_0+o(1),\qquad Z_2(p)=Z_2+o(1), \\
Z_1(p)&={1\over p}\diag(\log c_{m+1},\dots,\log c_{m+l})+Z_1+o(1).
\end{align*}
This $2\times2$ block form of $Z(p):={1\over p}\log(A^p\sigma B^p)\big|_\cH$ satisfies assumptions (a)--(c)
of Lemma \ref{L-A.3} for $p\in(0,p_0)$ with a sufficiently small $p_0>0$. Therefore, the lemma implies that
\[
\lim_{p\searrow0}(A^p\sigma B^p)^{1/p}\big|_\cH=\lim_{p\searrow0}e^{Z(p)}=e^{Z_0}\oplus0
\]
on $\cH=\cH_0\oplus\cH_1$. Finally, we have
\begin{align*}
Z_0&=\theta(\{(1-\beta)P_0(\log A)P_0+\beta P_0(\log B)P_0\}|_{\cH_0} \\
&\qquad+(1-\theta)\{(1-\gamma)P_0(\log A)P_0+\gamma P_0(\log B)P_0\}|_{\cH_0} \\
&=\{(1-\alpha)P_0(\log A)P_0+\alpha P_0(\log B)P_0\}|_{\cH_0}
\end{align*}
thanks to \eqref{F-A.20}. Hence the desired limit formula follows.

For the remaining case where $0<\theta<1$ and $\beta=0$ or $1$, the proof is similar to the above when
we take as $\cH$ the range of $A^0$ (for $\beta=0$) or $B^0$ (for $\beta=1$) instead of the range of
$A^0\vee B^0$.
\end{proof}

\begin{remark}\label{R-A.4}\rm
Assume that the operator mean $\sigma$ satisfies $P\sigma Q=P\wedge Q$ for all orthogonal projections
$P,Q$, that is, the representing function $f$ satisfies $f(0^+)=0$ and $\lim_{x\to\infty}f(x)/x=0$ (see
\cite[Theorem 3.7]{KA}). This is the case when $\sigma\#_\alpha$ for instance. Then, from the proof of
Step 1 above, for any $A,B\ge0$ we have a slightly improved form of \eqref{F-A.3} as follows:
\begin{align}\label{F-A.25}
A^p\sigma B^p=P_0+p\{(1-\alpha)P_0(\log A)P_0+\alpha P_0(\log B)P_0\}+o(p)
\quad\mbox{as $p\searrow0$}.
\end{align}
Indeed, set $L:=(1-\alpha)P_0(\log A)P_0+\alpha P_0(\log B)P_0$. By \eqref{F-A.13} and \eqref{F-A.16}
one has
\begin{align*}
A^p\sigma B^p&\ge A^p!_\alpha B^p\ge\{P_0(A^{-p}\triangledown_\alpha B^{-p})P_0\}^{-1} \\
&=\{P_0-pL+o(p)\}^{-1}=P_0+pL+o(p)\quad\mbox{as $p\searrow0$}.
\end{align*}
Also, by \eqref{F-A.18} one has
\[
A^p\sigma B^p=P_0(A^p\sigma B^p)P_0\le P_0(A^p\triangledown_\alpha B^p)P_0=P_0+pL+o(p).
\]
Therefore, \eqref{F-A.24} follows.
\end{remark}

\end{document}